\documentclass[12pt,twosided]{report}\usepackage{LUDisStyle}
\usepackage[hang,small,bf]{caption}
\usepackage{LUDisStyle}
\usepackage{graphicx}
\usepackage{subfigure}
\usepackage{amsmath}
\usepackage{amssymb}
\usepackage{amsthm}
\usepackage{verbatim}
\usepackage{xcolor}
\usepackage{bbm}

\newcommand{\norm}[1]{\left\lVert#1\right\rVert}

\newcommand{\be}{\mathbf{E}}
\newcommand{\bp}{\mathbf{P}}

\newcommand{\iot}{\int_{0}^{t}}

\newcommand{\var}{\textbf{Var}}

\newcommand{\lp}{\left(}
\newcommand{\rp}{\right)}
\newcommand{\lc}{\left[}
\newcommand{\rc}{\right]}

\newcommand{\lla}{\left\langle}
\newcommand{\rra}{\right\rangle}

\newcommand{\ep}{\varepsilon}

\newcommand{\si}{\sigma}

\newcommand{\vp}{\varphi}
\def \eref#1{\hbox{(\ref{#1})}}

\newcommand{\beq}{\begin{equation}}
\newcommand{\eeq}{\end{equation}}
\newcommand{\bea}{\begin{eqnarray}}
\newcommand{\eea}{\end{eqnarray}}
\newcommand{\beas}{\begin{eqnarray*}}
\newcommand{\eeas}{\end{eqnarray*}}

\newcommand{\cf}{{\mathcal F}}

\def\cH{\mathcal{H}}
\def\cF{\mathcal{F}}

\def\bR{\mathbb{R}}

\newcommand{\R}{{\mathbb R}}

\newcommand{\ind}{\boldsymbol{1}}

\newenvironment{proof*}{\noindent\textbf{Proof} \vspace*{0mm} \\}{\hfill $\blacksquare$ \vspace*{5mm}}

\makeatletter
\renewcommand \theequation {%
\ifnum \c@chapter>\z@ \@arabic\c@chapter.%
\fi \ifnum\c@subsection>\z@\@arabic\c@subsection.%
\fi\@arabic\c@equation} 

\newtheoremstyle{break}
  {\topsep}{\topsep}%
  {\itshape}{}%
  {\bfseries}{}%
  {\newline}{}%
\theoremstyle{break}
\newtheorem{theorem}{Theorem}[section]
\newtheorem{assumption}[theorem]{Assumption}

\newtheorem{definition}[theorem]{Definition}

\newtheorem{lemma}[theorem]{Lemma}

\newtheorem{proposition}[theorem]{Proposition}
\theoremstyle{remark}
\newtheorem{remark}[theorem]{Remark}
\newtheorem{example}[theorem]{Example}

\newcommand*{\rom}[1]{\expandafter\@slowromancap\romannumeral #1@}

\def\no{\nonumber}

    \gdef\upqed{\vskip-\belowdisplayskip\vskip-\baselineskip\kern0pt\let\@qed\mathend@qed}

\begin{document}
\pagestyle{plain}
\pagenumbering{roman}
\dissertationtrue
\thispagestyle{empty}%

	\vskip0.5in%

	\begin{center}

		\LARGE Intermittency for the stochastic heat and wave equations with generalized fractional noise

	\end{center}

	\vfill

	\begin{center}

		\rm by \\
		
                         \vfill
                 
		Ruxiao Qian \\

	\end{center}

	\vfill

	\begin{center}

		\rm A Dissertation \\

		Presented to the Graduate Committee \\

		of Lehigh University \\

		in Candidacy for the Degree of \\

		Doctor of Philosophy \\

		in \\

		Mathematics

	\end{center}

	\vfill

	\begin{center}

		Lehigh University \\

		August 2022 \\

	\end{center}\vskip.5in
	
	\newpage

\vspace*{\fill}
\begin{center}
Copyright\\
Ruxiao Qian
\end{center}
\newpage

\thispagestyle{plain}

Approved and recommended for acceptance as a dissertation in partial fulfillment of the requirements for the degree of Doctor of Philosophy.
\\
\\
Ruxiao Qian \\
Intermittency for the stochastic eeat and wave Equations with generalized fractional noise

\vspace{.1in}

\begin{tabular}{l}
\\
\hline
\textbf{Date \ \ \ \ \ \ \ \ \ \ \ \ \ \ \ \ \ \ \ }
\end{tabular}

\begin{flushright}
\begin{tabular}{l}
\\
\hline
\textbf{Daniel Conus, PhD}, Dissertation Director, Chair \\ 
\end{tabular}

\vspace{.05in}

\end{flushright}

\begin{tabular}{l}
\\
\hline
\textbf{Accepted Date \ \ \ \ \ \ \ }
\end{tabular}

\begin{flushright}
\begin{tabular}{l}
Committee Members \ \ \ \ \ \ \ \ \ \ \ \ \ \ \ \ \ \ \ \ \ \ \ \ \ \
\\ 
\\
\\
\hline
\textbf{Si Tang, PhD}\\
\\ 
\\
\\
\hline
\textbf{Wei-Min Huang, PhD}\\
\\ 
\\
\\
\hline
\textbf{Xiaoming Song, PhD}\\
\\
\\
\\
\end{tabular}
\end{flushright}
\newpage
\chapter*{Acknowledgement}
First and foremost, I would like to show my deepest gratitude to my advisor, Prof. Daniel Conus. His enthusiasm, extremely organized lecture and intuitive explanation inspired me to join math department and to enjoy the beauty of math. This work could never have been done without his significant inspiration, constant guidance and support. I will always hold this contagious passion that originally drew me in.

Secondly, I would like to express thanks to Prof. Don Davis. His strict and meticulous manner, and his Ultramarathoning career motivated me to work harder and to be a better mathematician.

I am thankful to Prof. Si Tang, Prof. Xiaoming Song, and Prof. Wei-Min Huang for providing me with detailed and constructive comments to improve this thesis.

Further, I am thankful to Prof. Lei Wu for his generous helping in many aspects of math and most importantly, a profound view of Analysis and PDEs.

Last but not least, I owe my loving thanks to my parents. Their encouragement and caring are always the strongest support in my life.
\tableofcontents
\nopagebreak
\addcontentsline{toc}{chapter}{List of Figures}
\listoffigures
\newpage
\pagestyle{plain}
\pagenumbering{arabic}
\addcontentsline{toc}{chapter}{Abstract}
\chapter*{Abstract}
\title{Abstract}
We are looking at the stochastic heat and wave equations with different types of fractional noise. We are interested in the intermittency property and Lyapunov exponent for the solution. First we look at the equation driven by the Dobri\'c-Ojeda noise and we show that the Lyapunov exponent matches that of the equation driven by standard fractional noise as obtained by Hu, Huang, Nualart, Tindel (2015) and Balan, Conus (2016).

In the second part, we introduce a generalized fractional noise that includes both standard fractional noise and Dobri\'c-Ojeda noise. It shows that, in this specific situation, the correlation structure of the noise does not change the Lyapunov exponent. We conjecture that this result would hold more generally.
\pagestyle{plain}

\chapter{Introduction}
\section{Heat and Wave equations}
In mathematics and physics, the heat equation is a certain second-order partial differential equation. The theory of the heat equation was first developed by Joseph Fourier in 1822 for the purpose of modeling how a quantity, such as heat, diffuses through a given region. The classic heat equation (without random noise) is given by
\begin{align}
\left\{
\begin{array}{l}
\displaystyle
\label{eqn:classic heat}
    \frac{\partial}{\partial t} u(t,x)=\frac{1}{2}\Delta u(t,x),\\\rule{0ex}{1.5em}
    u(0,x)=u_0(x),
    \end{array}
\right.
\end{align}
where $t>0, x\in \mathbb{R}^d$, $\Delta$ stands for the Laplacian operator on $\mathbb{R}^d$, $u_0$ is a bounded function. 

The solution to the heat equation is given by
\begin{equation} \label{heatsolution}
u(t,x)=(p_h*u_0)(t,x)=\int_{\bR^d} u_0(y) \cdot p_{h}(t,x-y)dy
\end{equation}
where $*$ denotes the convolution, and $p_h$ is the heat kernel defined below.
\begin{definition}[Heat kernel]
The fundamental solution of the heat equation (also known as the heat kernel) in dimension $d$ is defined by
\begin{align}
\label{heatkernel}
p_h(t,x) = \frac{1}{(2\pi t)^{d/2}} \exp\left(-\frac{|x|^2}{2t}\right),
\end{align}
for all $t > 0$, $x \in \mathbb{R}^d$.
\end{definition}
The Fourier transform of the heat kernel is given by:
\begin{align}
    \hat{p}_h(t,\xi)=\int_{\bR^d}e^{-i\xi x}p_h(t,x) dx=\exp\left(-\frac{t|\xi|^2}{2}\right). \label{eqn:heatFourier}
\end{align}
We notice that for all $d$, $$\displaystyle{\int_{\mathbb{R}^d}p_h(t,x)\,dx =1}. $$

Finally, we point out a connection between the heat equation and the theory of stochastic processes. The operator $\frac{1}{2}\Delta$ is the infinitesimal generator of Brownian motion.

Indeed, let $B(t)$ be an $\mathbb{R}^d$-valued Brownian motion. Let $t \geq 0$ be fixed and consider a bounded function $g$. The infinitesimal generator $\mathcal{A}g$ is defined to be the limit
\begin{align}
    (\mathcal{A}g)(x)=\lim_{\epsilon \downarrow 0}\frac{E^x[g(B_\epsilon)]-g(x)}{\epsilon},
\end{align}
where $E^x$ denotes the expectation under the condition $B_0 = x$. It is a known fact that for Brownian motion, 
\begin{align}
    (\mathcal{A}g) = \frac{1}{2}\Delta g.
\end{align}

The wave equation is a second-order linear partial differential equation for the description of waves or standing wave fields — as they occur in classical physics — such as mechanical waves (e.g. water waves, sound waves and seismic waves) or electromagnetic waves (including light waves). It arises in fields like acoustics, electromagnetism, and fluid dynamics. A single wave propagating in a pre-defined direction can also be described with the one-way wave equation.

The wave equation is given by
\begin{align}
\label{eqn:classic wave}
\left\{
\begin{array}{l}
\displaystyle
    \frac{\partial^2}{\partial t^2} u(t,x)=\Delta u(t,x),\\\rule{0ex}{1.5em}
    u(0,x)=u_0(x),\\
    \frac{\partial}{\partial t}u(0,x)=v_0(x),
    \end{array}
\right.
\end{align}
where $t>0, x\in \mathbb{R}^d$, $\Delta$ stands for the Laplacian operator on $\mathbb{R}^d$, $u_0$ and $v_0$ are bounded functions.

The solution to the wave equation is given by
\begin{equation} \label{wavesolution}
u(t,x)= \frac{d}{dt} (p_w*u_0)(t,x) + (p_w*v_0)(t,x)
\end{equation}
where $*$ denotes the convolution, and $p_w$ is the wave kernel defined below. In dimension $1$, this can be written as
\begin{equation*}
u(t,x) = \frac{u_0(x+t) + u_0(x-t)}{2} + \int_{x-t}^{x+t} v_0(y) \cdot dy.
\end{equation*}
 
\begin{definition}[Wave kernel]
The fundamental solution of the wave equation (or wave kernel) in dimension $d$ is defined by
\begin{align}
\label{wave kernel}
p_w(t,x)=\left\{ \begin{array}{ll}
\frac{1}{2}\mathbbm{1}_{\{|x|<t\}}     & \text{when } d=1, \\\rule{0ex}{1.5em}
\frac{1}{2\pi}\frac{1}{\sqrt{t^2-|x|^2}} \mathbbm{1}_{\{|x|<t\}}    & \text{when }d=2,\\\rule{0ex}{1.5em}
\frac{1}{4\pi t}\sigma_t& \text{when }d=3.
\end{array}\right.
\end{align}
Here $\sigma_t$ is the surface measure on $\partial \mathbf{B}(0,t)$. We notice that for all $d \leq 3$,
\begin{align}
\label{waveforalldim}
\displaystyle{\int_{\mathbb{R}^d}p_w(t,x)\,dx=t}.
\end{align}
For $d \geq 4$, $p_w$ is not a function, but a Schwartz distribution, it will not satisfy one of our assumption, hence the case will not be covered in this dissertation. The Fourier transform of the wave kernel is given by
\begin{align}
\label{e:G-Fourier}
    \hat{p}_w(t,\xi)=\int_{\bR^d} e^{-i\xi x}p_w(t,x)dx=\frac{\sin(t|\xi|)}{|\xi|}.
\end{align}
\end{definition}

\section{Stochastic heat and wave equations}
In this dissertation, we are interested in the model of a random walk in a random potential, often also called the parabolic Anderson model (PAM) or stochastic heat equation (SHE). Here the walk has a strong tendency to be confined to an extremely preferable part of the random medium. Therefore, the global properties of the system is not determined by an average behavior, but by some local extreme behavior. As a consequence, the rigorous work on the PAM owes much to theory of large deviations, which describes the exponential rates of probabilities of rare events. 

The PAM has become a popular model to study among probabilists and mathematical physicists, since there are a lot of connections to other interesting topics like branching random walks with random branching rates, the spectrum of random Schr\"odinger operators, and certain variational problems. The mathematical activity on the PAM is on a high level since the 1990s, and many specific and deeper questions and variants have been studied, especially in the last few years. For instance, models with time-dependent potentials, or connections with Anderson localisation. The Stochastic Heat Equation (SHE) that we will consider here is a continuous-time, continuous-space version of PAM. It is defined as
\begin{align}
\left\{
\begin{array}{l}
\displaystyle
\label{eqn:DO process}
    \frac{\partial}{\partial t} u(t,x)=\frac{1}{2}\Delta u(t,x)+ u(t,x)\dot{W}(t,x),\\\rule{0ex}{1.5em}
    u(0,x)=u_0,
    \end{array}
\right.
\end{align}
where $t>0, x\in \mathbb{R}^d$, $\Delta$ stands for the Laplacian operator on $\mathbb{R}^d$, $u_0$ is a bounded function, and $\dot{W}$ is a random noise that we will introduce below.

There are several motivations for us to investigate the classic stochastic partial differential equations one might wish to solve, either by extending Ito’s theory \cite{Dalang1999} and \cite{peszat1997} or by Malliavin calculus,see \cite{malliavincalc}:
 \begin{itemize}
\item For homogenization problems of PDEs driven by highly oscillating stationary random fields (see Iftimie et al.~\cite{Iftimie2008}), we often introduce the stochastic heat equation. Here we notice that limit theorems are often obtained through a Feynman-Kac representation of the solution to the heat equation.
\item There is a connection between equation \eqref{eqn:DO process} and the partition function of directed and undirected continuum polymers, exploited in Rovira and Tindel~\cite{tindel2005}. The basic properties of an equation of type \eqref{eqn:DO process} are translated into corresponding properties of the polymers.
\item The multiplicative stochastic heat equation exhibits concentration properties of its energy. This phenomenon is referred as intermittency for the process $u$ solution to (1.1), see Conus, Joseph, and Khoshnevisan~\cite{conuskhoshnevisan2013}, and as a localization property for the polymer measure (Carmona and Hu~\cite{carmonahu2006}). The intermittency property of our model is one of the main results of the current dissertation, and will be developed in Section 1.3.
\item The stochastic heat equation is related to the Kardar-Parisi-Zhang (KPZ) equation~\cite{KPZeqn} 
 \begin{align}
     \frac{\partial h}{\partial t}=\Delta h- |\nabla h|^2+\lambda \dot{W}(t,x),
 \end{align}
 which describes the growth of interfaces, such as the clustering of bacteria or the movement of galaxies, by the Hopf-Cole transformation. The Hopf-Cole transformation is a type of logarithmic transformation that connects the stochastic heat equation to the KPZ equation. This connection was rigorously proved by Martin Hairer (see \cite{hairer2013}).
 \end{itemize}

The stochastic wave equation is one of the fundamental stochastic partial differential
equations of hyperbolic type. The behavior of its solutions is significantly different from the stochastic heat equation, though they share some similarities which allow them at times to be treated with similar tools. The equation is defined as
\begin{align}
\label{eqnwave:DO process}
\left\{
\begin{array}{l}
\displaystyle
    \frac{\partial^2}{\partial t^2} u(t,x)=\Delta u(t,x)+  u(t,x)\dot{W}(t,x),\\\rule{0ex}{1.5em}
    u(0,x)=u_0,\\
    \frac{\partial}{\partial t}u(0,x)=v_0,
    \end{array}
\right.
\end{align}
where $t>0, x\in \mathbb{R}^d$, $\Delta$ stands for the Laplacian operator on $\mathbb{R}^d$, $u_0$ and $v_0$ are bounded functions.

 For the stochastic wave equation, we have several applications, see \cite{utahsummer2006}. The first example is the motion of a strand of DNA. A DNA molecule can be viewed as a long elastic string. Since a DNA molecule typically “floats” in a fluid, it is constantly moving. Thus we can describe its position by $u(t,x)$, just as a particle of pollen floating in a fluid moves according to Brownian motion, then by Newton's law of motion, we have
 \begin{align}
      \frac{\partial^2}{\partial t^2} u(t,x)=\Delta u(t,x)+  \int_0^1 k(x,y)u(t,y)\,dy+\dot{F}(t,x),
 \end{align}
where the first term on the right-hand side represents the elastic forces, the second term is a (non-local) friction term, and the third term $\dot{F}(t,y)$ is a Gaussian noise, with spatial correlation $k$, that is,
$E(\dot{F}(t,x)\dot{F}(s,y))=\delta_0(t-s)k(x,y)$. The function $k$ is the same in the friction term and in the correlation.

A second example is the internal structure of the sun, an international project named Project SOHO (Solar and Heliospheric Observatory)~\cite{utahsummer2006}, whose main objective was to obtain information about the internal structure of the sun using the motion of the sun's surface. They used a model given by
\begin{align}
    \frac{\partial^2 u}{\partial t^2} (t,x)=c^2(x)\rho_0(x)\Bigg{(}\nabla \cdot (\frac{1}{\rho_0(x)}\nabla u)+\nabla \cdot F(t,x)\Bigg{)},
\end{align}
where $x\in B(0,R)$, the ball of radius $R$ centered at the origin, $c^2(x)$ is the speed of wave propagation at position $x$, $\rho_0(x)$ is the density at position $x$ and the vector $F(t,x)$ models the shock that originates at time $t$ from position $x$.

In this model, $F$ is the 3-dimensional Gaussian noise concentrated on the sphere $\partial B(0 , r)$, where $0 < r < R$, see \cite{Dalang2004}.

 The most common situation studied in the literature for these two equations is when the driving noise is a centered Gaussian random field $W_0$ given (informally) by 
 \begin{equation} \label{eqn:WN}
E[\dot{W}_0(t,x)\dot{W}_0(s,y)] = \delta_0(t-s) f(x-y),
\end{equation}
where $\delta_0$ is the Dirac delta function and $f: \bR^d \rightarrow \R$ is a symmetric, positive-definite function known as the \emph{(space) covariance function}. For such a function, Bochner's theorem guarantees that there exists a positive measure $\mu$ such that $f = \hat \mu$. The measure $\mu$ is known as the \emph{spectral measure}.

Existence and uniqueness of a solution to both the heat and wave equation driven by $W_0$ is studied by Dalang in the seminal paper \cite{Dalang1999}. The solution is understood in integral (mild) form as a random field $\{u(t,x)\}_{t > 0, x \in \R^d}$. For the heat equation, we have
\begin{equation}
\label{whitenoisesoln}
u(t,x) = w(t,x) +\int_0^t \int_{\bR^d}  p_h(t-s,x-y) u(s,y) W_0(ds dy),
\end{equation}
where 
\begin{equation}
\label{e:h}
w(t,x)=(p_h*u_0)(t,x).
\end{equation}
For the wave equation, we have
\begin{equation}
\label{whitenoisesolnwave}
u(t,x) = w(t,x) +\int_0^t \int_{\bR^d}  p_w(t-s,x-y) u(s,y) W_0(ds dy),
\end{equation}
where
\begin{equation}
\label{e:w}
w(t,x)=\frac{d}{dt} (p_w*u_0)(t,x) + (p_w*v_0)(t,x).
\end{equation}
Notice that $w(t,x)$ are the same as the solution to the deterministic heat equation given in \eqref{heatsolution} and deterministic wave equation given in \eqref{wavesolution}. The stochastic integral in \eqref{whitenoisesoln} and \eqref{whitenoisesolnwave} is understood in the sense of martingale measure stochastic integration, see Walsh~\cite{walsh1986}.

More rigorously, consider $\mathcal{D}([0,\infty)\times \R^{d})$ to be the space of infinitely differentiable functions with compact support on $[0,\infty)\times \R^{d}$. On a complete probability space
$(\Omega,\cf,\bp)$, we consider a centered Gaussian family $\{W(\vp) ; \, \vp\in \mathcal{D}([0,\infty)\times \R^{d})\}$ such that
\begin{align}
E[W(\varphi)W(\psi)]=\int_{[0,\infty)} \int_{\R^{2d}}
\varphi(s,y)\psi(s,z) f(y-z) dy dz ds.
\end{align}

Let $\mathcal{H}$  be the completion of $\mathcal{D}([0,\infty)\times\R^d)$
endowed with the inner product
\begin{eqnarray}
\langle \varphi , \psi \rangle_{\mathcal{H}}&:=&
\int_{[0,\infty)} \int_{\R^{2d}}
\varphi(s,y)\psi(s,z) f(y-z) \, dy dz ds.
\end{eqnarray}
The mapping $\varphi \rightarrow W(\varphi)$ defined in $\mathcal{D}([0,\infty)\times\R^d)$ extends to a linear isometry between
$\mathcal{H}$ and $L^2(\Omega)$. We will denote this isometry by
\begin{equation}
\label{varphidef}
W(\varphi)=\int_0^{\infty}\int_{\R^d}\varphi(t,x)W(dt,dx)
\end{equation}
for $\phi \in \mathcal{H}$. Notice that if $\phi$ and $\psi$ are in
$\mathcal{H}$, then
$\be \lc W(\varphi)W(\psi)\rc =\langle\varphi,\psi\rangle_{\mathcal{H}}$. The theory of Walsh \cite{walsh1986}, extended by Dalang \cite{Dalang1999} shows that this extends to random integrands as well.

The main result in \cite{Dalang1999} establishes existence and uniqueness of a solution for both the heat and wave equations under the main assumption, known as \emph{Dalang's condition}.
\begin{assumption}[Dalang's condition]
\label{assumption:finite var}
The space covariance function and the spectral measure satisfy
$$\displaystyle{\int_{\mathbb{R}^d}\frac{\mu(d\xi)}{1+\norm{\xi}^2}<\infty}.$$
\end{assumption}
Moreover, it is known that for all $p \geq 2$, $$E[|u(t,x)|^p] < \infty, \text{ for all } t >0, x \in \R^d.$$


Recently, with the motivation to introduce some time correlation in the noise, several authors have considered the stochastic heat and wave equations driven by fractional noise (we call it \emph{standard fractional noise} in this dissertation). 

Informally, this noise is Gaussian and given by
\begin{equation} \label{eqn:frac_noise}
E[\dot{W}(t,x)\dot{W}(s,y)] = 
|t-s|^{2H-2} f(x-y),
\end{equation}
for $s,t > 0$ and $x,y \in \bR^d$, $f$ is the space covariance function as in \eqref{eqn:WN} and where $H \in (0,1)$ is known as the \emph{Hurst index}.

This noise is an extension of fractional Brownian motion to include a space variable. Fractional Brownian motion was introduced by Mandelbrot and Van Ness~\cite{mandelbrot1968}, as a Gaussian process with time dependence through the parameter $H$, called Hurst index, which estimates the intensity of long-range dependence.

\begin{definition}[Fractional Brownian motion]
Let $H\in (0,1)$. Fractional Brownian motion with Hurst index $H$ is a real-valued centered Gaussian process $(Z_H(t))_{t\in [0,\infty)}$, with $Z_H(0)=0$ almost surely and \\
\begin{align}
    E[Z_H(t)Z_H(s)]=\frac{1}{2}(t^{2H}+s^{2H}-|t-s|^{2H}),
\end{align}
for all $t,s\geq 0$.
\end{definition}

\begin{remark}
If we take $H=1/2$, then 
$Z_H(t)$ becomes standard Brownian motion and the fractional noise $\dot{W}$ becomes the time white noise $\dot{W}_0$.
\end{remark}

If $H >\frac{1}{2}$, the increments of the process are positively correlated and the closer $H$ is to 1, the stronger long-memory the process exhibits. If $H < \frac{1}{2}$ , the increments of fractional Brownian motion are negatively correlated. 

Fractional Brownian motion is widely used in financial applications. For example, in the modifications of the Black-Scholes SDE by Hu and \O ksendal~\cite{huoksendal2003}, namely
\begin{align}
    dS_t=S_t (\mu dt+\sigma dZ_H(t)),
\end{align}
the solution is given by
\begin{align}
S_t=S_0 \exp\{\sigma Z_H(t)+\mu t-\frac{1}{2}\sigma^2 t^{2H}\}.   
\end{align}
One of the advantage of introducing fractional Brownian motion is that there is actually a dependency of the stock prices on the past. An empirical study of daily returns from 1962 to 1987 shows that the Hurst index of S\&P 500 index is approximately 0.61 with a 95\% confidence interval (0.57,0.69), see \cite{Mackenzie2016}. 

Note that fractional Brownian motion is not a martingale, nor a semi-martingale. Hence all the theory of stochastic integration with respect to fractional Brownian motion has to be done in the Skorohod sense, using Malliavin Calculus \cite{malliavincalc}. It\^o Calculus can not directly be used because of the lack of martingale property.

In the context of stochastic partial differential equations (SPDEs), the fractional noise with covariance function $f$ is informally defined by
\begin{align}
    E[W(t,x) W(s,y)]=\frac{1}{2}\{t^{2H}+s^{2H}-|t-s|^{2H}\}f(x-y),
\end{align}
where $f$ is the space covariance function as in \eqref{eqn:WN}. Taking the derivative with respect to $t$ and $s$ yields
\begin{align}
    E[\dot W(t,x)\dot W(s,y)]= \alpha_H |t-s|^{2H-2} f(x-y),
\end{align}
where 
\begin{equation}
\label{fractionalconstant}
\alpha_H=H(2H-1).
\end{equation}
This corresponds to \eqref{eqn:frac_noise} up to a constant. 

More rigorously, the noise $W$ is built similarly as $W_0$ above, but where the inner product $\langle \cdot,\cdot\rangle_{\mathcal{H}}$ is replaced by
\begin{equation}\label{eqn:frac_innerproduct}
 \langle \varphi , \psi \rangle_{\mathcal{H}}:=
\int_{[0,\infty)^2} \int_{\R^{2d}}
\varphi(s,y)\psi(r,z) |s-r|^{2H-2} f(y-z) \, dy dz ds dr.   
\end{equation}
Existence and uniqueness for the heat equation driven by standard fractional noise is studied in Hu et al.~\cite{HuHuangNualartTindel} and the wave equation driven by standard fractional noise $W$ is studied in Balan and Tudor~\cite{BalanTudor}. In both cases, the solution is understood in (mild) integral form using Malliavin Calculus tools. More specifically it is understood in the sense of Skorohod (see e.g. Sanz-Sole~\cite{malliavincalc}). Indeed, in this case, the noise is not a martingale in time (unlike white noise) and the theory of Walsh \cite{walsh1986} (which is inspired by It\^o Calculus) does not apply. 

In this dissertation, we are mainly interested in the \emph{intermittency} property for the solution to the stochastic PDEs introduced above. This notion is related to Anderson localization and is introduced in the next section.

\section{Intermittency}
Intermittency properties are often characterized by the Lyapunov exponent in the literature, see \cite[Chapter 1]{carmona1994}, \cite[Theorem 2.2]{Foondun2008}. In the physics literature, a space–time random field is called \emph{physically intermittent} if it develops very high peaks concentrated on small spatial islands, as time becomes large \cite{BertiniCancrini}. To give a precise definition of intermittency for a random-field $u = \{u(t, x); t \geq 0, x \in \mathbb{R}^d \}$, we first introduce the upper Lyapunov exponent.

\begin{definition}[Lyapunov exponent]
For $p \in (0,\infty)$, the upper $p$-th moment Lyapunov exponent is defined as 
\begin{align}
    \gamma(p):=\limsup_{t\rightarrow \infty} \frac{1}{t} \ln \sup_{x\in \mathbb{R}^d}E(|u(t,x)|^p). 
\end{align}
\end{definition}

The mathematical definition of intermittency can now be defined in terms of the Lyapunov exponent.

\begin{definition}[Full intermittency] \label{def:fullint}
Let $p_0$ be the smallest integer $p$ for which  $\gamma(p)>0$, and $p_0=\infty$ if no such $p$ exists. When
$p_0<\infty$, we say that the random field $u(t,x)$ shows an (asymptotic) intermittency of order $p_0$. A random field $u$ is \emph{fully intermittent} if the map $p \mapsto \gamma(p)/p$ is strictly increasing for all $p \geq 2$. 
\end{definition}

Notice that, as a direct consequence of Jensen's inequality, the map $p \mapsto \gamma(p)/p$ is always increasing, but not necessarily strictly. When the map $p \mapsto \gamma(p)/p$ is strictly increasing, it means that for $p > q$, the moment of order $p$ of $u(t,x)$ is asymptotically of a larger order than its moment of order $q$. It is possible to show that in such a case, the random field must develop large, but rare values, see Carmona and Molchanov~\cite{carmona1994}. This establishes the connection between Definition \ref{def:fullint} with physical intermittency.

A weaker notion, known as weak intermittency is easier to establish, in particular for random fields defined by implicit integral equations as solutions to stochastic PDEs. In good situations, it implies full intermittency.

\begin{definition}[Weak intermittency]
A random field $u$ is called \emph{weakly intermittent} if 
\begin{align}
    \gamma(2)>0 \hspace{4mm}\text{ and } \hspace{4mm}\gamma(p)<\infty \hspace{4mm}\text{ for all }  p\geq 2.
\end{align}
\end{definition}
\begin{remark}
    If $\gamma(1)=0$ and $u(t,x)\geq 0$ almost surely, then weak intermittency implies full intermittency. This property is a consequence of the convexity of $\gamma(p)$, see Carmona and Molchanov~\cite[Theorem III.1.2]{carmona1994} for details.
\end{remark}

We will now provide three basic examples of random fields: one fully intermittent, one weakly intermittent and one that is not intermittent, so that the reader can develop a better intuition about it.

\begin{example}
Consider a standard Brownian motion $(B_t)_{t\geq0}$. We have
\begin{align}
     \no  \limsup\limits_{t\rightarrow \infty}\frac{1}{t}\ln E(B_t^p) = \limsup\limits_{t\rightarrow \infty} \frac{1}{t}\ln(t^{p/2}) = 0.
\end{align}
The moments do not exhibit exponential growth and the intermittency phenomenon does not happen.
\end{example}

\begin{example}
Consider a process $(X_t)_{t \geq 0}$ solution to the following stochastic differential equation:
\begin{equation*}
   dX_t=rX_tdt+\sigma X_tdW_t
\end{equation*}
with $X_0 =1$. The solution is given by 
\begin{equation*}
X_t=\exp \left((r-\frac{\sigma^2}{2})t+\sigma W_t\right).
\end{equation*}
Now we calculate the order of the $p$-th moment:
\begin{align}
\no    E(X_t^p)=\exp\left((r-\frac{\sigma^2}{2})pt+\frac{\sigma^2p^2}{2}t\right).
\end{align}
When $t\rightarrow \infty$, we obtain
\begin{align}
 \no  \gamma(p)=\limsup\limits_{t\rightarrow \infty}\frac{1}{t}\ln E(X_t^p) \approx p^2
\end{align}
Hence $(X_t)$ has the full intermittency property since $\gamma(p)/p$ is strictly increasing.
\end{example}

\begin{example}
In discrete time, assume we have the process $(S_n)_{n \geq 0}$ defined as follows. Consider an i.i.d sequence of random variables $(X_j)_{j \geq 0}$ which take values $0$ or $2$ with probability $1/2$. Let $S_n = \prod_{j=1}^{n} X_j$. Then the distribution of $S_n$ is given by
\begin{equation}
  S_n = \begin{cases}
\no    &0   \hspace{10mm} \text{with probability } 1-1/2^n,\\
    &2^n   \hspace{8mm} \text{with probability } 1/2^n.
    \end{cases}
\end{equation}

It is easy to see that $E(S_n^p)=2^{n(p-1)}$. Then the Lyapunov exponent becomes 
\begin{align}
     \no  \limsup\limits_{n\rightarrow \infty}\frac{1}{n}\ln E(S_n^p) = \limsup\limits_{n\rightarrow \infty} \frac{1}{n}\ln(2^{n(p-1)}) \approx p
\end{align}
In this case, we have weak intermittency since $\gamma(p)$ is non-trivial, but $\gamma(p)/p$ is constant.
\end{example}

For the stochastic heat equation \eqref{eqn:DO process} with the noise $W_0$ given in \eqref{eqn:WN}, Foondun and Khoshnevisan established intermittency and determined the Lyapunov exponents in \cite{Foondun2008}. The stochastic wave equation case is studied by Dalang and Mueller in \cite{Mueller2009}. In particular, in the case of space-time white noise (i.e. $f = \delta_0$), we have
$$E[|u(t,x)|^p] \sim \exp(p^{\gamma} t),$$
where $\gamma = 3$ for the stochastic heat equation, and $\gamma = 3/2$ for the stochastic wave equation. An exact calculation of the Lyapunov exponent for the solution to the stochastic heat equation has been provided recently by Xia Chen~\cite{xiachenprecise}, namely
$$E[|u(t,x)|^p] = \exp\left(\frac{p(p^2-1)}{24} t\right).$$

When the equations are driven by fractional noise \eqref{eqn:frac_noise}, intermittency and the Lyapunov exponent are established in Hu et al.~\cite{HuHuangNualartTindel} for the stochastic heat equation and in Balan-Conus~\cite{Balan2016} for the stochastic wave equation (See also Song et al.~\cite{songbound}). In the latter case, a lower bound is only available for $p=2$.  For a noise that is fractional in time and white in space (i.e. $f=\delta_0$), we have
\begin{equation}
\label{standardfracnoiseexponent}
E[|u(t,x)|^p] \sim \exp(p^{\gamma} t^{\rho}),
\end{equation}
where $\gamma$ is the same as above, but $\rho = 4H-1$ for the stochastic heat equation and $\rho = H+\frac{1}{2}$ for the stochastic wave equation.

\section{An alternative fractional noise: the DO noise}
\label{sec:DOnoise}

When considering a noise that is correlated in time, the standard fractional noise introduced above in \eqref{eqn:frac_noise} presents a difficulty: the solution can only be defined using Malliavin Calculus. This highly technical theory makes some of the questions of interest, such as intermittency, more difficult to obtain. Moreover, more advanced results regarding the position or size of the peaks, as introduced in \cite{conuskhoshnevisan2013}, cannot be directly extended to this case.

In this work, we will consider an alternative, different form of fractional noise, the \emph{Dobri\'c-Ojeda noise}. The Dobri\'c-Ojeda noise is an extension of the Dobri\'c-Ojeda process to include a space variable. The Dobri\'c-Ojeda process (DO process for short) is introduced in \cite{dooriginal}, to be a process $(V_s)_{s \geq 0}$ that satisfies
\begin{align}\label{eqn:defDOprocess}
    dV_s = t^{H-1/2} dB_s + (2H-1) \frac{1}{s} V_s \, ds,
\end{align}
where $(B_s)$ is a standard Brownian motion.
The DO process is shown to provide a reasonable approximation to fractional Brownian motion for most values of $H$, while being a semi-martingale, and being represented as an Ito diffusion.

The Dobri\'c-Ojeda process is used in finance applications, see for instance Wildman~\cite{Mackenzie2016}, Wildman and Conus~\cite{wildmanConus}, Carr and Itkin~\cite{carritkin}. In these financial applications, we find that models involving the DO process actually give good approximations of option prices (\cite{Mackenzie2016}) or fixed-income derivatives (\cite{carritkin}), compared to using fractional Brownian motion when the parameter $H$ is similar. 

The results obtained with the DO process motivate us to generalize the fractional noise for applications in stochastic PDEs, as was introduced in \cite{wildmanConus}. In the context of SPDEs, we assume that the noise has no drift, since a drift term is usually not the relevant factor for intermittency (compare \eqref{eqn:defDOprocess} with Defintion \ref{DO noise def} below).

In the context of Stochastic PDEs, the DO noise will allow us to use Walsh integrals to define a solution and avoid Malliavin Calculus, which the standard fractional noise does not allow.

\begin{definition}[Dobri\'{c}-Ojeda noise]
\label{DO noise def}
Let $H\in (0,1)$. The Dobri\'c-Ojeda noise is a centered Gaussian noise with covariance function given by 
\begin{align*}
        E[\dot{W}(s,x)\dot{W}(t,y)] = s^{H-1/2}t^{H-1/2}\delta_0(t-s)f(x-y),
\end{align*} 
where $t,s > 0$, $x,y \in \R^d$. More rigorously, the Dobri\'c-Ojeda noise is defined similarly as $W_0$, but where the inner product $\langle \cdot,\cdot\rangle_{\mathcal{H}}$ is replaced by
\begin{equation} \label{eqn:DO_innerproduct}
 \langle \varphi , \psi \rangle_{\mathcal{H}}:=
\int_{[0,\infty)} \int_{\R^{2d}}
\varphi(s,y)\psi(s,z) s^{2H-1} f(y-z) \, dy dz ds.   
\end{equation}
\end{definition}

In order for the definition above to make sense rigorously, we need to be able to integrate with respect to the Dobri\'c-Ojeda noise. This is done as follows. Let $F:\bR^d \times [0,\infty) \rightarrow \R$ be a function such that 
\begin{align}
\label{eqn: variance}
    \int_0^{\infty} \int_{\bR^d} F^2(y,s) s^{2H-1} dy ds < \infty.
\end{align}
We define
\begin{align}
&\int_0^{\infty} \int_{\R^d} F(y,s) W(dyds) := \int_0^{\infty} \int_{\bR^d} F(y,s) s^{H-1/2} W_0(dyds),
\end{align}
where $W_0$ stands for the time white noise with space covariance function $f$ defined in \eqref{eqn:WN}. By \eqref{eqn: variance}, the integral on the right of \eqref{Do definition eqn} is well-defined by Walsh~\cite{walsh1986} and Dalang~\cite{Dalang1999}.

In other words, this corresponds to considering the DO noise to be given by
\begin{align}
\label{Do definition eqn}
    W(dsdy)=s^{H-1/2}W_0(dsdy).
\end{align}

\begin{remark} \label{covariancedoesnotmatter}
One remarkable aspect of the DO noise is that it shares the same order of variance as the standard fractional noise, despite their different covariance structure.

We will calculate the variance for both fractional noise and DO noise here to show their similarity.

Let $t > 0$ and $\mathbb{A} \subset \bR^d$ a Borel set. We will consider $W(\varphi)$ for both the standard fractional noise and the DO noise, where $\varphi(s,y)=\mathbbm{1}_{[0,t]}(s)\mathbbm{1}_{\mathbb{A}}(y)$. It is not hard to show that $\varphi \in \mathcal{H}$ in both cases.

By \eqref{eqn:DO_innerproduct}, we have
\begin{align}
\label{varianceforDO} E[W(\varphi)^2]&=\int_0^t\int_0^t\int_{\mathbb{A}}\int_{\mathbb{A}}s^{H-1/2}r^{H-1/2}\delta_0(s-r)f(y-z)dydzdsdr\\
\no &=\int_0^t ds s^{2H-1}\int_{\mathbb{A}}\int_{\mathbb{A}}dydz f(y-z)\\
\no &=C(\mathbb{A}) \frac{t^{2H}}{2H},
\end{align}
where $C(\mathbb{A})$ is a constant independent of $t$.

For the standard fractional noise, we have by \eqref{eqn:frac_innerproduct},
\begin{align*}
E[W(\varphi)^2]& = \int_0^t\int_0^t\int_{\mathbb{A}}\int_{\mathbb{A}}|s-r|^{2H-2}f(y-z)dydzdsdr\\
& =2\int_0^t ds \int_s^t (r-s)^{2H-2}dr \int_{\mathbb{A}}\int_{\mathbb{A}}f(y-z)dydz \\
 &=C(\mathbb{A}) \frac{t^{2H}}{H(2H-1)}.
\end{align*}
We can see that they both have the same order $t^{2H}$ in time for the variance.
\end{remark}


\section{Main results and outline}

The main objective of this dissertation is to study the intermittency property of the solution to the stochastic heat and wave equations driven by the DO noise introduced in Section \ref{sec:DOnoise}.

We will establish existence and uniqueness of the solution and obtain both upper and lower bounds on the moment Lyapunov exponents for the two equations. Existence and uniqueness in the DO noise case will be based on the general results of Dalang \cite{Dalang1999}. The upper bound on the Lyapunov exponent will be established by a careful choice of norm to estimate the asymptotic order of the moments, as was proposed by Foondun and Khoshnevisan \cite{Foondun2008}. In the DO noise case, the bounds will require the use of special functions of fractional types (i.e. incomplete gamma function or Miller-Ross function), similarly as in \cite{Balan2016}. 

In \cite{dalang_mueller_tribe}, Dalang, Mueller, and Tribe introduced a Feynman-Kac type representation of the moments of the solution to certain SPDEs using Poisson processes. This representation was used to study intermittency for the stochastic wave equation drive by time white noise in \cite{Mueller2009} and by fractional noise in \cite{Balan2016}. We will take advantage of a similar representation to establish the lower bound on the Lyapunov exponent in the DO noise case.

The most remarkable fact is that these exponents exactly match the ones established for standard fractional noise in \cite{HuHuangNualartTindel} for the heat equation and in \cite{Balan2016} for the wave equation. Hence, it appears that despite their difference in correlation structure, both the solutions driven by standard fractional noise or the DO noise share the same Lyapunov exponent and, hence, the same level of intermittency. 

From there, the general conjecture that we aim to investigate is that the Lyapunov exponent of the solution to the heat or wave equation is only determined by the asymptotic variance of the noise but not the specific covariance structure (check Remark~\ref{covariancedoesnotmatter}).

As a first step towards understanding this conjecture, we will provide a framework, that we call \emph{generalized noise}, that will contain both the standard fractional noise and the DO noise as special cases. This will provide a better understanding on the reasons why both noises share similar Lyapunov exponents while starting with different covariance structures.

In the generalized fractional noise setting, we prove that the Lyapunov exponent only depends on the order of the variance of the noise, and not the specific covariance structure. Hence, the heat equation driven by an interpolation of DO noise and fractional Brownian motion will maintain the same intermittency level throughout, see Remark~\ref{covariancedoesnotmatter}. We conjecture that this will also hold for the wave equation, but this is subject of ongoing research.
 
The dissertation is organized as follows. The second chapter will focus on the existence and uniqueness of the solution to \eqref{eqn:DO process} and \eqref{eqnwave:DO process} driven by the DO noise defined in Section \ref{sec:DOnoise}. In the third chapter, we provide the derivation of the formula for the $n$-th moment of the solution $u(t,x)$. The fourth chapter will explore the Lyapunov exponent for the solution (both the upper and the lower bound). From Chapter 5 to Chapter 8, we will follow a similar program when considering the stochastic heat equation \eqref{eqn:DO process} driven by a generalized fractional noise interpolating between the DO noise and the standard fractional noise. Finally some technical lemmas are gathered in the Appendix.

\chapter{Existence and Uniqueness of solution for Dobri\'c-Ojeda noise}
In this chapter, we will establish the existence and uniqueness of a solution to \eqref{eqn:DO process} and \eqref{eqnwave:DO process} driven by the DO noise. We first start by a few technical settings.
\section{Basic settings}

In this dissertation, we will consider some specific forms of the space covariance function $f$, mostly focusing on the Riesz kernel.
 
\begin{definition}[Riesz Kernel]
We call the Riesz kernel with exponent $\alpha$ the function given by
\begin{align}
\label{riesz kernel}
f(x)=|x|^{-\alpha} \text{ for some } 0 < \alpha < d.
\end{align}
Its Fourier transform is given by
\begin{align}
    \hat{f}(\xi)= C_{\alpha,d}|\xi|^{(d-\alpha)}, 
\end{align}
where $C_{\alpha, d}$ is a constant independent of $\xi$.
 \end{definition}
 
 Namely, we will be interested in the three specific cases below for the covariance function 

\begin{equation}
\label{threecasesforf}
    f(x)=|x|^{-a}
\end{equation}

\begin{itemize}
    \item Case (1): 
    the function $f$ satisfies
    $f(x) < A$ for all $x \in \bR^d$ (see Assumption \ref{assumption:bounded}). In this case, we can show that the noise is continuous and differentiable, and we refer to it as \emph{spatially smooth noise}. Note that this condition implies that $\mu$ is a finite measure ($\mu(\bR^d) < \infty$).
    \item Case (2): the function $f$ is the Riesz kernel provided in Definition \ref{riesz kernel}. 
    \item Case (3): 
    space-time white noise in dimension 1, i.e. when $f(x) = \delta_0(x)$. Notice here that since $d=1$, we can let $\alpha \rightarrow 1$ in the Riesz kernel case. Then the solution converges to the spatial white noise case. Informally, the main idea is to replace $\delta_0(x)$ with $|x|^{-1}$ since they have the same Fourier transform. This argument can be made rigorous, as in \cite{Balan2016}. 
\end{itemize}
For each case above, we will consider equations \eqref{eqn:DO process} and \eqref{eqnwave:DO process} driven by the noise $W$ defined in \eqref{Do definition eqn}. 
We need to introduce a paramater $a$ depending on the cases above. We define
\begin{equation} \label{eqn:a_param}
a := \left\{ \begin{array}{ll} 
0 & \text{in case (1),} \\
\alpha & \text{in case (2),} \\
1 & \text{in case (3).}
\end{array} \right.
\end{equation}

In each of cases (1), (2), and (3), it is easy to check that the covariance function $f$ satisfies Dalang's condition (Assumption \ref{assumption:finite var}) provided that
\begin{equation}\label{eqn:alphaless2}
    a < \min\{2,d\}.
\end{equation}

\section{Existence and Uniqueness of the solution}

We are now ready to state the main result on existence and uniqueness.

\begin{theorem}
\label{thm:existence}

Let $W$ be the noise introduced in Definition~\ref{DO noise def}. Let $H\in (a/4,1)$, where $a$ is defined in \eqref{eqn:a_param}. The stochastic heat equation \eqref{eqn:DO process} has an almost-sure unique solution $u$ that satisfies\\
\begin{align}
    \label{eqn:solution}
\no    u(t,x)&=\int_{\bR^d} u_0(y)\cdot p_{h}(t,x-y)dy+\int_0^t\int_{\bR^d}  u(s,y) \cdot p_{h}(t-s,x-y)W(dyds)\\
    &=w(t,x)+\int_0^t\int_{\bR^d}  u(s,y) \cdot p_{h}(t-s,x-y)s^{H-1/2}W_0(dyds),
\end{align}
where $p_{h}$ is the heat kernel introduced in \eqref{heatkernel} and $w(t,x)$ is the solution to the deterministic heat equation given in \eqref{heatsolution}. Also, $u$ satisfies 
\begin{align}
    \sup_{x\in \mathbb{R}}\sup_{0\leq t\leq T} \mathbb{E}(|u(t,x)|^2)<\infty  \hspace{5mm}  \text{for all } T>0.
\end{align}
\end{theorem}

A similar result holds for the stochastic wave equation.

\begin{theorem}
\label{thm:existencewave}

Let $W$ be the noise introduced in Definition~\ref{DO noise def}. Let $H\in (\frac{a-2}{2},1)$, where $a$ is defined in \eqref{eqn:a_param}. The stochastic wave equation \eqref{eqnwave:DO process} has an almost-sure unique solution $u$ that satisfies
\begin{align}
    \label{eqn:solutionwave}
\no    u(t,x)&=\int_{\bR^d} v_0(y)\cdot p_{w}(t,x-y)dy+\frac{\partial}{\partial t} \int_{\bR^d} u_0(y) p_w(t,x-y) dy\\
\no&+\int_0^t\int_{\bR^d}  u(s,y) \cdot p_{w}(t-s,x-y)W(dyds)\\
    &=w(t,x)+\int_0^t\int_{\bR^d} u(s,y) \cdot p_{w}(t-s,x-y)s^{H-1/2}W_0(dyds),
\end{align}
where $p_{w}$ is the fundamental solution to the wave equation introduced in \eqref{wave kernel} and $w(t,x)$ is the solution to the deterministic wave equation given in \eqref{wavesolution}. Also, $u$ satisfies 
\begin{align}
    \sup_{x\in \mathbb{R}}\sup_{0\leq t\leq T} \mathbb{E}(|u(t,x)|^2)<\infty
    \hspace{5mm}  \text{for all } T>0.
\end{align}
\end{theorem}

\begin{remark}
There is a trade-off between the space noise and time noise as well. In Theorem \ref{thm:existence} we can see that $\alpha/4<H<1$. In dimension $d$, the case $\alpha = d$ corresponds to space white-noise. So if we want to have a solution for space-time white noise ($H=1/2$), we will need to choose $d<2$. This illustrates the known result that the heat equation driven by space-time white noise admits a solution in $d=1$ only (see \cite{Dalang1999}). However, if we want the white noise in space for $d \geq 2$, then the time noise must be chosen so that $H>1/2$, i.e. smoother than white noise. On the other hand, if we want white noise in time ($H=1/2$), we will require $\alpha < 2$. i.e. smoother space noise. Hence, we can always handle rougher noise in space or time, provided that the other one is chosen smooth enough. Note that for $d \geq 4$, in order to have white noise in space, we would have to choose $H > 1$. This is acceptable for DO noise, but the comparison with standard fractional noise is lost. 
\end{remark}
\begin{proof}[Proof of Theorem~\ref{thm:existence}]
Since the DO noise is defined directly from the space-time white noise, the proof of existence and uniqueness will rely on general results established by Dalang~\cite{Dalang1999} based on martingale techniques. According to Dalang~\cite[Theorem 13]{Dalang1999} (see also \cite{erratumDalang1999}), we have three requirements for a solution to exist.

Let $ \Gamma(t,s,x-y):=p_h(t-s,x-y) s^{H-1/2}\mathbbm{1}_{[0,t]}(s)$. The function $\Gamma $ satisfies Hypothesis B in \cite{Dalang1999} by the three lemmas below. This is sufficient to establish the existence and uniqueness of the solution.
\end{proof}
\begin{lemma}

The total mass $\Gamma(t,s,\mathbb{R}^d)$ is finite,i.e.,\\
\begin{align}
    \int_{\bR^d} \Gamma(t,s,x-y)dy <\infty \text{ for all $x$},
\end{align}
see \cite[Assumption B and Remark 20]{Dalang1999}.
\end{lemma}
\begin{proof}
Since $p_h(t-s,x-y)$ is the heat kernel, the integral is equal to 1. Hence
\begin{align}
    \int_{\mathbb{R}^d}\Gamma(t,s,x-y)\,dy=s^{H-1/2}\mathbbm{1}_{[0,t]}(s).
\end{align}
This is finite for all $s>0$.\\
\end{proof}

\begin{lemma}\label{lemma:FTL2}
 The Fourier transform of $\Gamma(t,s,\cdot)$ is square integrable with respect to the spectral measure $\mu$:
\begin{align}
\label{eqn:finite FT l2}
    \int_0^\infty \,ds\int_{\mathbb{R}^d}\mu(d\xi)|\mathcal{F}\Gamma(t,s)(\xi)|^2<+\infty,
\end{align}
In our case $\mathcal{F}\Gamma(t,s)(\xi)=\hat{p}_{t-s}(\xi) s^{H-1/2}\mathbbm{1}_{[0,t]}(s) $, where $\mathcal{F}\Gamma(t,s)(\xi)$ is the Fourier transform of $\Gamma (t,s,\cdot)$. See Dalang\cite[Theorem 2]{Dalang1999}.
\end{lemma}
\begin{proof}
\begin{align}
\label{eqn:l2}
    \no\int_0^t\int_{\mathbb{R}^d}\hat{p}^2_{t-s}(\xi) s^{2H-1}\mu(d\xi)ds&=\int_0^t s^{2H-1}\int_{\mathbb{R}^d}e^{-(t-s)|\xi|^2}|\xi|^{-(d-\alpha)}\,d\xi ds\\
    &=\int_0^t s^{2H-1}(t-s)^{-\frac{\alpha}{2}}\Bigg(\int_{\mathbb{R}^d}e^{-|\eta|^2}|\eta|^{-(d-\alpha)}\,d\eta\Bigg)ds.
\end{align}
We make the substitution $r=\frac{s}{t}$, and notice that the space integral is a constant $C_{\alpha,d}$. Then \eqref{eqn:l2} is equal to\\
\begin{align}
\label{eqn:H value}
C_{\alpha,d} t^{2H-\frac{\alpha}{2}}\int_0^1(1-r)^{-\frac{\alpha}{2}}r^{2H-1}\,dr.
\end{align}
In order for \eqref{eqn:H value} to be finite, we need $H>\alpha/4$. 
\end{proof}
\begin{remark}
Notice that $H\geq 1$ is acceptable in Lemma \ref{lemma:FTL2}, but it does not make sense for standard fractional noise, so we only consider the case where $H \in (\alpha/4,1)$.
\end{remark}

\begin{remark}
Notice that when $H=1/2$, we recover the case $\alpha<4H=2$, which correspond the condition $\alpha=\min(2,d)$ (see \eqref{eqn:alphaless2}).
\end{remark}

\begin{lemma}
The function $\Gamma$ satisfies:
(1) $s \mapsto \mathcal{F}\Gamma (t,s)(\xi)$ is continuous, for all $\xi \in \bR^d, t>0$; \\
(2) The following equation holds true:
\begin{align}
\label{assumption:fourier transform conv}
    \lim_{h\downarrow 0} \int_0^\infty\,ds\int_{\mathbb{R}^d}\mu(d\xi)\sup_{s<r<s+h}|\mathcal{F}\Gamma(t,r)(\xi)-\mathcal{F}\Gamma(t,s)(\xi)|^2=0.
\end{align}
\end{lemma}
\begin{proof}
Property one is clear from \eqref{eqn:heatFourier}. As for (2), for $s<r<s+h$,
\begin{align}
\no|\mathcal{F}\Gamma(t,r)(\xi)-\mathcal{F}\Gamma(t,s)(\xi)|^2 &\leq |e^{-(s+h)|\xi|^2}-e^{-s|\xi|^2}|^2\\
&=e^{-2s|\xi|^2}|e^{-h|\xi|^2}-1|^2.
\end{align}
Since $|e^{-h|\xi|^2}-1|^2\leq 4$ for all $\xi\in \mathbb{R}^d, h>0$, by the Dominated Convergence Theorem and Lemma \ref{lemma:FTL2},
\begin{align}
\no \lefteqn{\lim_{h\downarrow 0}  \int_0^t\,ds\int_{\mathbb{R}^d}\mu(d\xi)\sup_{t<r<t+h}|\mathcal{F}\Gamma(t,r)(\xi)-\mathcal{F}\Gamma(t,s)(\xi)|^2}\\
&\leq \int_0^t\,ds\int_{\mathbb{R}^d}\lim_{h\downarrow 0}|e^{-h|\xi|^2}-1| \cdot e^{-s|\xi|^2}  \mu(d\xi) =0.
\end{align}
\end{proof}
Since all three requirements are satisfied, we have the existence and uniqueness of a solution to stochastic heat equation. The result for the stochastic wave equation (Theorem \ref{thm:existencewave}) follows in a similar manner, replacing the definition of $\Gamma$ by 
$$\Gamma (t,s,x-y)=p_w(t-s,x-y)s^{H-1/2}\mathbbm{1}_{[0,t]}(s) $$ 
in the three requirements above. The requirements are proved using the same technique, in particular using that $\mathcal{F}\Gamma(t,s)(\xi) = \frac{\sin(s|\xi|)}{|\xi|}\mathbbm{1}_{[0,t]}(s)$. For all $\xi\in \mathbb{R}^d$, equation \eqref{eqn:finite FT l2} with $p_w$ yields the condition $H>\frac{a-2}{2}$.

\chapter{Moment formula for the DO noise case}
In this chapter, we will establish an explicit formula for the moments of the solution to equation~\eqref{eqn:DO process}. This moment formula relies on a representation technique of the moments in terms of expectations with respect to Poisson processes. It was initially introduced by Dalang, Mueller and Tribe~\cite{dalang_mueller_tribe}. This will be the key tool used to obtain the lower bound on the Lyapunov exponent in Chapter 4.
\section{Preliminaries}
We consider a stochastic process $u(t,x)$ that is solution to the following stochastic PDE, in integral form
$$u(t,x) = w(t,x) + \int_{0}^{t} \int_{\R^d} s^{H-1/2} p_h(t-s,x-y) u(s,y) W_0(ds,dy),$$
where $H > a/4$, $w(t,x)$ is the solution to the corresponding deterministic equation, the noise $W$ is white in time and correlated in space, with covariance informally given by
$$E[W(t,x)W(s,y)] = \delta_0(t-s) f(x-y),$$
(see \eqref{eqn:WN}). The parameter $a$ is defined in \eqref{eqn:a_param} as a function of the space covariance function $f$, and $p_h$ is the heat kernel defined in \eqref{heatkernel}.

Now, consider a Poisson process $(N_t)_{t \geq 0}$ with rate 1, and let $(\tau_n)_{n \geq 0}$ denote its jump times. Also, consider $B^{(i)} = \left(B^{(i)}_t\right)_{t \geq 0}$ to be i.i.d copies of a $d$-dimensional standard Brownian motion. We define a stochastic process $X = (X_t)_{t \geq 0}$ as follows: for $0 < t \leq \tau_1$, let
\begin{align}
X_t = X_0 + B^{(1)}_t.
\end{align}
For $i \geq 1$ and $\tau_i < t \leq \tau_{i+1}$, let
\begin{align}
\label{processformation}
X_t = X_{\tau_i} + B^{(i+1)}_{t - \tau_i}.
\end{align}
We denote by $P_x$ the probability measure under which $X_0 = x$ and $E_x$ the expected value under $P_x$.

In this situation, we build the process $(X_t)$ from a set of independent Brownian motions, since the operator $\Delta$ arising in the stochastic heat equation is the infinitesimal generator of Brownian motion. A direct consequence of this is that the heat kernel $p_h(t,\cdot)$ is the probability density function of $B_t$.

\section{Second moment formula.}

We will establish the following result.

\begin{theorem}
 \label{thm:2nd_mom_formula}
We have
\begin{align}
\no \lefteqn{E[u(t,x)u(t,y)]}\\
&= e^t E_x\left[w(t - \tau_{N_t}, X^{1}_{\tau_{N_t}}) w(t - \tau_{N_t}, X^{2}_{\tau_{N_t}}) \prod_{i=1}^{N_t} \left((t-\tau_i)^{2H-1} f(X^1_{\tau_i} - X^2_{\tau_i})\right)\right],
\end{align}
where $X^1$ and $X^2$ are two i.i.d copies of the process $X$ defined above.
\end{theorem}

\begin{proof}
By a standard argument for non-linear SPDE of affine type (e.g. Parabolic Anderson model, see \cite{carmona1994}), it is easy to see that the process $u(t,x)$ can be written as 
\begin{equation}
u(t,x) = \sum_{m=0}^{\infty} I_m(t,x), \label{eqn:series}
\end{equation}
where $I_0(t,x) = w(t,x)$ and, for $m \geq 1$, the process is defined inductively by 
\begin{equation}
I_m(t,x) = \int_{0}^{t} \int_{\R^d} s^{H-1/2} p_{t-s}(x-z) I_{m-1}(s,z) W_0(ds,dz). \label{eqn:Im}
\end{equation}
Note that the processes $I_m$ are orthogonal in $L^2(\Omega)$, since they are all in Wiener chaoses of different order.

We will prove by induction on $m$ that
\begin{equation} \label{eqn:2nd_mom_Im}
E[I_m(t,x)I_m(t,y)] = J_m(t,x,y),
\end{equation}
where
\begin{align} 
\label{heatexpectation}
\no \lefteqn{J_m(t,x,y)}\\
&= E_x\left[\ind_{\{N_t = m\}} w(t - \tau_m, X^{1}_{\tau_m}) w(t - \tau_m, X^{2}_{\tau_m}) \prod_{i=1}^{m} \left((t-\tau_i)^{2H-1} f\left(X^1_{\tau_i} - X^2_{\tau_i}\right)\right) \right].
\end{align}
Once this is established, the result of Theorem \ref{thm:2nd_mom_formula} follows by direct summation on $m$ using the orthogonality of the processes $I_m$.

Equation \eqref{eqn:2nd_mom_Im} in the case $m=0$ is immediate. We now assume that \eqref{eqn:2nd_mom_Im} holds up to $m-1$ and will establish the result for $m$. Let $\mathcal{F}_1 = \sigma(\tau_1, X^1_{\tau_1}, X^2_{\tau_1})$. By standard results on conditional expectations, 
\begin{align*}
\lefteqn{J_m(t,x,y)} \\
& = E_x\Bigg[\ind_{\{\tau_1 \leq t\}} e^{\tau_1} (t-\tau_1)^{2H-1} f\left(X^1_{\tau_1} - X^2_{\tau_1}\right) \phantom{\prod_{i=2}^{m}}  \\
& \phantom{E_x\Bigg[}\times  E_x\Bigr[\ind_{\{N_t-N_{\tau_1} = m-1\}} w(t - \tau_m, X^{1}_{\tau_m}) w(t - \tau_m, X^{2}_{\tau_m})\\
& \phantom{E_x\Bigg[ \ind_{\{N_t-N_{\tau_1}}}\times \prod_{i=2}^{m} \left((t-\tau_i)^{2H-1} f\left(X^1_{\tau_i} - X^2_{\tau_i}\right)\right) \mid \mathcal{F}_1 \Bigr] \Bigg].
\end{align*}
By the strong Markov property of the Poisson process $(N_t)$ at time $\tau_1$, the independence and stationarity of the increments of the processes $X^{i}$ ($i=1,2$) up to and after $\tau_1$, the conditional expectation above is equal to $J_{m-1}(t-\tau_1,X^1_{\tau_1},X^2_{\tau_2})$. Hence, we obtain
\begin{align*}
J_m(t,x,y) & = E_x\left[\ind_{\{\tau_1 \leq t\}} e^{\tau_1} (t-\tau_1)^{2H-1} f\left(X^1_{\tau_1} - X^2_{\tau_1}\right) J_{m-1}(t-\tau_1,X^1_{\tau_1},X^2_{\tau_1}) \right].
\end{align*}
Using first the fact that the random variables $X^{1}_t$ and $X^{2}_t$ are independent, both with centered normal distribution with variance $t$, and then that the law of $\tau_1$ is exponential with parameter 1, we can rewrite the expectation as an integral. Namely,
\begin{align*}
\lefteqn{J_m(t,x,y)} \\
 & = E\left[ \ind_{\{\tau_1 \leq t\}} e^{\tau_1} (t-\tau_1)^{2H-1} \phantom{\int_{\R^d}} \right. \\
 & \qquad \times \left. \int_{\R^d} dz_1 \int_{\R^d} dz_2 \, p_{\tau_1}(x-z_1) p_{\tau_1}(x-z_2) f\left(z_1 - z_2\right) J_{m-1}(t-\tau_1,z_1,z_2) \right] \\
 & = \int_{0}^{\infty} ds \, e^{-s} \ind_{\{s \leq t\}} e^{s} (t-s)^{2H-1} \\
 & \qquad \times \int_{\R^d} dz_1 \int_{\R^d} dz_2 \, p_{s}(x-z_1) p_{s}(y-z_2) f\left(z_1 - z_2\right) J_{m-1}(t-s,z_1,z_2) \\
  & = \int_{0}^{t} ds \, s^{2H-1} \int_{\bR^d} dz_1 \int_{\bR^d} dz_2 \, p_{t-s}(x-z_1) p_{t-s}(y-z_2) f\left(z_1 - z_2\right) J_{m-1}(s,z_1,z_2).
\end{align*}
Now, using the induction assumption, we can write $J_{m-1}$ as an expectation in the terms of the process $I_m$. Using standard results about moments of Walsh integrals, we obtain
\begin{align*}
\lefteqn{J_m(t,x,y)} \\
& = \int_{0}^{t} ds \, s^{2H-1} \int_{\bR^d} dz_1 \int_{\bR^d} dz_2 \, p_{t-s}(x-z_1) p_{t-s}(y-z_2) f\left(z_1 - z_2\right)  \\
&\phantom{\int_{0}^{t} ds }\times E[I_{m-1}(s,z_1)I_{m-1}(s,z_2)]\\
& = E\left[\int_{0}^{t} \int_{\bR^d} s^{H-1/2} p_{t-s}(x-z_1) I_{m-1}(s,z_1) W_0(ds,dz_1) \right. \\
& \phantom{\int_{0}^{t} ds }\times \left. \int_{0}^{t} \int_{\bR^d} s^{H-1/2} p_{t-s}(y-z_2) I_{m-1}(s,z_2) W_0(ds,dz_2) \right] \\
& = E[I_m(t,x)I_m(t,y)].
\end{align*}
The result is proved.
\end{proof}

\begin{remark}
\label{generalnoiseformula}
The reader will notice that 
\begin{align}
    \sum I_m(t,x)=\sum J_m(t,x,x)= \int_0^t ds_1\int_0^t ds_2\, e^{-s_1}e^{-s_2} G(s_1,s_2,x,x) \varphi(s_1,s_2),
\end{align}
 where 
 \begin{align*}
 \lefteqn{G(s_1,s_2,x,y)}\\
 =&\int_{\bR^d}dz_1\int_{\bR^d}dz_2\,p_h(t-s_1,x-z_1)f(z_1-z_2)p_h(t-s_2,y-z_2) u(s_1,x)u(s_2,y),
 \end{align*}
and $\varphi(s_1,s_2)=s_1^{H-1/2}s_2^{H-1/2}\delta_0(s_1-s_2)$.
\end{remark}

\section{$n$-th moment formula}

We will generalize the result of the previous subsection to moments of order $n$. First of all, let's establish some additional settings, following \cite{dalang_mueller_tribe}.

First notice that the processes $(X^1)$ and $(X^2)$ built in the previous sections were nothing else than Brownian motions. The specific construction, restarting the Brownian process at Poisson jump times, allowed to understand the role of the Poisson jump times better and will help with the proof in the wave equation case. 

In this section, instead of considering two processes, we will consider $n$ copies of the process $(X_t)$. These $n$ copies will be conditioned to restart the Brownian process by pairs at some specific Poisson jump times.

More precisely, we consider $\mathcal{P}_n$ to be the set of unordered pairs of distinct elements  of the set $\{1,\ldots,n\}$. For $\rho \in \mathcal{P}_n$, we denote $\rho=(a,b)$ with $a < b$. Notice that $|\mathcal{P}_n| = \frac{n(n-1)}{2} := \nu_n$. For each $\rho \in \mathcal{P}_n$, we consider $(N_t(\rho))_{t \geq 0}$ to be independent Poisson processes with rate 1. Let $N_t(\mathcal{P}_n) := \sum_{\rho \in \mathcal{P}_n} N_t(\rho)$ be a Poisson process with rate $\nu_n$.

Let $\sigma_1 < \sigma_2 < \ldots$ be the jump times of the process $(N_t(\mathcal{P}_n))$. For each $i$, denote by $\rho_i = (a_i,b_i)$ the pair such that $\sigma_i$ is a jump time of the process $(N_t(\rho_i))$.

For an integer $k \leq n$, we write $N_t(k) = \sum_{\rho} N_t(\rho)$, where the sum is taken over all the pairs in $\mathcal{P}_n$ that contain the integer $k$. Let $\tau^k_1 < \tau^k_2 < \ldots$ denote the jump times of the process $(N_t(k))$. With this notation, we have
$$N_t(\mathcal{P}_n) = \frac{1}{2} \sum_{k=1}^{n} N_t(k),$$
since each pair is counted twice in the sum on the right-hand side.

From there, for $k=1,\ldots, n$ and $i \geq 0$, we consider $B^{(k,i)} = \left(B^{(k,i)}_t\right)_{t \geq 0}$ to be i.i.d copies of a $d$-dimensional standard Brownian motion, independent of all the Poisson processes previously defined. For each $k = 1,\ldots, n$, we define the stochastic processes $X^k = (X^k_t)_{t \geq 0}$ as follows: for $0 < t \leq \tau^k_1$, let
$$X^k_t = X^k_0 + B^{(k,1)}_t.$$
For $i \geq 1$ and $\tau^k_i < t \leq \tau^k_{i+1}$, let
$$X^k_t = X^k_{\tau^k_i} + B^{(k,i+1)}_{t - \tau^k_i}.$$

We now consider $\mathbf{t} := (t_1,\ldots,t_n)$ and $\mathbf{x} := (x_1,\ldots,x_n)$. We denote by $P_{\mathbf{t},\mathbf{x}}$ the probability measure under which the processes $(X^k_t)$ will be defined for $t < 0$ by 
$$X^k_t := x_k + B^{(k,0)}_{t_k + t} \qquad (-t_k < t \leq 0).$$
Hence, under $P_{\mathbf{t}, \mathbf{x}}$, $X^k_{-t_k} = x_k$ for all $k=1,\ldots,n$. In addition, we will set $\tau^k_0 := -t_k$. We will denote by $E_{\mathbf{t}, \mathbf{x}}$ the expected value under $P_{\mathbf{t}, \mathbf{x}}$. Notice that each of the processes $(X^k_t)$ are Brownian motion themselves.

Throughout this section, we will establish the following result as a corollary of a series of intermediate lemmas.

\begin{theorem}[$n$-th moment formula] \label{thm:N_mom_formula}
	We have
\begin{align}
\no\lefteqn{E[u(t,x_1) \cdots u(t,x_n)]}\\
&= e^{t \frac{n(n-1)}{2}} E_{\mathbf{0},\mathbf{x}}\left[ \prod_{k=1}^{n} w\left(t - \tau^k_{N_t(k)}, X^{k}_{\tau^k_{N_t(k)}}\right) \prod_{i=1}^{N_t(\mathcal{P}_n)} \left((t-\sigma_i)^{2H-1} f(X^{a_i}_{\sigma_i} - X^{b_i}_{\sigma_i})\right)\right],
\end{align}
where $\mathbf{0} = (0, \ldots, 0)$.
\end{theorem}

In order to use martingale properties for the processes $I_m$ defined in Section 3.1, we extend their definition to separate the upper-bound of the integral, and the time parameter in the kernel into two different variables. Namely, let $I_0(s;t,x) := I_0(t,x)$ be constant in $s$, and define
$$I_m(s;t,x) := \int_{0}^{s} \int_{\R^d} r^{H-1/2} p_{t-r}(x-z) I_{m-1}(r,z) W_0(dr,dz)$$
(compare with \eqref{eqn:Im}). With this notation, we have that $I_m(\cdot;t,x)$ is a martingale with quadratic variation
\begin{align}
\no \lefteqn{\langle I_m(\cdot;t,x)\rangle_s}\\
\no&= \int_{0}^{s} \int_{\R^d} \int_{\R^d} r^{2H-1} p_{t-r}(x-z_1) f(z_1-z_2) p_{t-r}(x-z_2) I_{m-1}(r,z_1) I_{m-1}(r,z_2) dz_1 \, dz_2 \, dr,
\end{align}
and $I_m(t,x) = I_m(t;t,x)$.

With this notation, writing $\mathbf{m}:=(m_1,\ldots,m_n)$, and for $t < \min\{t_1,\ldots,t_n\}$, we will denote
$$I_{\mathbf{m}}(t;\mathbf{t},\mathbf{x}) := E[I_{m_1}(t,t_1,x_1)\cdots I_{m_n}(t;t_n,x_n)].$$

We are now ready to prove the following Lemma, providing an inductive integral formula for $I_{\mathbf{m}}$. This is similar to Lemma 5.2 in \cite{dalang_mueller_tribe}.
\begin{lemma} \label{lemma:I}
	If $m=0$, we have
	$$I_{\mathbf{0}}(t;\mathbf{t},\mathbf{x}) = \prod_{k=1}^{n} w(t_k,x_k).$$
	If $m > 0$, we have
	\begin{align*}
	I_{\mathbf{m}}(t;\mathbf{t},\mathbf{x}) 
	& = \sum_{\underset{m_a m_b > 0}{(a,b) \in \mathcal{P}_n}} \int_{0}^{t} ds \int_{\bR^d} dy_1 \int_{\bR^d} dy_2 \, s^{2H-1} p_{t_a-s}(x_a-y_1) f(y_1-y_2) p_{t_b-s}(x_b - y_2) \\
	& \qquad \qquad \times I_{\mathbf{m}'(a,b)}(s;\phi_{a,b}(\mathbf{t},s),\psi_{a,b}(\mathbf{x},y_1,y_2)),
	\end{align*}
	where
	\begin{itemize}
		\item $m'(a,b)_a = m_a -1$, $m'(a,b)_b = m_b - 1$, and $m'(a,b)_k = m_k$ for $k \neq a,b$;
		\item $\phi_{a,b}(\mathbf{t},s)_a = \phi_{a,b}(\mathbf{t},s)_b = s$, and $\phi_{a,b}(\mathbf{t},s)_k = t_k$ for $k \neq a,b$;
		\item $\psi_{a,b}(\mathbf{x},y_1,y_2)_a = y_1$, $\psi_{a,b}(\mathbf{x},y_1,y_2)_b = y_2$, and $\psi_{a,b}(\mathbf{x},y_1,y_2)_k = x_k$ for $k \neq a,b$.
	\end{itemize}
\end{lemma}

\begin{proof}
We let $m := \frac{1}{2}(m_1 + \cdots + m_n)$. If $m=0$, then all $m_1 = \cdots = m_n = 0$. Since $I_0(s;t,x) = w(t,x)$, the result is immediate. If $m=1$, then all but one of the $m_k$ are 0 and there is only one random integral. Since Walsh integral have zero expectation, the left-hand side vanishes. The right-hand side also vanishes, since the sum is empty: no pairs have both $m_a, m_b$ positive. If $m \geq 2$ and only one $m_k >0$, the same argument as in the case $m=1$ holds.

Finally, in order to prove the result for $m \geq 2$, when at least two $m_k > 0$, we can use It\^o formula with the function $f(x_1,\ldots,x_n) := x_1 \cdots x_n$ and the martingales $I_m(\cdot; t_k,x_k)$ ($k=1,\ldots,n$). Namely, given that
$$\frac{\partial f}{\partial x_k}(x_1,\ldots,x_n) = \prod_{\underset{j \neq k}{j=1}}^{n} x_j, \quad \frac{\partial^2 f}{\partial x_k^2}(x_1,\ldots,x_n) = 0, \text{ and } \frac{\partial^2 f}{\partial x_k x_l}(x_1,\ldots,x_n) = \prod_{\underset{j \neq k,l}{j=1}}^{n} x_j,$$
and given that Walsh integrals with square-integrable integrands have vanishing expectation, we obtain
\begin{align*}
I_{\mathbf{m}}(t;\mathbf{t},\mathbf{x}) &= E\left[I_{m_1}(t;t_1,x_1) \cdots I_{m_n}(t;t_n,x_n)\right]\\
	& = E[f(I_{m_1}(t;t_1,x_1), \ldots, I_{m_n}(t;t_n,x_n))] \\
	& = \sum_{(a,b) \in \mathcal{P}_n} E\left[ \int_{0}^{t} \left(\prod_{\underset{j \neq a,b}{j=1}}^{n} I_{m_j}(s;t_j,x_j)\right) d\langle I_{m_a}(\cdot; t_a,x_a), I_{m_b}(\cdot; t_b,x_b)\rangle_s \right] \\
	& = \sum_{\underset{m_a m_b > 0}{(a,b) \in \mathcal{P}_n}} \int_{0}^{t} ds \int_{\R^d} dy_1 \int_{\R^d} dy_2 \, s^{2H-1} p_{t_a-s}(x_a - y_1) f(y_1-y_2) p_{t_b-s}(x_b - y_2) \\
	&\qquad \times E\left[ \left(\prod_{\underset{j \neq a,b}{j=1}}^{n} I_{m_j}(r;t_j,x_j)\right) I_{m_a - 1}(s;s,y_1), I_{m_b - 1}(s;s,y_2) \right]. \\
\end{align*}
Notice that the double sum applies only to terms for which $m_a, m_b > 0$, since $I_0$ would have vanishing quadratic variation. The last expectation on the left-hand side is equal to $I_{\mathbf{m}'(a,b)}(s;\phi_{a,b}(\mathbf{t},s),\psi_{a,b}(\mathbf{x},y_1,y_2))$. This proves the result.
\end{proof}

Now, in order to establish the main result, we will look at the type of expectation present on the right-hand side of Theorem \ref{thm:N_mom_formula}. For $\mathbf{m}$, $\mathbf{t}$ and $\mathbf{x}$ as above, we define
\begin{align}
\label{eqn:defJ} 
\no J_{\mathbf{m}}(t;\mathbf{t},\mathbf{x}) 
& := e^{\nu_n t} E_{\mathbf{t},\mathbf{x}}\left[\prod_{k=1}^{n} \ind_{\{N_t(k) = m_k\}} \prod_{k=1}^{n} w(t - \tau^k_{m_k}, X^{k}_{\tau^k_{m_k}})\right.\\
&\left. \phantom{e^{\nu_n t}}\times\prod_{i=1}^{\frac{1}{2}(m_1 + \cdots + m_n)} \left((t-\sigma_i)^{2H-1} f(X^{a_i}_{\sigma_i} - X^{b_i}_{\sigma_i})\right)\right]. 
\end{align}

The terms $J_{\mathbf{m}}$ defined above also satisfy an inductive formula. This result is similar to Lemma 5.3 in \cite{dalang_mueller_tribe}.

\begin{lemma} \label{lemma:J}
	If $m=0$, we have
	$$J_{\mathbf{0}}(t;\mathbf{t},\mathbf{x}) = \prod_{k=1}^{n} w(t+t_k,x_k).$$
	If $m > 0$, we have
	\begin{align*}
	\lefteqn{J_{\mathbf{m}}(t;\mathbf{t},\mathbf{x})}\\
	& = \sum_{\underset{m_a m_b > 0}{(a,b) \in \mathcal{P}_n}} \int_{0}^{t} ds \int_{\bR^d} dy_1 \int_{\bR^d} dy_2 \, (t-s)^{2H-1} p_{t_a+s}(x_a-y_1) f(y_1-y_2) p_{t_b+s}(x_b - y_2) \\
	& \qquad \qquad \times J_{\mathbf{m}'(a,b)}(t-s;\gamma_{a,b}(\mathbf{t},s),\psi_{a,b}(\mathbf{x},y_1,y_2)),
	\end{align*}
	where $\mathbf{m}'(a,b)$ and $\psi_{a,b}$ are defined as in Lemma \ref{lemma:I} and
	\begin{itemize}
		\item $\gamma_{a,b}(\mathbf{t},s)_a = \gamma_{a,b}(\mathbf{t},s)_b = 0$, and $\gamma_{a,b}(\mathbf{t},s)_k = t_k + s$ for $k \neq a,b$;
	\end{itemize}
\end{lemma}

\begin{proof}
We let $m := \frac{1}{2}(m_1 + \cdots + m_n)$. If $m=0$, then all $m_1 = \cdots = m_n = 0$. Since $\tau^k_0 = -t_k$, and $X^k_{-t_k} = x_k$, the result follows given that $P(N_t(k) = 0 \; \forall k) = P(N_t(\mathcal{P}_n) = 0) = e^{-\nu_n t}$. 

If $m=1$, then all but one of the $m_k$ are 0. The left-hand side vanishes since it is impossible for only one process $N_t(k)$ to be non-zero. Remember that the process is first defined as $N_t(\mathcal{P}_n)$ and that the processes jump in pairs. The right-hand side also vanishes, since the sum is empty: no pairs have both $m_a, m_b$ positive. If $m \geq 2$ and only one $m_k > 0$, the same argument as in the case $m=1$ holds.

Now, if $m \geq 2$ and there is at least one pair for which $m_a m_b > 0$, we will condition on the first jump time. Since there is at least one pair jumping, we must have $0 < \sigma_1 < t$. Namely, let $\mathcal{F}_1 := \sigma(\sigma_1,\rho_1,X^{a_1}_{\sigma_1},X^{a_2}_{\sigma_2})$. Then, we have
\begin{align*}
\lefteqn{J_{\mathbf{m}}(t;\mathbf{t},\mathbf{x})} \\
& = \sum_{\underset{m_a m_b > 0}{(a,b) \in \mathcal{P}_n}} e^{\nu_n t} E_{\mathbf{t}, \mathbf{x}} \left[ \ind_{\{\sigma_1 \leq t, \rho_1=(a,b)\}} (t-\sigma_1)^{2H-1} f(X^{a_1}_{\sigma_1} - X^{b_1}_{\sigma_1})\phantom{\prod_{k=a,b}}\right.\\
&\times E_{\mathbf{t},\mathbf{x}} \left[\prod_{k=a,b} \ind_{\{N_t(k) - N_{\sigma_1}(k) = m_k-1\}}  \prod_{\underset{k\neq a,b}{k=1}}^{n} \ind_{\{N_t(k) - N_{\sigma_1}(k) = m_k\}}\right. \\
& \left. \left.\times \prod_{k=1}^{n} w(t - \tau^k_{m_k}, X^{k}_{\tau^k_{m_k}})
\prod_{i=2}^{m/2} \left((t-\sigma_i)^{2H-1} f(X^{a_i}_{\sigma_i} - X^{b_i}_{\sigma_i})\right) \bigg\vert \mathcal{F}_1 \right] \right].
\end{align*}
Now, if we consider the conditional expectation above, we notice that at time $\sigma_1$, and since $\rho_1 = (a,b)$, the processes $X^a$ and $X^b$ restart independently of the past. Furthermore, the other processes $X^k$ have not had a jump since time $-t_k$. Hence, given that there is a duration $t-\sigma_1$ left until the end of the time interval, the Markov property at time $\sigma_1$ shows that the conditional expectation is equal to $e^{-\nu_n(t-\sigma_1)} J_{\mathbf{m}'(a,b)}(t-\sigma_1;\gamma_{a,b}(\mathbf{t},\sigma_1),\psi_{a,b}(\mathbf{x},X^a_{\sigma_1},X^b_{\sigma_1}))$, since for indices $a,b$ the process restarts; and for $k \neq a,b$, the process hasn't seen a jump for $t_k+\sigma_1$ units of time. Altogether, given that under $P_{\mathbf{t},\mathbf{x}}$, the random variable $X^j_s$ has density $p_{s+t_j}(x_j - \cdot)$, we obtain
 \begin{align*}
 	\lefteqn{J_{\mathbf{m}}(t;\mathbf{t},\mathbf{x})} \\
 	& = \sum_{\underset{m_a m_b > 0}{(a,b) \in \mathcal{P}_n}} e^{\nu_n t} E_{\mathbf{t}, \mathbf{x}} \left[ \ind_{\{\sigma_1 \leq t, \rho_1=(a,b)\}} (t-\sigma_1)^{2H-1} f(X^{a_1}_{\sigma_1} - X^{b_1}_{\sigma_1}) e^{-\nu_n (t-\sigma_1)} \right. \\
 	& \qquad \qquad \times \left. J_{\mathbf{m}'(a,b)}(t-\sigma_1;\gamma_{a,b}(\mathbf{t},\sigma_1),\psi_{a,b}(\mathbf{x},X^a_{\sigma_1},X^b_{\sigma_1})) \right] \\
 	& = \sum_{\underset{m_a m_b > 0}{(a,b) \in \mathcal{P}_n}} E\left[ \ind_{\{\sigma_1 \leq t, \rho_1=(a,b)\}} (t-\sigma_1)^{2H-1} e^{\nu_n \sigma_1} \int_{\bR^d} dy_1 \, p_{t_a + \sigma_1}(x_a - y_1) \right. \\
 	& \qquad \left. \times \int_{\bR^d} dy_2 \, p_{t_b + \sigma_1}(x_b - y_2) f(y_1 - y_2) J_{\mathbf{m}'(a,b)}(t-\sigma_1;\gamma_{a,b}(\mathbf{t},\sigma_1),\psi_{a,b}(\mathbf{x},y_1,y_2)) \right].
 \end{align*}
Now, given that $\sigma_1$ is distributed exponentially with parameter $\nu_n$ and that $P(\rho_1 = (a,b)) = \frac{1}{\nu_n}$, we have that for any function $g$,
$$E[\ind_{\{\sigma_1 \leq t, \rho_1=(a,b)\}} g(\sigma_1)] = \frac{1}{\nu_n} \int_{0}^{\infty} \nu_n e^{-nu_n s} \ind_{[0,t]}(s) g(s) \, ds = \int_{0}^{t} e^{-\nu_n s} g(s) \, ds.$$
Using this fact, we obtain
  \begin{align*}
 	\lefteqn{J_{\mathbf{m}}(t;\mathbf{t},\mathbf{x})} \\
 	& = \sum_{\underset{m_a m_b > 0}{(a,b) \in \mathcal{P}_n}} \int_{0}^{t} ds \, e^{-\nu_n s}  (t-s)^{2H-1} e^{\nu_n s} \int_{\bR^d} dy_1 \, p_{t_a + s}(x_a - y_1) \int_{\bR^d} dy_2 \, p_{t_b + s}(x_b - y_2)  \\
 	& \qquad \qquad \times f(y_1 - y_2)J_{\mathbf{m}'(a,b)}(t-s;\gamma_{a,b}(\mathbf{t},s),\psi_{a,b}(\mathbf{x},y_1,y_2)) \\
 	& = \sum_{\underset{m_a m_b > 0}{(a,b) \in \mathcal{P}_n}} \int_{0}^{t} ds \,  (t-s)^{2H-1} \int_{\bR^d} dy_1 \, p_{t_a + s}(x_a - y_1) \int_{\bR^d} dy_2 \, p_{t_b + s}(x_b - y_2)  \\
 	& \qquad \qquad \times f(y_1 - y_2)J_{\mathbf{m}'(a,b)}(t-s;\gamma_{a,b}(\mathbf{t},s),\psi_{a,b}(\mathbf{x},y_1,y_2))
 \end{align*}
establishing Lemma \ref{lemma:J}.
\end{proof}

We can now establish the connection between the terms $I_{\mathbf{m}}$ and $J_{\mathbf{m}}$.

\begin{lemma} \label{lemma:IJ}
	For all $\mathbf{m}$, $\mathbf{t}$, $\mathbf{x}$, and $t \leq \min\{t_1, \ldots, t_n\}$ as above, we have
	$$I_{\mathbf{m}}(t;\mathbf{t},\mathbf{x}) = J_{\mathbf{m}}(t;\mathbf{t}-t,\mathbf{x}),$$
	where the term $\mathbf{t} - t$ stands for the vector $(t_1 - t, \ldots, t_n - t)$.
\end{lemma}

\begin{proof}
	We prove the result by induction on $m := m_1 + \cdots m_n$. For $m=0$, the result is true by the results of Lemmas \ref{lemma:I} and \ref{lemma:J}. Indeed,
	$$J_{\mathbf{0}}(t;\mathbf{t}-t,\mathbf{x}) = \prod_{k=1}^{n} w(t+(t_k-t),x_k) = \prod_{k=1}^{n} w(t+(t_k-t),x_k) = I_{\mathbf{m}}(t;\mathbf{t},\mathbf{x}).$$
	For $m=1$, the result is true since both terms vanish (see Lemmas \ref{lemma:I} and \ref{lemma:J}).
	
	Now, we assume that the result is proved up to $m-1$ and consider the case $m$. We assume that there exists a pair such that $m_a m_b > 0$, otherwise both term vanish and the result is direct. Lemma \ref{lemma:I} and the induction assumption (since $\sum m'(a,b)_k = \sum m_k - 2$) yield
	\begin{align*}
		\lefteqn{I_{\mathbf{m}}(t;\mathbf{t},\mathbf{x})} \\
		& = \sum_{\underset{m_a m_b > 0}{(a,b) \in \mathcal{P}_n}} \int_{0}^{t} ds \int_{\bR^d} dy_1 \int_{\bR^d} dy_2 \, s^{2H-1} p_{t_a-s}(x_a-y_1) f(y_1-y_2) p_{t_b-s}(x_b - y_2) \\
		& \qquad \qquad \times I_{\mathbf{m}'(a,b)}(s;\phi_{a,b}(\mathbf{t},s),\psi_{a,b}(\mathbf{x},y_1,y_2)) \\
		& = \sum_{\underset{m_a m_b > 0}{(a,b) \in \mathcal{P}_n}} \int_{0}^{t} ds \int_{\bR^d} dy_1 \int_{\bR^d} dy_2 \, s^{2H-1} p_{t_a-s}(x_a-y_1) f(y_1-y_2) p_{t_b-s}(x_b - y_2) \\
		& \qquad \qquad \times J_{\mathbf{m}'(a,b)}(s;\phi_{a,b}(\mathbf{t},s)-s,\psi_{a,b}(\mathbf{x},y_1,y_2)).
	\end{align*}
	Using the change of variables $r = t-s$, we obtain
	\begin{align*}
		\lefteqn{I_{\mathbf{m}}(t;\mathbf{t},\mathbf{x})} \\
		& = \sum_{\underset{m_a m_b > 0}{(a,b) \in \mathcal{P}_n}} \int_{0}^{t} dr \int_{\bR^d} dy_1 \int_{\bR^d} dy_2 \, (t-r)^{2H-1} p_{t_a-(t-r)}(x_a-y_1) f(y_1-y_2) p_{t_b-(t-r)}(x_b - y_2) \\
		& \qquad \qquad \times J_{\mathbf{m}'(a,b)}(t-r;\phi_{a,b}(\mathbf{t},t-r)-(t-r),\psi_{a,b}(\mathbf{x},y_1,y_2)) \\ \\
		& = \sum_{\underset{m_a m_b > 0}{(a,b) \in \mathcal{P}_n}} \int_{0}^{t} dr \int_{\bR^d} dy_1 \int_{\bR^d} dy_2 \, (t-r)^{2H-1} p_{(t_a-t)+r}(x_a-y_1) f(y_1-y_2) p_{(t_b-t)+r}(x_b - y_2) \\
		& \qquad \qquad \times J_{\mathbf{m}'(a,b)}(t-r;\gamma_{a,b}(\mathbf{t}-t,r),\psi_{a,b}(\mathbf{x},y_1,y_2)) \\ \\
		& = J_{\mathbf{m}}(t;\mathbf{t}-t,\mathbf{x}),
	\end{align*}
	by the result of Lemma \ref{lemma:J}. Here, we have used the fact that $\phi_{a,b}(\mathbf{t},t-r) = \gamma_{a,b}(\mathbf{t}-t,r)$. Indeed, for indices $a,b$, we have
	$$\phi_{a,b}(\mathbf{t},t-r)_a-(t-r) = t-r - (t-r) = 0 = \gamma_{a,b}(\mathbf{t}-t,r)_a,$$
	$$\phi_{a,b}(\mathbf{t},t-r)_b-(t-r) = t-r - (t-r) = 0 = \gamma_{a,b}(\mathbf{t}-t,r)_b.$$
	Moreover, for an index $k \neq a,b$, we have
	$$\phi_{a,b}(\mathbf{t},t-r)_k-(t-r) = t_k - (t-r) = (t_k-t) + r = \gamma_{a,b}(\mathbf{t}-t,r)_k.$$
	This proves the result.
\end{proof}	

We are now ready to prove Theorem \ref{thm:N_mom_formula}.
\vskip 12pt

\begin{proof}[Proof of Theorem \ref{thm:N_mom_formula}]
	First notice that for $n=2$, the result is directly Theorem \ref{thm:2nd_mom_formula}. First of all, considering the representation \eqref{eqn:series}, it is easy to see that
	\begin{align*}
	E[u(t,x_1)\ldots u(t,x_n)] & = \sum_{m_1=0}^{\infty} \cdots \sum_{m_n = 0}^{\infty} E[I_{m_1}(t,x_1) \ldots I_{m_n}(t,x_n)] \\
	& = \sum_{m_1=0}^{\infty} \cdots \sum_{m_n = 0}^{\infty} E[I_{m_1}(t;t,x_1) \ldots I_{m_n}(t;t,x_n)] \\
	& = \sum_{m_1=0}^{\infty} \cdots \sum_{m_n = 0}^{\infty} I_{\mathbf{m}}(t;(t,\ldots,t),\mathbf{x}).
	\end{align*}

	Using the result of Lemma \ref{lemma:IJ}, we know that $$I_{\mathbf{m}}(t;(t,\ldots,t),\mathbf{x}) = J_{\mathbf{m}}(t;(0,\ldots,0),\mathbf{x}).$$ 
	Using the definition \eqref{eqn:defJ} of $J_{\mathbf{m}}$, we obtain
	\begin{align*}
		\lefteqn{E[u(t,x_1)\ldots u(t,x_n)] }\\
		& = \sum_{m_1=0}^{\infty} \cdots \sum_{m_n = 0}^{\infty} J_{\mathbf{m}}(t;(0,\ldots,0),\mathbf{x}) \\
		& = e^{\nu_n t} \sum_{m_1=0}^{\infty} \cdots \sum_{m_n = 0}^{\infty} E_{\mathbf{0},\mathbf{x}}\left[\prod_{k=1}^{n} \ind_{\{N_t(k) = m_k\}} \prod_{k=1}^{n} w(t - \tau^k_{m_k}, X^{k}_{\tau^k_{m_k}}) \right. \\
		& \left. \qquad \times \prod_{i=1}^{\frac{1}{2}(m_1 + \cdots + m_n)} \left((t-\sigma_i)^{2H-1} f(X^{a_i}_{\sigma_i} - X^{b_i}_{\sigma_i})\right)\right] \\
		& = e^{\nu_n t} E_{\mathbf{0},\mathbf{x}}\left[\prod_{k=1}^{n} w(t - \tau^k_{N_t(k)}, X^{k}_{\tau^k_{N_t(k)}}) \prod_{i=1}^{N_t(\mathcal{P}_n)} \left((t-\sigma_i)^{2H-1} f(X^{a_i}_{\sigma_i} - X^{b_i}_{\sigma_i})\right)\right], 
	\end{align*}
	where we have identified $N_t(k) = m_k$ and $N_t(\mathcal{P}_n) = \frac{1}{2}(m_1+\cdots+m_n)$. The result is proved.
\end{proof}
\section{Wave equation formula}

In this section, we will establish a similar representation for the solution to the wave equation, rather than the heat equation. First notice that we cannot directly rely on the fact that $\Delta$ is the generator of Brownian motion, since that would not directly relate to the wave kernel. Indeed, there is no Markov process for which the wave kernel would directly be the probability density function. However, even though this held for the heat equation, it was not necessary. The fact that the wave kernel at a fixed time $t$ can be written as a probability density function is sufficient for the representation formula.

Consider a Poisson process $(N_t)_{t \geq 0}$ with rate 1, and let $(\tau_n)_{n \geq 0}$ denote its jump times. Also, consider $\beta^{(i)} = \left(\beta^{(i)}_t\right)_{t \geq 0}$ to be i.i.d copies of a stochastic process for which the densty at each time $t$ is proportional to $p_w(t,\cdot)$. (Note that $(\beta_t)$ does not have to be a Markov process.) We define a stochastic process $X = (X_t)_{t \geq 0}$ as follows: for $0 < t \leq \tau_1$, let
$$X_t = X_0 + \beta^{(1)}_t.$$
For $i \geq 1$ and $\tau_i < t \leq \tau_{i+1}$, let
$$X_t = X_{\tau_i} + \beta^{(i+1)}_{t - \tau_i}.$$
The process above is an equivalent to equation~\eqref{processformation} where this time the density of the process $(X_t)$ is given by $p_w(t,\cdot)$ (up to a renormalization constant).

For instance, in dimension $d=1$, and following \eqref{wave kernel}, we can choose the process $\beta_t = t U$, where $U$ is a uniform distribution on $[-1,1]$. In dimension $d=3$, we can choose $\beta_t = t U$, where $U$ is a uniformly distributed random variable on $\partial B(0,1)$. (See for instance \cite{dalang_mueller_tribe} for more examples of the possible choices of the process $(\beta_t)$.)

\begin{theorem}[Second moment formula]
 \label{wavethm:2nd_mom_formula}
 We have
\begin{align}
\no E[u(t,x)u(t,y)] &= e^t E_{x,y}\left[w(t - \tau_{N_t}, X^{1}_{\tau_{N_t}}) w(t - \tau_{N_t}, X^{2}_{\tau_{N_t}}) \phantom{\prod_{i=1}^{N_t}}\right.\\
&\phantom{e^t E_{x,y}} \times \left.\prod_{i=1}^{N_t} \left((t-\tau_i)^{2H-1}(\tau_i-\tau_{i-1})^2 f(X^1_{\tau_i} - X^2_{\tau_i})\right)\right],
\end{align}
where $X^1$ and $X^2$ are two i.i.d copies of the process $X$ defined above.
\end{theorem}

\begin{proof}
The proof of the wave equation second moment formula is similar to the heat equation proof, using the newly defined process $(X_t)$, rather than the one built on Brownian motion. 
The extra term $(\tau_i-\tau_{i-1})$ in the formula comes from the fact that $\displaystyle \int_{\bR^d} p_w(t,x)dx\neq 1$. Hence, the wave kernel needs to be renormalized to be represented as a probability density. In order to have the Poisson type expectation representation as in \eqref{heatexpectation}, we must use
\begin{align*}
   & \int_{\bR^d}dz_1\int_{\bR^d}dz_2 \,p_w(\tau_1,x-z_1)f(z_1-z_2) p_w(\tau_1,y-z_2)\\
    =&\int_{\bR^d}dz_1\int_{\bR^d}dz_2 \,\frac{p_w(\tau_1,x-z_1)}{\tau_1}f(z_1-z_2) \frac{p_w(\tau_1,y-z_2)}{\tau_1}\tau_1^2 \\
    =&E[\tau_1^2 f(X_{\tau_1}^1-X_{\tau_1}^2)].
\end{align*}
The rest of the proof follows the steps of the proof of Theorem \ref{thm:2nd_mom_formula} with the adjustment above for renormalization.
\end{proof}

\begin{remark}
Theorem \ref{wavethm:2nd_mom_formula} holds true for dimensions $d \leq 3$, see equation \eqref{waveforalldim}. 
\end{remark}

Extending Theorem \ref{wavethm:2nd_mom_formula} to the $n$-th moment is done in a similar way as for the heat equation. The changes are:
\begin{enumerate}
\item The processes $X^k$ are the ones corresponding to the wave kernel~\eqref{wave kernel} defined earlier.
\item The terms $\tau_{i+1}-\tau_{i}$ appear in the formula similarly as in Theorem \ref{wavethm:2nd_mom_formula}.
\end{enumerate}

The formula is provided below. For conciseness, the details of the proof are omitted.

\begin{theorem}[$n$-th moment formula] \label{thmwave:N_mom_formula}
We have
\begin{align*}
\lefteqn{E[u(t,x_1) \cdots u(t,x_n)] }\\
&= e^{t \frac{n(n-1)}{2}} E_{\mathbf{0},\mathbf{x}}\left[ \prod_{k=1}^{n} w\left(t - \tau^k_{N_t(k)}, X^{k}_{\tau^k_{N_t(k)}}\right)\right.\\
&\phantom{= e^{t \frac{n(n-1)}{2}} E_{\mathbf{0},\mathbf{x}}} 
\times \left.\prod_{i=1}^{N_t(\mathcal{P}_n)} \left((t-\sigma_i)^{2H-1}(\tau_i^{a_i}-\tau_{i-1}^{a_i})(\tau_i^{b_i}-\tau_{i-1}^{b_i}) f(X^{a_i}_{\sigma_i} - X^{b_i}_{\sigma_i})\right)\right],
\end{align*}
where $\mathbf{0} = (0, \ldots, 0)$.
\end{theorem}

\chapter{Intermittency result for the DO noise case}
\label{chap:intermittency}
 In this section, we are going to obtain bounds on the Lyapunov exponent for equation~\eqref{eqn:DO process} and \eqref{eqnwave:DO process}, similarly to the one obtained in \cite{Foondun2008} and \cite{Mueller2009} in the case of white noise in time. This extension follows the approach of Balan and Conus\cite{Balan2016} for the parabolic case. 
 
 The bounds on the Lyapunov exponent obtained in this chapter will match exactly the bounds obtained for the equation driven by standard fractional noise. We remind that the DO noise shares the same asymptotic variance as the standard fractional noise, but their covariance structure is different (see Remark \ref{covariancedoesnotmatter}). In Chapter 5 and beyond, we will present a first result towards establishing that the specific covariance structure of the noise is not relevant for intermittency, but only the asymptotic order of the variance. 

The upper bound result is proved in Section 4.2, the lower bound result is proved in Section 4.3. The notation follows the one from Chapter 3.
First, we start with a reminder about some special functions that arise in our upcoming proofs.


The results on intermittency below will be established in the case where $u_0$ and $v_0$ (for the wave equation) are constant. We are convinced that these results will hold for any initial conditions such that $c \leq u_0 \leq C$ and (for the wave equation) $|v_0| \leq C'$, where $c,C,C'$ are universal positive constants. 

\section{Preliminaries}

The fractional nature of the noise will lead us to having to deal with some results from fractional Calculus and special functions of fractional type. We present a few useful reminders here.

\begin{definition}[Fractional Integral]
\label{fractionalintegral}
The fractional integral, also known as the Riemann–Liouville integral, associates with a real function $ f:\mathbb {R} \rightarrow \mathbb{R}$  another function $I^{\alpha} f$ of the same kind for each value of the parameter $\alpha>0$
\begin{align}
    I^{\alpha}f(t)=\frac{1}{\Gamma(\alpha)} \int_0^t (t-\tau)^{\alpha-1}f(\tau)d\tau.
\end{align}
\end{definition}

The first of the special function that we will need is the incomplete Gamma function. This function arises as a special case of the Miller-Ross function, which arises as a fractional integral of the exponential (see \cite{incompletegamma}).

\begin{definition}[Incomplete Gamma function]
\label{defincomplete}
The incomplete gamma function is defined as:
\begin{align}
    \gamma^{*}(\nu,z):=e^{-z}\sum_{m=0}^\infty \frac{z^m}{\Gamma(\nu+m+1)},
\end{align}
for all $z\in \bR$ and $\nu>0$.
\end{definition}

The Beta function, also called the Euler integral of the first kind, is a special function that is closely related to the gamma function and to binomial coefficients.

\begin{definition}[Beta function]
 The Beta function is defined by the integral:
\begin{align}
\no    B(x,y)&:=\int_0^1 t^{x-1} (1-t)^{y-1}\,dt\\
    &=\frac{\Gamma (x)\Gamma(y)}{\Gamma(x+y)},
\end{align}
where $x>0$, $y>0$.
\end{definition}

Once the incomplete Gamma function arises in the results below, we will need to understand its asymptotic behavior as $t \rightarrow \infty$. This result follows from Stirling's formula.

\begin{lemma}[Asymptotic behavior of the incomplete gamma function]
\label{incompletegamma}
Let $\gamma*$ denote the incomplete Gamma function. Then, we have
$$\lim_{t\rightarrow \infty} t^{\nu}\gamma^{*}(\nu,t)=1$$
for all $\nu>0$.
\end{lemma}
For a proof, see \cite[Appendix B]{incompletegamma}.

We are now ready to proceed to the main results regarding the Lyapunov exponent.

\section{Upper bound for the heat equation}
We will first consider that the spatial covariance function satisfies the following assumption. This is case (1) introduced earlier also referred to as the smooth noise case.

\begin{assumption}
\label{assumption:bounded}
 The spatial covariance function $f$ of the noise is bounded (hence uniformly continuous and attains its maximum at 0). We let $A_0 := f(0)$.
\end{assumption}
\begin{theorem}[Second moment upper bound with smooth noise]
\label{thm:2nd_mom_formula upper bound}
We assume that Assumptions \ref{assumption:finite var} and \ref{assumption:bounded} hold. Let $u$ be the solution to the stochastic heat equation~\eqref{eqn:DO process} with $u_0$ constant. There is a positive constant $C < \infty$ such that for all $t\geq 0$ and $x,y\in \mathbb{R}^d$,
\begin{align}
    E[u(t,x)u(t,y)]\leq Cu_0^2 \exp(2A_0 t^{2H}).
\end{align}
\end{theorem}
This is the easiest case since the covariance function $f$ is bounded.
\begin{proof}
By Theorem \ref{thm:2nd_mom_formula} we have\\
\begin{align}
\label{second moment upper bound}
\displaystyle{|E[u(t,x)u(t,y)]|\leq u_0^2 e^th(t)},
\end{align}

where\\
\begin{align}
\displaystyle{h(t)=E_{x,y}\Bigg{[}\prod_{i=1}^{N(t)}A_0  (t-\tau_i)^{2H-1}\Bigg{]}}.
\end{align}

Using the strong Markov property at the first jump time $\tau_1$ of $N(t)$, letting $\mathcal{F}_1=\sigma(\tau_1,X_{\tau_1}^1,X_{\tau_1}^2)$ and using that the product is 1 if $N_t=0$, we see that
\begin{align}
 \no  \lefteqn{ h(t)}\\
   &=E[\mathbbm{1}_{\{N(t)=0\}}]+E\Bigg{[}\mathbbm{1}_{\{N(t)>0\}}(A_0  (t-\tau_1)^{2H-1})E\Bigg[\prod_{i=2}^{N(t)}A_0  (t-\tau_i)^{2H-1}|\mathcal{F}_1\Bigg]\Bigg{]}\no\\
    &=e^{-t}+A_0 \int_0^t (t-s)^{2H-1}h(t-s) e^{-s}\,ds.
\end{align}
Next using the substitution $s=t-s$, we obtain
\begin{align}
    h(t)=e^{-t}\Bigg(1+A_0 \int_0^t s^{2H-1}h(s)e^s\,ds\Bigg).
\end{align}
Notice that this is a renewal equation \cite[Section 4.4]{durrett2010}. Taking the derivative with respect to $t$ on both sides yields
\begin{align}
    h'(t)&=-h(t)+e^{-t}(A_0  t^{2H-1}h(t)e^t)\no\\
    &=(A_0  t^{2H-1}-1)h(t).
\end{align}
This is a separable differential equation, and thus
\begin{align}
\displaystyle{h(t)=C_1 \exp\Bigg(\frac{A_0  t^{2H}}{2H}-t\Bigg)}.
\end{align}
From \eqref{second moment upper bound} we have \\
\begin{align}
  \no  |E[u(t,x)u(t,y)]|&\leq C_1 u_0^2 e^t \exp\Bigg(\frac{A_0  t^{2H}}{2H}-t\Bigg)\\
    &= Cu_0^2 e^{2A_0  t^{2H}},
\end{align}
 where $C_1, C$ are universal constants.
\end{proof}

Theorem \ref{thm:2nd_mom_formula upper bound} has established that the second order Lyapunov exponent under DO noise corresponds to the Lyapunov exponent for the equation driven by fractional Brownian motion in the case of spatially smooth noise. Now we will investigate higher moments as well as more general correlation functions $f$ (cases (2) and (3)).

\begin{theorem}[$n$-th moment upper bound]
Let $u$ be the solution to equation \eqref{eqn:DO process} driven by the noise $W$ defined in \eqref{Do definition eqn}, where $W_0$ is any of the noises of cases (1),(2),(3) in Chapter 2. Then there exists a constant $C$ such that
\label{thm:upper bound}
\begin{align}
    E[u^n(t,x)]\leq C u_0^n \cdot \exp(C  n^{\frac{4-a}{2-a}} t^{\frac{4H-a}{2-a}}).
\end{align}
\end{theorem}

\begin{remark}
Here notice that the Riesz kernel does not satisfy Assumption \ref{assumption:bounded}, so we will use another method to prove case (2). The proof follows the ideas introduced by Foondun and Khoshnevisan \cite{Foondun2008}.
\end{remark}

Now we prove Theorem \ref{thm:upper bound} in case (1). In this case, the spatial covariance function is bounded and can be controlled directly. The core of the proof comes down to controlling the temporal terms resulting from the fractional noise.

\begin{proof}[Proof of case (1)]
By Theorem \ref{thm:N_mom_formula} and Assumption \ref{assumption:bounded}, let $x_1=x_2=...=x_n=x$. Given the constant initial condition, we have $w(t,x_i)=u_0$ for all $i$. Since $u(t,x) \geq 0 $ for all $t>0, x\in \bR^d$, we have
\begin{align}
\label{upper bound 1}
   E[u^n(t,x)]\leq u_0^n e^{t n(n-1)/2} E_{(0,x),...,(0,x)}[A_0^{N_t(\mathcal{P}_n)}Z_{n,t}],
\end{align}
where $\displaystyle{Z_{n,t}=\prod_{\ell=1}^{n} \prod_{i=1}^{N_t(\ell)}(t-\tau_i^\ell)^{H-1/2}}$.\\
By Lemma \ref{lemma:upper bound}, now we have
\begin{align}
\no E[u^n(t,x)] &\leq e^{tn(n-1)/2} u_0^nE_{(0,x),...,(0,x)}[A_0^{N_t(\mathcal{P}_n)}Z_{n,t}]\\
&\leq  u_0^n \exp(A_0  t^{2H}n^2),
\end{align}
where the second inequality comes from Lemma~\ref{lemma:upper bound} below. This proves the result in case (1).
\end{proof}
\begin{lemma}
 \label{lemma:upper bound}
There is a universal constant $C<\infty $ such that for all $n\geq 2$, $t\geq 0$ and $x\in \mathbb{R}^d$,
\begin{align}
    e^{tn(n-1)/2}E_{(0,x),...,(0,x)}[A_0^{N_t(\mathcal{P}_n)}Z_{n,t}]\leq C\exp(A_0 t^{2H}n^2).
\end{align}
\end{lemma}
\begin{proof}
The proof is inspired by \cite{Mueller2009}. Let $\displaystyle{\nu_n=\frac{n(n-1)}{2}}$, the number of the ordered pairs. Using the notation defined in Chapter 3, we have
\begin{equation}
\sum_{\rho\in \mathcal{P}_n}N_t(\rho)=N_t(\mathcal{P}_n)=\frac{1}{2}\sum_{\ell=1}^nN_t(\ell).  
\end{equation}
Here $N_t(\mathcal{P}_n)$ is a Poisson random variable with parameter $t\nu_n$. For $k \in \mathbb{N}$, given that $N_t(\mathcal{P}_n) = k$, the number of
factors in the product that deﬁnes $Z_{n,t}$ is $2k$. Temporarily we focus on $E_{(0,x),...,(0,x)}[A_0^{N_t(\mathcal{P}_n)}Z_{n,t}|N_t(\mathcal{P}_n)=k]$. By the arithmetic-geometric inequality, i.e.,for positive numbers $a_1,...,a_k$,
\begin{align}
\label{arithmeticmean}
\displaystyle{\left(\prod_{\ell=1}^ka_\ell \right)^{1/k}\leq \frac{\sum_{\ell=1}^k a_\ell}{k}}.
\end{align}
Then we have
\begin{align}
\label{eqn:lemma upper bound}
 \no\lefteqn{E_{(0,x),...,(0,x)}[A_0^{N_t(\mathcal{P}_n)}Z_{n,t}|N_t(\mathcal{P}_n)=k]}\\
 &\leq \Bigg{(}\frac{1}{2k}\sum_{\ell=1}^n\sum_{i=1}^{N_t(\ell)}(A_0^{1/2}(t-\tau_{i}^\ell)^{H-1/2})\Bigg{)}^{2k}.
\end{align}
Now we focus on the sum
\begin{align}
    \sum_{\ell=1}^n\sum_{i=1}^{N_t(\ell)}(A_0^{1/2}(t-\tau_{i}^\ell)^{H-1/2})=2\cdot A_0^{ N_t(\mathcal{P}_n)} \sum_{i=1}^{N_t(P_n)}(t-\tau_i)^{H-1/2}.
\end{align}
Since $\tau_i$ can be seen as the $i$-th ordered statistics of a sample of $N_t(\mathcal{P}_n)$ uniform random variables on $[0,t]$, we have\\
\begin{align}
\no\sum_{i=1}^{N_t(P_n)}(t-\tau_i)^{H-1/2}&=t^{H-1/2}\sum_{i=1}^{N_t(P_n)}(1-U_{(i)})^{H-1/2}\\
\label{uniform dist}
&\overset{\text{law}}{=} t^{H-1/2} \sum_{i=1}^{N_t(P_n)}U_{(i)}^{H-1/2},
\end{align}
where $U_{(i)}$ is the $i$-th order statistics of a sample of $N_t(\mathcal{P}_n)$ uniform random variables on $[0,1]$. Then after we take the expectation of the uniform distribution directly, \eqref{uniform dist} yields
\begin{align}
t^{H-1/2}\int_0^1\int_0^1...\int_0^1\sum_{i=1}^k s_i^{H-1/2}ds_1...ds_{k}=t^{H-1/2}\frac{1}{H+\frac{1}{2}}\cdot k.
\end{align}
Now we get back to equation \eqref{eqn:lemma upper bound}, we know that
\begin{align}
\label{Z_{n,t}}
\no\lefteqn{E_{(0,x)...(0,x)}[A_0^{N_t(\mathcal{P}_n)}Z_{n,t}]}\\
&=\sum_{k=0}^\infty E_{(0,x),...,(0,x)}[A_0^{N_t(\mathcal{P}_n)}Z_{n,t}|N_t(\mathcal{P}_n)=k]\cdot P(N_t(\mathcal{P}_n)=k).
\end{align}
By \eqref{Z_{n,t}} and the probability density function of the Poisson process $N_t(\mathcal{P}_n)$,
\begin{align}
    \no \lefteqn{E_{(0,x),...,(0,x)}[A_0^{N_t(\mathcal{P}_n)}Z_{n,t}]}\\
\no    &\leq e^{-\nu_n t}+ \sum_{k=1}^\infty\left(\frac{1}{2k}\cdot 2\cdot A_0^{1/2} \frac{k}{H+\frac{1}{2}}t^{H-1/2}\right)^{2k}\cdot \frac{(\nu_n t)^k e^{-\nu_n t}}{k!}\\\no
    &= e^{-\nu_n t}+e^{-\nu_n t}\sum_{k=1}^\infty \frac{1}{k!}\left( \frac{A_0\cdot  t^{2H}}{(H+\frac{1}{2})^2}\nu_n\right)^k\\
    &=e^{-\nu_n t}+\exp\left(-\nu_n t+\frac{\nu_n\cdot A_0 t^{2H}}{(H+\frac{1}{2})^2}\right).
\end{align}
Finally, in \eqref{upper bound 1}, we take the temporal term back,
\begin{align}
\no E[u^n(t,x)] &\leq e^{\nu_n t}  E_{(0,x),...,(0,x)}[A_0^{N_t(\mathcal{P}_n)}Z_{n,t}]\\
\no&\leq 1+C_1 \exp(A_0 n^2 t^{2H})\\
&\leq C \exp(A_0 n^2 t^{2H}).
\end{align}
\end{proof}

Next we prove the case of Riesz kernel for $0<\alpha<2\wedge d$. Note that a similar argument can also be applied to justify the case $\alpha=0$ (case(1)). This idea is inspired by \cite{Foondun2008}. In this case, the space covariance function is unbounded, so we need to control it as well. This is done by choosing an appropriate norm and controlling its behavior by adjusting the parameters.

\begin{proof}[Proof of case (2)]
It suffices to consider the case when $k$ is an even integer, due to \cite[Theorem \rom{3}.1.2]{carmona1994} and \cite[Section 1]{Foondun2008}.
 Denote 
 $$S_{t,s}(x,y,z):=p_{h}(t-s,x-y)p_{h}(t-s,x-z)f(y-z)s^{2H-1}.$$
 By Burkholder's inequality and \cite[Equation 2.28]{Foondun2008},we have
 \begin{align}
 \label{kth norm}
\no\lefteqn{\mathbb{E}[(u(t,x))^k]^{1/k}}\\
&\leq 1+c_k  \cdot \mathbb{E}\Bigg[\Bigg(\int_0^t\int_{\mathbb{R}^d}\int_{\mathbb{R}^d}S_{t,s}(x,y,z)  u(s,y)u(s,z)dydzds\Bigg)^{k/2}\Bigg]^{1/k}.
  \end{align}
Define the norm $\displaystyle{\norm{X}_m=\mathbb{E}[X^m]^{1/m}}$. Equation~\eqref{kth norm} becomes
 \begin{align}
 \label{Normofsolution}
 \no \lefteqn{\norm{u(t,x)}_k}\\
\no &\leq 1+c_k\norm{\int_0^t\int_{\bR^d}\int_{\bR^d}S_{t,s}(x,y,z)u(s,y)u(s,z)dydzds)}_{k/2}^{1/2}\\
    \no &\leq 1+c_k \Bigg(\int_0^t \int_{\bR^d}\int_{\bR^d}S_{t,s}(x,y,z)\norm{u(s,y)u(s,z)}_{k/2}dydzds\Bigg)^{k/2\cdot(1/k)}\\
    &\leq 1+c_k\Bigg( \int_0^t\int_{\bR^d}\int_{\bR^d}S_{t,s}(x,y,z)\norm{u(s,y)}_{k}\norm{u(s,z)}_{k}dydzds\Bigg)^{k/2\cdot(1/k)}.
    \end{align}
The inequalities can be easily derived by H\"{o}lder's inequality and the triangle inequality.

 Now we introduce a factor $e^{-\beta t^\gamma}$ to calculate the asymptotic order of $\norm{u(t,x)}_k$.
 
 Let $\mathcal{N}_{\beta,\gamma,k}(u):=\sup_{s\geq 0, y\in \bR^d}(e^{-\frac{1}{2}\beta s^\gamma}\norm{u(s,y)}_k)$. This choice of norm is motivated by \cite[Theorem 5.1]{conuskhoshnevisan2013}. The parameters have to be adjusted to match the order of the fractional noise.
 
 Let $\displaystyle{\Upsilon_{\beta}(t):=\int_0^t\int_{\bR^d}\int_{\bR^d}S_{t,s}(x,y,z)e^{ \beta(s^\gamma)}dydzds}$. Now we want to prove that $\mathcal{N}_{\beta,\gamma,k}(u)$ is finite from \eqref{Normofsolution}.
\begin{align}
\label{gronwall}
 \no   \lefteqn{e^{-\frac{1}{2}\beta t^\gamma}\norm{u(t,x)}_k} \\
\no &\leq 1+c_k\Bigg( \int_0^t\int_{\bR^d}\int_{\bR^d}S_{t,s}(x,y,z)e^{-\frac{\beta}{2}s^\gamma}\norm{u(s,y)}_{k}e^{-\frac{\beta}{2}s^\gamma}\norm{u(s,z)}_{k}e^{\beta s^\gamma}dydzds\Bigg)^{k/2\cdot(1/k)}\\
\no   &\leq  e^{-\frac{1}{2}\beta t^\gamma}+ \mathcal{N}_{\beta,\gamma,k}(u)\cdot c_ke^{-\frac{1}{2}\beta  t^\gamma}\Bigg(\int_0^t\int_{\bR^d}\int_{\bR^d}S_{t,s}(x,y,z)e^{\beta  s^\gamma}dydzds\Bigg)^{1/2}\\
&=   e^{-\frac{1}{2}\beta t^\gamma} + \mathcal{N}_{\beta,\gamma,k}(u) c_k [e^{-\beta t^{\gamma}}\Upsilon_{\beta}(t)]^{1/2}.
\end{align}
We define
\begin{align}
\label{upsilonbetaheat}
    \Upsilon(\beta)=\sup_{t>0} e^{-\beta t^\gamma} \Upsilon_{\beta}(t).
\end{align}

The next step is to prove that $\sup_t e^{-\beta t^\gamma}\Upsilon_\beta(t)<\infty$ so that we can take $t\rightarrow \infty $ to control the supremum and obtain the optimal order of $\gamma.$

Since the Riesz kernel covariance function for any dimension $d\geq 1$ is 
\begin{align}
    f(x)=|x|^{-\alpha}.
\end{align}
(see Definition \ref{riesz kernel}), we have
\begin{align}
\label{noeffectofdimension}
\no\int_{\mathbb{R}^d}\hat{p}_{t-s}^2(\xi)\hat{f}(\xi)d\xi&=\int_{\mathbb{R}^d}e^{-(t-s)|\xi|^2}\mu(d\xi)\\
\no&=\int_{\mathbb{R}^d}e^{-(t-s)|\xi|^2}\xi^{-(d-\alpha)}\,d\xi\\
\no &=(t-s)^{-\frac{\alpha}{2}}(\int_{\mathbb{R}^d}e^{-|\eta|^2}\eta^{-(d-\alpha)}\,d\eta)\\
&=C\cdot (t-s)^{-\frac{\alpha}{2}},
\end{align}
where we let $\mu(d\xi)=\hat{f}(\xi)d\xi$ and $\eta=(t-s)^{1/2}\xi$.\\
It turns out that the space dimension has no impact for the order in time. By \eqref{noeffectofdimension} we have
\begin{align}
\label{upsilonbeta}
\no \Upsilon_{\beta}(t)&=\int_0^t\int_{\mathbb{R}^d}\int_{\mathbb{R}^d}p_{t-s}(x-y)p_{t-s}(x-z)f(z-y)s^{2H-1}e^{\beta (s^\gamma)}dydzds\\
\no &=\int_0^t\int_{\mathbb{R}^d}\int_{\mathbb{R}^d}p_{t-s}(\tilde y)p_{t-s}(\tilde z)f(\tilde z-\tilde y)s^{2H-1}e^{\beta (s^\gamma)}d\tilde yd\tilde zds\\
\no &=\int_0^t\int_{\mathbb{R}^d}\int_{\mathbb{R}^d}p_{t-s}(\tilde y)(f*p_{t-s})(\tilde y) s^{2H-1}e^{\beta (s^\gamma)}d\tilde yds\\
\no &=\int_0^t\int_{\mathbb{R}^d}\int_{\mathbb{R}^d}\hat p_{t-s}(\tilde y)\hat p_{t-s}(\tilde y)  s^{2H-1}e^{\beta (s^\gamma)}\mu(d\xi) ds\\
\no &=C\int_0^t(t-s)^{-\frac{\alpha}{2}}s^{2H-1}e^{\beta (s^\gamma)}ds\\
 &=C \cdot I^{-\frac{\alpha}{2}+1}(\bullet^{2H-1}e^{\beta \bullet ^\gamma})(t),
\end{align}
where we have used the substitution $\tilde{z}=x-z$, $\tilde{y}=x-y$. A similar argument is in Dalang \cite[Section 2]{Dalang1999}. The operator $I$ is the fractional integral defined in Section 4.1.
Expanding the right-hand side into a Taylor series, we obtain
\begin{align}
\label{fractionalcalc}
\no    \lefteqn{I^{-\frac{\alpha}{2}+1}(\bullet^{2H-1}e^{\beta \bullet ^\gamma})(t)} \\
\no    &=\sum_{m=0}^\infty \frac{\beta^m}{m!}\int_0^\infty (t-s)^{-\frac{\alpha}{2}}\cdot s^{\gamma m+2H-1}ds\\
    &=\sum_{m=0}^\infty \frac{\beta^m}{\Gamma(m+1)} t^{\gamma m+2H-1+-\frac{1}{2}(\alpha)+1} \int_0^1 (1-r)^{-\frac{\alpha}{2}}r^{\gamma m+2H-1}\,dr.
    \end{align}
    By the definition of Beta function, we have \\
    \begin{align}
    \label{betafunction}
       \displaystyle{\int_0^1 (1-r)^{-\frac{\alpha}{2}}r^{\gamma m+2H-1}\,dr=\frac{\Gamma(\gamma m +2H)\Gamma(-\frac{\alpha}{2}+1)}{\Gamma(-\frac{\alpha}{2}+1+\gamma m +2H)}} .
    \end{align}
Inserting \eqref{betafunction} into \eqref{fractionalcalc} implies
    \begin{align}
    \label{fracresult}
\no &I^{-\frac{1}{2}(\alpha)+1}(\bullet^{2H-1}e^{\beta \bullet ^\gamma})(t)\\
\no&=\sum_{m=0}^\infty \frac{\beta^m}{\Gamma(m+1)}t^{\gamma m+2H-1-\frac{1}{2}(\alpha)+1} \frac{\Gamma(\gamma m+2H)\Gamma(-\frac{1}{2}(\alpha)+1)}{\Gamma (\gamma m+2H-\frac{1}{2}(\alpha)+1)}\\
    &=\Gamma(-\frac{\alpha}{2}+1)  \sum_{m=0}^\infty
    (\gamma m)^{-(-\frac{\alpha}{2}+1)}\frac{\beta^m (t^\gamma)^m}{\Gamma(m+1)} t^{2H-1-\frac{1}{2}(\alpha)+1}.
\end{align}
The last equality comes from 
$$\frac{\Gamma (\gamma m+2H)
}{\Gamma (\gamma m+2H-\frac{1}{2}(\alpha)+1)}\sim
(\gamma m)^{-(-\frac{1}{2}(\alpha)+1)}=(\gamma m)^{\frac{\alpha}{2}-1},$$ by Stirling's formula. Again Stirling's formula yields, $\frac{m^{-(-\frac{1}{2}\alpha+1)}}{\Gamma(m+1)}\sim \frac{1}{\Gamma(m+1-\frac{1}{2}(\alpha)+1)}$.\\
We write \eqref{fracresult} using the incomplete Gamma function defined in Definition~\ref{defincomplete}, with $z=\beta t^\gamma$, $\nu= -\frac{1}{2}(\alpha)+1$. we have
\begin{align}
I^{-\frac{1}{2}(\alpha)+1}(\bullet^{2H-1}e^{\beta \bullet ^\gamma})(t)=\Gamma (-\frac{\alpha}{2} +1)(\gamma)^{\frac{\alpha}{2}-1} e^{\beta t^\gamma} \gamma^* \left(-\frac{1}{2}\alpha+1,\beta  t^\gamma\right) t^{2H-1-\frac{1}{2}\alpha+1}.
\end{align}
Taking the temporal term $e^{-\beta t^\gamma }$ back in \eqref{upsilonbeta}, in order to satisfy $e^{-\beta t^\gamma }\Upsilon_\beta (t)<\infty$, we only need to show that
\begin{align}
 \gamma^* \left(-\frac{1}{2}\alpha+1,\beta  t^\gamma\right) t^{2H-1-\frac{1}{2}\alpha+1}<\infty \hspace{5mm} \text{asymptotically as } t\rightarrow \infty.
\end{align} 
In order for this to hold, by Lemma~\ref{incompletegamma} we need $\alpha <2$ and $\gamma \geq \frac{4H-\alpha}{2-\alpha}$. We take the optimal order, namely $\gamma=\frac{4H-\alpha}{2-\alpha}$.

Now we take the supremum on both sides to obtain
\begin{align}
\label{Nlessthan1}
\mathcal{N}_{\beta,\gamma,k}(u)=\sup_{t,x} e^{-\beta t^\gamma} \norm{u(t,x)}_k\leq 1+c_k\mathcal{N}_{\beta,\gamma,k}(u)\Bigg(\sup_t e^{-\beta t^\gamma}\Upsilon_\beta(t)\Bigg)^{1/2}.
\end{align}

To ensure $\mathcal{N}_{\beta,\gamma,k}(u)<\infty$ from \eqref{Nlessthan1}, we need $c_k[\Upsilon (\beta)]^{1/2}<1$ ,which will allow us to determine the optimal order of $\beta$. Since $ c_k \sim 2\sqrt{k}$ as $k$ goes to infinity (see Foondun-Khoshnevisan\cite{Foondun2008}), we must have $\beta^{\frac{\alpha}{4}-\frac{1}{2}}\sqrt{k} <1$, which indicates that $\beta>k^{\frac{2}{2-\alpha}}$. Therefore, we know $\norm{u(t,x)}_k$ is of order $k^{\frac{2}{2-\alpha}} t^{\frac{4H-\alpha}{2-\alpha}}$, and thus $E[u(t,x)^k]$ is of order $ \exp(k^{\frac{4-\alpha}{2-\alpha}} t^{\frac{4H-\alpha}{2-\alpha}})$.
\end{proof}
We finish by proving case (3) of Theorem \ref{thm:upper bound}. We directly use an argument out of \cite{Balan2016}.

\begin{proof}[Proof of case (3)]
By Section 7.3 in \cite{Balan2016}, since the Fourier transform of $f$ is the same as the Fourier transform of $\delta_0$ as $\alpha \rightarrow 1$, Lemmas 7.3, 7.4 and 7.5 in Balan-Conus\cite{Balan2016} still hold, and we can take the limit as $\alpha\rightarrow 1$ in the result for case (2) to obtain
\begin{align}
    E[u(t,x)^k]\leq C e^{k^3 t^{4H-1}}.
\end{align}
\end{proof}
\begin{remark}
Notice that as $H\rightarrow 1/2$, we recover the known result from Foondun-Khoshnevisan \cite{Foondun2008}.
\end{remark}

\section{Lower bound for the heat equation}
For all three cases (1)–(3), we follow the approach of Dalang and Mueller\cite{Mueller2009} to derive the lower bound result. We first introduce an assumption on the space covariance function that will serve as a symmetric equivalent to Assumption \ref{assumption:bounded} that will allow us to control $f$ from below. 

\begin{assumption}
\label{assumption:lower bound}
The covariance function $f$ has the following property: there exist $\delta>0$ and $A_0>0$ such that for $\norm{x}<2\delta, f(x) \geq A_0$.
\end{assumption}
\begin{theorem}[Lower bound for heat equation]
\label{theoremlowerbound}
Under Assumption \ref{assumption:finite var} and \ref{assumption:lower bound}, let $u$ be the solution to the stochastic heat equation~\eqref{eqn:DO process}. There exists a universal constant $c>0$ such that for all even $n\geq 2$, $x\in \mathbb{R}^d$ and $t>0$, we have
\begin{align}
E[u^n(t,x)]\geq C\exp(A  t^{\frac{4H-a}{2-a}}n^{\frac{4-a}{2-a}}).
\end{align}
\end{theorem}
\begin{remark}
Without Assumption \ref{assumption:lower bound}, the inequality $E[u^n(t,x)]\geq u_0^n$ holds for all $t\geq 0$. Indeed, by Theorem \ref{thm:N_mom_formula},
\begin{align}
E[u^n(t,x)]\geq e^{tn(n-1)/2}E_{(0,x),...,(0,x)}[\mathbbm{1}_{\{N_t(\mathcal{P}_n)=0\}}u_0^n]=u_0^n.
\end{align}
But we cannot derive the intermittency property (exponentially growth) from it.
\end{remark} 
Before proceeding with the proof, we will introduce some notations, inspired by Dalang-Mueller\cite{Mueller2009}.
For any $x,y \in \mathbb{R}^d $ with $x = y$, we consider the solid (infinite) cone $C(x,y)$ in $\mathbb{R}^d$, with vertex $y$, axis oriented in the direction of the vector $x-y$ and an angle of $\pi/4$ between the axis and any lateral side. 

By Assumption \ref{assumption:lower bound}, with the notation in Chapter 3,  we will have $f(X_{\sigma_i}^{R_i^1}-X_{\sigma_i}^{R_i^2})\geq A$ if $|X_{\sigma_i}^{R_i^1}-X_{\sigma_i}^{R_i^2}|<2\delta$. In order to guarantee that the event $|X_{\sigma_i}^{R_i^1}-X_{\sigma_i}^{R_i^2}|<2\delta$ happens, we introduce two events $D(t)$ and $C(k,n,t)$, where $D(t)$ restricts the direction of the process and $C(k,n,t)$ restricts the length of time between each jump of the Poisson process.\\
To be more specific, let's define the event $D(t)$ to be
\begin{align}
\label{eventD}
 D(t)&=\bigcap_{k=1}^n\bigcap_{i=1}^{N_t(k)} \Bigg{(}\frac{B^{(k,i)}_{\tau_{i+1}^k-\tau_i^k}}{\sqrt{\tau_{i+1}-\tau_i}}\in C(x-X_{\tau_i}^{(k)},0)\cap B(0,1) \Bigg{)}.
\end{align}
\begin{remark}
For the event $D(t)$ in \eqref{eventD}, if $d=1$, we have
\begin{align*}
D(t)&=\bigcap_{k=1}^n\bigcap_{i=1}^{N_t(k)} \Bigg{(}\{B^{(k,i)}_{\tau_{i+1}^k-\tau_i^k}\in (-\sqrt{\tau_{i+1}^k-\tau_i^k},0)\}\cap \{X_{\tau_i^k}^{(k)}-x>0\}\Bigg{)}\\
\no&\cup \Bigg{(}\{B^{(k,i)}_{\tau_{i+1}^k-\tau_i^k}\in (0,\sqrt{\tau_{i+1}^k-\tau_i^k})\}\cap \{X_{\tau_i^k}^{(k)}-x<0\}\Bigg{)}.
\end{align*}
\end{remark}
A graphical example is presented in Figure 4.1. 

\begin{figure}[ht]
\centering
\includegraphics[width=8cm]{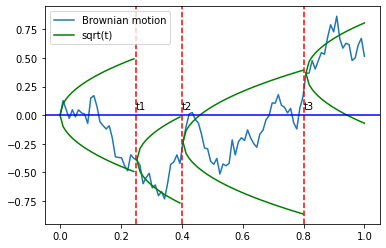}    
\caption{Illustration of the process $X_t$ starting from $x=0$.}
\end{figure}
In the graph we see that for each jump, $X_{\tau_i}$ will move in the opposite direction of $X_{\tau_{i-1}}-x$ and the end point should be in the ball $B(0,\sqrt{t})$. But between two jumps, the process can be outside the cone and the ball.

 Let $m\in \mathbb{N}$, here we assume $m=M\cdot t$, where $M$ is a positive real number indepedent of $t$ and set $\displaystyle{k=\frac{m}{\delta^2}\frac{n}{2}}$. Let $\ell=\frac{\delta^2 t}{2(m+1)}$ and, for $j=1,...,\frac{m}{\delta^2}$, let $t_j=\frac{jt\delta^2}{2(m+1)}$ and $I_j=[a_j,b_j]$, where $a_j=t_j-\ell/4$ and $b_j=t_j+\ell/4$, so that the length of $I_j$ is $\frac{\ell}{2}=\frac{\delta^2 t}{4(m+1)}$ and $I_j$ and $I_{j+1}$ are separated by an interval of length $a_{j+1}-b_j=\ell/2$.\\
Further let
\begin{align*}
C(k,n,t)=\bigcap_{j=1}^{m/\delta^2} G_j(n,k),
\end{align*}
where
\begin{align*}
    G_j(n,k)=\{N_{b_j}(\mathcal{P}_n)-N_{a_j}(\mathcal{P}_n)=\frac{n}{2}\}\cap \{N_{b_j}(q)-N_{a_j}(q)=1,q=1,...,n\}.
\end{align*}
Notice that on $C(k,n,\ell)$, $N_t(\mathcal{P}_n)=\frac{m}{\delta^2}\frac{n}{2}=k$, and for each time interval $I_j$, each process $X^{q}$ reset once. Also, for $C(k,n,t)$, we have
\begin{align*}
    \frac{\delta^2 t}{4(m+1)}\leq \tau_{i+1}^q-\tau_i^q \leq \frac{\delta^2 t}{m+1}, i=0,...,m.
\end{align*}
We can derive that $\tau_{i+1}^q-\tau_i^q\leq \delta^2 $ if $m$ large enough.
\begin{lemma}
 On the event $D(t)$, if $|X^k_{\tau_i}-x|<\delta$ and $|\tau^k_{i+1}-\tau^k_i|<\delta^2$, then we have
 \begin{align}
 |X^k_{\tau_{i+1}^k}-x|<\delta.
 \end{align}
\end{lemma}
\begin{proof}
If $0<X^k_{\tau_i^k}-x<\delta$, then the event $D(t)$ requires that \\
\begin{align}
\label{ballrestriction}
-\delta<B^{(k)}_{\tau^k_{i+1}-\tau_i^k}<0.
\end{align}
We duduce that 
\begin{align}
 \no   X_{\tau_{i+1}^k}^k-x&=X^k_{\tau_i^k}-x+B^{(k)}_{\tau_{i+1}^k-\tau_i^k}.
 \end{align}
The assumption and \eqref{ballrestriction}  imply
 \begin{align}
 \no-\delta&<X_{\tau_{i+1}^k}^k-x<\delta.
\end{align}
If $-\delta<X^k_{\tau_i^k}-x<0$, then the event $D(t)$ requires that \\
$$0<B^{(k)}_{\tau^k_{i+1}-\tau_i^k}<\delta.$$
Similarly as above, we use
\begin{align}
 \no   X_{\tau_{i+1}^k}^k-x&=X^k_{\tau_i^k}-x+B^{(k)}_{\tau_{i+1}^k-\tau_i^k}
 \end{align}
 to deduce that 
 \begin{align}
\no   -\delta&<X_{\tau_{i+1}^k}^k-x<\delta.
\end{align}
\end{proof}
\begin{lemma}
\label{conditionq}
 We have 
     \begin{align}
 \displaystyle{P(D(t)|N_t(\mathcal{P}_n)=k)=(\Phi(1)-1/2)^{2k}}.
    \end{align}

\end{lemma}
\begin{proof}
Let $q=P(0<B_{\tau}<\sqrt{\tau})$. Then, by independence of the $B_t^{(k)}$, we have
     \begin{align}
     \label{qtwok}
 \displaystyle{P(D(t)|N_t(\mathcal{P}_n)=k)=\prod_{\ell=1}^n q^{N_t(\ell)}=q^{N_t(\mathcal{P}_n)}=q^{2k}}.
    \end{align}
Moreover, 
\begin{align}
\no q&=P(0<B_{\tau}<\sqrt{\tau})\\
\no&=\int_0^\infty e^{-t} P(0<B_t<\sqrt{t})\,dt\\
\no &=\int_0^\infty e^{-t} P(0<t^{-1/2}B_t<1)\,dt\\
\no&=\int_0^\infty e^{-t}P(0<Z<1)\,dt\\
\no & =P(0<Z<1)\\
\no &=\Phi(1)-1/2.
\end{align}
\end{proof}
\begin{lemma}
\label{dalangmuellerlemma}
We have
  \begin{align}
\no     P(C(k,n,t)|N_t(\mathcal{P}_n)=k)\geq k!(\sqrt{2}c)^{2k/n}e^{-k}n^k(\frac{1}{8kn})^k,
 \end{align}
where $c$ is a constant indepedent of $t$.
\end{lemma}
 \begin{proof}
 The proof of this result is in Dalang-Mueller \cite[Equation (4.7)]{Mueller2009}.
 \end{proof}

We are now ready to prove Theorem \ref{theoremlowerbound} in the spatial smooth noise case.

\begin{proof}[Proof of Theorem \ref{theoremlowerbound} case (1)]
Let $\tilde{Z}_{n,t}=\prod_{\ell=1}^n\prod_{i=1}^{N_t(\ell)}(t-\tau_i^\ell)^{H-1/2}$. By formula \eqref{thm:N_mom_formula} and Assumption \ref{assumption:lower bound}, we have
\begin{align}
  \no  \lefteqn{E[u^n(t,x)]}\\
    &\geq e^{\nu_n t} E_{(0,x),...,(0,x)}[A^{N_t(\mathcal{P}_n)}\tilde{Z}_{n,t}\mathbbm{1}_{D(t)}\mathbbm{1}_{C_{k,n,t}}|N_t(\mathcal{P}_n)=k]\cdot P(N_t(\mathcal{P}_n)=k)
    \label{heatZnt}.
\end{align}
For the lower bound we can focus solely on one term in the conditioning, namely when $N_t(\mathcal{P}_n)=k$. The specific value of $k$ will be chosen later. This is valid since each term is non-negative.\\
On the event $C(k,n,t)$ we have $t-\tau_i > \frac{t}{2}$ for all $i$. Thus, using Lemma~\ref{conditionq} we obtain
\begin{align}
\lefteqn{E[A ^{N_t(\mathcal{P}_n)}\tilde{Z}_{n,t}\mathbbm{1}_{C(k,n,t)}\mathbbm{1}_{D(t)}|N_t(\mathcal{P}_n)=k]}\no\\
&\geq A^{k}q^{2k}\bigl(\frac{t}{2}\bigl)^{2k\cdot (H-1/2)}P(C(k,n,t)|N_t(\mathcal{P}_n)=k).
\end{align}

We can bound the term
 $\displaystyle{P(C(k,n,t)|N_t(\mathcal{P}_n)=k)}$ by Lemma \ref{dalangmuellerlemma}. Finally,
 \begin{align*}
E[u^n(t,x)]&\geq e^{\nu_n t} q^{2k}A^{k}\bigl(\frac{t}{2}\bigl)^{2k\cdot (H-1/2)} k!(\sqrt{2}c)^{2k/n}e^{-k}n^k\Bigl(\frac{1}{8kn}\Bigl)^k\cdot \frac{e^{-\nu_n t}(\nu_n t)^k}{k!}\\
&= \Bigg(\frac{A (\frac{t}{2})^{2H-1}q^2 (\sqrt{2}c)^{2/n}e^{-1}n \nu_n t}{8kn}\Bigg)^k.
 \end{align*}
 Now we want to make the term in the parentheses equal to $e$, so we need to set $k= A  t^{2H} q^2(2c^2)^{1/n} e^{-1}\nu_n 2^{-(2H-1)}\cdot e^{-1}$.\\
 Notice that $\nu_n\sim n^2$,
 so $\displaystyle{E[u^n(t,x)]\geq \exp(A t^{2H} (2c^2)^{1/n}q^2 e^{-1}\nu_n 2^{-(2H-1)})}$, which proves the result.
\end{proof}
Now let's prove Theorem \ref{theoremlowerbound} in case (2).
\begin{proof}[Proof of case(2)]
By Assumption~\ref{assumption:lower bound} we obtain
\begin{align}
\no\lefteqn{E[u^n(t,x)]}\\
&\geq e^{\nu_n t} E_{(0,x),...,(0,x)}[A^{N_t(\mathcal{P}_n)}\tilde{Z}_{n,t}\mathbbm{1}_{D(t)}\mathbbm{1}_{C(k,n,t)}|N_t(\mathcal{P}_n)=k]\cdot P(N_t(\mathcal{P}_n)=k).
\end{align}

We proceed as we did in the smooth noise case,
the event $D(t)$, $C(k,n,t)$ and $\tilde{Z}_{n,t}$ remaining the same as in case (1). In the Riesz kernel case, we now have $A=(2\delta)^{-\alpha}$ in Assumption~\ref{assumption:lower bound}. This gives 
  \begin{align}
\no E[u^n(t,x)]&\geq e^{\nu_n t} q^{2k}A^{k}(\frac{t}{2})^{2k\cdot (H-1/2)} k!(\sqrt{2}c)^{2k/n}e^{-k}n^k(\frac{1}{8kn})^k\cdot \frac{e^{-\nu_n t}(\nu_n t)^k}{k!}\\
\no &= \Bigg(\frac{(2\delta)^{-\alpha} (\frac{t}{2})^{2H-1}q^2 (\sqrt{2}c)^{2/n}e^{-1}n \nu_n t}{8kn}\Bigg)^k\\
&=\Bigg(\frac{(2(\frac{M\cdot t n}{2k})^{1/2})^{-\alpha} (\frac{t}{2})^{2H-1}q^2 (\sqrt{2}c)^{2/n}e^{-1}n \nu_n t}{8kn}\Bigg)^k.
 \end{align}
 Now we choose $k=\bigl(C t^{2H-\alpha/2} n^{2-\alpha/2}\bigl)^{\frac{1}{1-\alpha/2}}$, which yields
 \begin{align}
E[u^n(t,x)]\geq \exp(C  t^{\frac{4H-\alpha}{2-\alpha}}n^{\frac{4-\alpha}{2-\alpha}}).
   \end{align}
 \end{proof}
 \begin{proof}[Proof of case (3)]
 Here then by Conus-Balan \cite[Section 7.3]{Balan2016}, we can take the limit as $\alpha\rightarrow 1$ in the result. This leads to  $E[u^n(t,x)]\geq \exp(C t^{4H-1} n^{3}).$
\end{proof}

\section{Wave equation}
In this section, we establish similar upper and lower bounds as in the previous two section, but for the wave equation. We have the following theorem.

\begin{theorem}[Upper Bound for the wave equation]
Let $u$ be the solution to equation \eqref{eqnwave:DO process} driven by the noise $W$ defined in \eqref{Do definition eqn}, where $W_0$ is any of the noises of cases (1),(2),(3) in Chapter 2. Then there exists a constant $C$ such that
\label{thmwave:upper bound}

\begin{align}
    E[u^n(t,x)]\leq C u_0^n \cdot \exp(C  n^{\frac{4-a}{3-a}} t^{\frac{2H+2-a}{3-a}}).
\end{align}
\end{theorem}
\begin{proof}[Proof of case (1)]
The proof is similar to the heat equation case. The main difference is that
\begin{align}
Z_{n,t}=\prod_{\ell=1}^n\prod_{i=1}^{N_t(\ell)}((t-\tau_{i}^\ell)^{H-1/2}(\tau_i^\ell-\tau_{i-1}^\ell)). 
\end{align}
By Theorem \ref{thmwave:N_mom_formula} and Assumption \ref{assumption:bounded}, let $x_1=x_2=...=x_n=x$. Since the initial conditions are constant and, we have $w(t,x_i)=u_0+t v_0=w_0$. Since $u(t,x)\geq 0 $ for all $t>0,x\in \bR^d$, we have
\begin{align}
\label{wave:upper bound 1}
   E[u^n(t,x)]\leq w_0^n e^{t n(n-1)/2} E_{(0,x),...,(0,x)}[A_0^{N_t(\mathcal{P}_n)}Z_{n,t}].
\end{align}
By Lemma \ref{wavesmoothupperbound}, now we have
\begin{align}
\no E[u^n(t,x)] &\leq e^{tn(n-1)/2} u_0^nE_{(0,x),...,(0,x)}[A_0^{N_t(\mathcal{P}_n)}Z_{n,t}]\\
&\leq  u_0^n \exp(A_0  t^{2H}n^2),
\end{align}
where the second inequality comes from Lemma \ref{wavesmoothupperbound} below. This proves the result for case (1).
\end{proof}
\begin{lemma}
There is a universal constant $C<\infty $ such that for all $n\geq 2$, $t\geq 0$ and $x\in \mathbb{R}^d$,
\label{wavesmoothupperbound}
 \begin{align}
    e^{tn(n-1)/2}E_{(0,x),...,(0,x)}[A_0^{N_t(\mathcal{P}_n)}Z_{n,t}]\leq C\exp(A_0 t^{(2H+2)/3}n^{4/3}).
\end{align}
\end{lemma}
\begin{proof}

Let $\displaystyle{\nu_n=\frac{n(n-1)}{2}}$,$\displaystyle{\sum_{\rho\in \mathcal{P}_n}N_t(\rho)=N_t(\mathcal{P}_n)=\frac{1}{2}\sum_{\ell=1}^nN_t(\ell)}$.

By the arithmetic-geometric inequality, same as in \eqref{arithmeticmean},
$$\displaystyle{\left(\prod_{\ell=1}^ka_\ell \right)^{1/k}\leq \frac{\sum_{\ell=1}^k a_\ell}{k}}.$$
Then we have
\begin{align}
\label{wavesmootheqn}
    E_{(0,x),...,(0,x)}[A_0^{N_t(\mathcal{P}_n)}Z_{n,t}|N_t(\mathcal{P}_n)=k]\leq \Bigg{(}\frac{1}{2k}\sum_{\ell=1}^n\sum_{i=1}^{N_t(\ell)}A_0^{1/2}(t-\tau_{i}^\ell)^{H-1/2}(\tau_i^\ell-\tau_{i-1}^\ell)\Bigg{)}^{2k}.
\end{align}

Now we focus on the sum
\begin{align}
\sum_{\ell=1}^n\sum_{i=1}^{N_t(\ell)}((t-\tau_{i}^\ell)^{H-1/2}(\tau_i^\ell-\tau_{i-1}^\ell)).
\end{align}
We know $\tau_i$ is the ith ordered statistics from a sample of $N_t(\mathcal{P}_n)$ uniform distributions on $[0,t]$. Writing $U_{(i)}=\frac{\tau_i}{t}$, we have
\begin{align}
\no\sum_{\ell=1}^n\sum_{i=1}^{N_t(\ell)}(t-\tau_i)^{H-1/2}(\tau_i^\ell-\tau_{i-1}^\ell)&=t^{H+1/2}\sum_{\ell=1}^n\sum_{i=1}^{N_t(\ell)}(1-U_{(i)})^{H-1/2}(U_{(i)}-U_{(i-1)})\\\no
&\leq t^{H+1/2}\sum_{\ell=1}^n\sum_{i=1}^{N_t(\ell)}(U_{(i)}-U_{(i-1)})\\
&\leq t^{H+1/2}n.
\end{align}
Now we get back to Equation \eqref{wavesmootheqn}, we know, by the results above,
\begin{align*}
    \lefteqn{E_{(0,x),...,(0,x)}[A_0^{N_t(\mathcal{P}_n)}Z_{n,t}]}\\    &=\sum_{k=0}^\infty E_{(0,x),...,(0,x)}[A_0^{N_t(\mathcal{P}_n)}Z_{n,t}|N_t(\mathcal{P}_n)=k]\cdot P(N_t(\mathcal{P}_n)=k)\\
    &\leq e^{-\nu_n t}+ \sum_{k=1}^\infty(\frac{1}{2k}\cdot 2\cdot A_0^{1/2} t^{H+1/2}n)^{2k}\cdot \frac{(\nu_n t)^k e^{-\nu_n t}}{k!}\\
    &\leq e^{-\nu_n t}+e^{-\nu_n t} C_1\cdot  \sum_{k=1}^\infty( A_0^{1/3}\cdot  t^{(2H+2)/3}n^{4/3})^{3k}/(3k)!\\
    &\leq C e^{-\nu_n t}\exp(A_0^{1/3}t^{(2H+2)/3}n^{4/3}).
\end{align*}
The second inequality comes from Stirling's formula, $k!(2k)^{2k}\geq C_1 (3k)!$ for $k\geq k_0$, see \cite[Lemma 3.3]{Mueller2009}. The last inequality is true since each term of the series is non-negative. We deduce
\begin{equation}
e^{\nu_n t}  E_{(0,x),...,(0,x)}[A_0^{N_t(\mathcal{P}_n)}Z_{n,t}] \leq C \exp(A_0^{1/3}t^{(2H+2)/3}n^{4/3}).
\end{equation}
\end{proof}
\begin{proof}[Proof of case (2)]
We follow the same approach as in the heat equation case, until equation~\eqref{upsilonbetaheat}. There, considering $p_w$ instead of $p_h$, we let $\eta=(t-s)\xi$ to obtain
\begin{align}
\no\int_{\mathbb{R}^d}\hat{p}_{w}^2(t-s,\xi)\mu(d\xi)&=\int_{\mathbb{R}^d}\Bigg(\frac{\sin((t-s)|\xi|)}{ |\xi|}\Bigg)^2\xi^{-(d-\alpha)}\,d\xi\\
\no&=\int_{\mathbb{R}^d}\frac{\sin^2(|\eta|)}{|\eta|^2}\eta^{-(d-\alpha)}(t-s)^{d-\alpha-d+2}\,d\eta\\
&=C (t-s)^{2-\alpha}.
\end{align}
Here the integral is finite, see Balan-Conus\cite[Lemma 4.1]{Balan2016}. Further, let $\Upsilon_\beta(t)$ be defined the same as in \eqref{upsilonbetaheat}, then
\begin{align}
\no \Upsilon_\beta(t)&=\int_0^t\int_{\mathbb{R}}p_{t-s}(x-z)p_{t-s}(x-y)f(y-z)s^{2H-1}e^{-\beta (t^\gamma -s^\gamma)}dydzds\\
\no &=C\int_0^t(t-s)^{(2-\alpha)}s^{2H-1}e^{-\beta (t^\gamma -s^\gamma)}ds\\
 &=C e^{-\beta t^\gamma}I^{(3-\alpha)}(\bullet^{2H-1}e^{\beta \bullet ^\gamma})(t).
\end{align}
Use Taylor series expansion upon $I^{(3-\alpha)}(\bullet^{2H-1}e^{\beta \bullet ^\gamma})(t)$, we obtain
\begin{align}
\label{incomplete gamma fcn}
\no    I^{3-\alpha}(\bullet^{2H-1}e^{\beta \bullet ^\gamma})(t)&=\sum_{m=0}^\infty \frac{\beta^m}{m!}\int_0^\infty (t-s)^{2-\alpha}\cdot s^{\gamma m+2H-1}ds\\
    &=\sum_{m=0}^\infty \frac{\beta^m}{\Gamma(m+1)} t^{\gamma m+2H-1+(2-\alpha)+1} \int_0^1 (1-r)^{2-\alpha}r^{\gamma m+2H-1}\,dr.
    \end{align}
    By the definition of the Beta function, we have
    \begin{align}
    \displaystyle{\int_0^1 (1-r)^{2-\alpha}r^{\gamma m+2H-1}\,dr=\frac{\Gamma(\gamma m +2H)\Gamma(3-\alpha)}{\Gamma(3-\alpha+\gamma m +2H)}}.
        \end{align}
    So the equation \eqref{incomplete gamma fcn} becomes
    \begin{align}
\no   \lefteqn{ I^{3-\alpha}(\bullet^{2H-1}e^{\beta \bullet ^\gamma})(t)}\\
\no   &=\sum_{m=0}^\infty \frac{\beta^m}{\Gamma(m+1)}t^{\gamma m+2H-1+3-\alpha} \frac{\Gamma(\gamma m+2H)\Gamma(3-\alpha)}{\Gamma (\gamma m+2H+3-\alpha)}\\
    &=\Gamma(3-\alpha) (\gamma m)^{-(3-\alpha)} \sum_{m=0}^\infty \frac{(\beta)^m(t^\gamma)^m}{\Gamma(m+1)} t^{2H-1+(3-\alpha)} .
\end{align}

The last equality comes from 
$\frac{\Gamma (\gamma m+2H)}{\Gamma (\gamma m+2H+3-\alpha)}\sim(\gamma m)^{-(3-\alpha)}=(\gamma m)^{\alpha-3}$ by Stirling's formula. Use Stirling formula again,
\begin{align}
\frac{m^{-(3-\alpha)}}{\Gamma(m+1)}\sim \frac{1}{\Gamma(m+1+(3-\alpha))},
\end{align}
and the incomplete Gamma function (Definition \ref{incompletegamma}) 
with $z=\beta t^\gamma$, $\nu= 3-\alpha$.\\ By Lemma~\ref{incompletegamma},
\begin{align}
I^{3-\alpha}(\bullet^{2H-1}e^{\beta \bullet ^\gamma})(t)=\Gamma (3-\alpha)\gamma^{-(3-\alpha)} \gamma^* (3-\alpha,\beta  t^\gamma) t^{2H-1+3-\alpha}.
\end{align}
Taking the temporal term $e^{-\beta t^\gamma }$ back in \eqref{upsilonbeta}, in order to guarantee that $e^{-\beta t^\gamma }\Upsilon_\beta (t)<\infty$, we only need to show that
\begin{align}
\gamma^* (3-\alpha,\beta  t^\gamma) t^{2H-1+3-\alpha}<\infty \hspace{5mm} \text{asymptotically as } t\rightarrow \infty.
\end{align} 
In order for this to hold, by Lemma~\ref{incompletegamma} we need $\alpha <3$ and $\gamma \geq \frac{2H+2-\alpha}{3-\alpha}$. We take the optimal order, namely $\gamma=\frac{2H+2-\alpha}{3-\alpha}$.

Now we take the supremum on both sides to obtain
\begin{align}
\label{Nlessthan1wave}
 \mathcal{N}_{\beta,\gamma,k}(u)=\sup_{t,x} e^{-\beta t^\gamma}\norm{u(t,x)}_k\leq 1+c_k\mathcal{N}_{\beta,\gamma,k}(u)\Bigg(\sup_t e^{-\beta t^\gamma}\Upsilon_\beta(t)\Bigg)^{1/2}.
\end{align}

To ensure $\mathcal{N}_{\beta,\gamma,k}(u)<\infty$ from \eqref{Nlessthan1wave}, we also need $c_k[\Upsilon (\beta)]^{1/2}<1$ ,which will allow us to determine the optimal order of $\beta$. Since $ c_k \sim 2\sqrt{k}$ as $k$ goes to infinity, (see Foondun-Khoshnevisan\cite{Foondun2008}), we must have $\beta^{-\frac{1}{2}(3-\alpha)}\sqrt{k} <1$, which indicates that $\beta>k^{\frac{1}{3-\alpha}}$. Therefore, we know $\norm{u(t,x)}_k$ is of order $k^{\frac{1}{3-\alpha}} t^{\frac{2H+2-\alpha}{3-\alpha}}$, and thus $E[u(t,x)^k]$ is of order $ \exp(k^{\frac{4-\alpha}{3-\alpha}} t^{\frac{2H+2-\alpha}{3-\alpha}})$.
 \end{proof}

\begin{theorem}[Lower Bound for the Wave Equation]
Let $u$ be the solution to equation \eqref{eqnwave:DO process} driven by the noise $W$ defined in \eqref{Do definition eqn}, where $W_0$ is any of the noises of cases (1),(2),(3) in Chapter 2. Then there exists a constant $C$ such that
 \begin{align}
 E[u^n(t,x)]\geq \exp(C n^{\frac{4-a}{3-a}} t^{\frac{2H+2-a}{3-a}}).
  \end{align}
  \end{theorem}
   Here the event $D(t)$ and $C(k,n,t)$ remain the same, $P(D(t)|N_t(\mathcal{P}_n)=k) =q^{2k}$ as in \eqref{qtwok} but with a different $q$ according to the wave equation, $C(k,n,t)$ is identically the same since $\delta$ will not affect the answer.
   
We will illustrate the proof of case (2). The other cases follow similarly by an adjustment of the proof of the heat equation.

 \begin{proof}
 We follow the same proof as for heat equation until \eqref{heatZnt}. Notice that $\tilde{Z}_{n,t}=\prod_{\ell=1}^n\prod_{i=1}^{N_t(\ell)}(t-\tau_i^\ell)^{H-1/2}(\tau^\ell_i-\tau^\ell_{i-1})$, and given that we are on the event $C(k,n,t)$, we have
  \begin{align}
 \no    \lefteqn{ E[A^{N_t(\mathcal{P}_n)}\tilde{Z}_{n,t}\mathbbm{1}_{C(k,n,t)}|N_t(\mathcal{P}_n)=k]}\\
    \no  &\geq A^{k}(\frac{t}{2})^{2k\cdot (H-1/2)}(\frac{\delta t}{4(m+1)})^{2k}P(C(k,n,t)|N_t(\mathcal{P}_n)=k)\\
 \no     &\geq A^{k}(\frac{t}{2})^{2k\cdot (H-1/2)}(\frac{ctn}{k})^{2k}P(C(k,n,t)|N_t(\mathcal{P}_n)=k)\\
      &\geq A^{k}\frac{t^{2kH+k}n^{2k}}{k^{2k}}P(C(k,n,t)|N_t(\mathcal{P}_n)=k).
  \end{align}
  Again, we will bound $P(C(k,n,t)|N_t(\mathcal{P}_n)=k)$ as in the heat equation proof to obtain 
\begin{align}
\no E[u^n(t,x)]&\geq e^{\nu_n t} q^{2k}A^{k}C\frac{t^{2kH+k} n^{2k}}{k^{2k}} k!(\sqrt{2}c)^{2k/n}e^{-k}n^k(\frac{1}{8kn})^k\cdot \frac{e^{-\nu_n t}(\nu_n t)^k}{k!}\\
\no&= \left(\frac{A(2\delta)^{-\alpha}C (\frac{t^{2H+1}n^{2}}{k^{2}})q^2 (\sqrt{2}c)^{2/n}e^{-1}n \nu_n t}{8kn}\right)^k\\
&=\left(\frac{A(2(\frac{M\cdot t n}{2k}))^{-\alpha} C(\frac{t^{2H+1}n^{2}}{k^{2}})q^2 (\sqrt{2}c)^{2/n}e^{-1}n \nu_n t}{8kn}\right)^k.
\end{align}
Then we set $k=(C  t^{2H+2-\alpha}n^{4-\alpha})^{\frac{1}{3-\alpha}}$ and obtain
\begin{align}
E[u^n(t,x)]\geq \exp(C n^{\frac{4-\alpha}{3-\alpha}} t^{\frac{2H+2-\alpha}{3-\alpha}}).
\end{align}
 \end{proof}
 Case (1) is treated in a similar way. Case (3) is obtained by taking the limit as $\alpha \rightarrow 1$ using the technique of Balan, Conus \cite[Section 7.3]{Balan2016}.

\chapter{Generalized fractional noise}
From now on, the notation $\dot{W}$ will stand for the partial derivative
$\frac {\partial^{d+1}W}{ \partial t \partial x_1 \cdots \partial x_d} $ ,
where $W$ is a random field  formally defined in the next section. Informally
we assume that $\dot{W}$ has a covariance of the form
\begin{align}
\label{generalfracnoise}
\be \lc\dot{W}(t,x) \dot{W}(s,y)\rc=\gamma(s,t) \, f(x-y),
\end{align}
$f$ is the space covariance function as in \eqref{threecasesforf}, and $\gamma(s,t)=(st)^{a_1}|s-t|^{a_2}$, with $a_1\geq 0, -1 < a_2\leq 0$ and
\begin{align}
\label{a1a2condition}
2a_1+a_2=2H-2.
\end{align} 

Notice here that if $a_1=0, a_2= 2H-2$, we recover the standard fractional noise of \cite{Balan2016} and \cite{HuHuangNualartTindel} defined in \eqref{eqn:frac_noise}. If $ a_1=H-1/2,a_2=-1 $, it becomes the DO noise discussed in Chapter 2-4, see Defintion~\ref{DO noise def}. 

\section{Definition of the noise}
This section gives us a formal definition
of the noise $W$. Some basic
elements of Malliavin calculus will be introduced next.

Let us start by introducing some basic notions on Fourier transforms
of functions: the space of real valued infinitely differentiable
functions with compact support is denoted by $\mathcal{D} (
\mathbb{R}^d)$ or $\mathcal{D}$. The Fourier
transform is defined as:
\[ \mathcal{F}u ( \xi)  = \int_{\mathbb{R}^d} e^{- \imath\xi x} u ( x) d x, \]
so that the inverse Fourier transform is given by $\mathcal{F}^{- 1} u ( \xi)
= ( 2 \pi)^{- d} \mathcal{F}u ( - \xi)$.

\smallskip
Here, we will follow a similar procedure as was introduced in Chapter 1 for the fractional noise, see also \cite{HuHuangNualartTindel}. On a complete probability space
$(\Omega,\cf,\bp)$, we consider a Gaussian noise $W$ encoded by a
centered Gaussian family $\{W(\vp) ; \, \vp\in
\mathcal{D}([0,\infty)\times \R^{d})\}$, whose covariance structure
is given by
\begin{equation}\label{cov1}
\be\lc W(\vp) \, W(\psi) \rc
= \int_{\R_{+}^{2}\times\R^{2d}}
\varphi(s,x)\psi(t,y)s^{a_1}t^{a_1}|s-t|^{a_2}|x-y|^{-\alpha}dxdydsdt.
\end{equation}


Let $\mathcal{H}$  be the completion of
$\mathcal{D}([0,\infty)\times\R^d)$
endowed with the inner product
\begin{eqnarray}\label{innprod1}
\langle \varphi , \psi \rangle_{\mathcal{H}}&=&
\int_{\R_{+}^{2}\times\R^{2d}}
\varphi(s,x)\psi(t,y)\gamma(s,t)f(x-y) \, dxdydsdt\\ \notag
&=&\int_{\R_{+}^{2}\times\R^{d}}  \cf \varphi(s,\xi) \overline{ \cf \psi(t,\xi)}\gamma(s,t) \mu(d\xi) \, dsdt,
\end{eqnarray}
where $\cf \varphi$ refers to the Fourier transform with respect to the space variable only. We remind that $\gamma(s,t)=s^{a_1}t^{a_1}|s-t|^{a_2}$.
 The mapping $\varphi \rightarrow W(\varphi)$ defined in $\mathcal{D}([0,\infty)\times\R^d)$  extends to a linear isometry between
$\mathcal{H}$ and the Gaussian space
spanned by $W$ in $L^2(\Omega)$. We will denote this isometry by
\begin{equation*}
W(\phi)=\int_0^{\infty}\int_{\R^d}\phi(t,x)W(dt,dx)
\end{equation*}
for $\phi \in \mathcal{H}$. Notice that if $\phi$ and $\psi$ are in
$\mathcal{H}$, then
$\be \lc W(\phi)W(\psi)\rc =\langle\phi,\psi\rangle_{\mathcal{H}}$. Furthermore, $\mathcal{H}$  contains
the class of measurable functions $\phi$ on $\R_+\times
\R^d$  such that
\begin{equation}\label{abs1}
\int_{\R^2_+  \times\R^{2d}} |\phi(s,x)\phi(t,y)| \, \gamma(s,t)f(x-y) \, dxdydsdt <
\infty\,.
\end{equation}

\smallskip

We remind that the measure $\mu$ satisfies Dalang's condition (Assumption~\ref{assumption:finite var}). 

The proof is in Chapter 2. Here a necessary and sufficient condition for $\mu$ is available for the wave equation with colored noise, see \cite{Xiachen2022}. But we will use the sufficient condition (Dalang's condition) in this dissertation, even though it is not optimal. 

Let's calculate the $E[W(\varphi)^2]$ for the general fractional noise as in \eqref{varianceforDO}.
\begin{align}
\no E[W(\varphi)^2]& = \int_0^t\int_0^t\int_{\mathbb{A}}\int_{\mathbb{A}}(sr)^{a_1}|s-r|^{a_2}f(y-z)dydzdsdr\\
\no & =C_\beta \int_0^t ds\, s^{2a_1+a_2+1}   \int_{\mathbb{A}}\int_{\mathbb{A}}f(y-z)dydz \\
 &=C(\mathbb{A})C'_\beta t^{2H}.
\end{align}
where $C_\beta, C'_\beta$ are universal constants indepedent of $t$.
The second equality comes from \eqref{generalnoisecalc}, we can see that the generalized fractional noise have the same order $t^{2H}$ in time for the variance.

\section{Elements of Malliavin calculus}
We have introduced $W$ as an isonormal Gaussian process. Specifically, on a complete probability space $(\Omega, \cF, P)$, let $W=\{W(\varphi), \varphi\in \cH\}$ be a Gaussian family with  covariance given by
\begin{equation}\label{e:cov'}
\be[W(\varphi) W(\phi)]=\langle \varphi, \phi\rangle_\cH. 
\end{equation}
Then $W(\varphi)$ for $\varphi\in\cH$ is called the Wiener integral of $\varphi$ with respect to $W$ and we also denote $\int_{\R_+}\int_{\R^d} \varphi(t,x) W(dt, dx):=W(\varphi)$, see also \eqref{varphidef}. Our goal is to define the integral with respect to a random function $\varphi$ with Malliavin Calculus.\\
Before we move on to Malliavin Calculus, we should provide a brief reminder of It\^o Calculus and its limitation. 

Let $(W_t)_{t\geq 0}$ be a Brownian motion. If $\varphi$ is a deterministic function in $L^2(\bR_+)$, then the Ito integral
\begin{align*}
    W(\varphi)=\int_0^\infty \varphi(t)\,dW_t
\end{align*}
is well defined. 

For a simple process $X_t$, if we want the Ito integral $\int_0^T X_t\,dW_t$ to be well defined, we need two assumptions:
\begin{itemize}
    \item $\displaystyle \int_0^T E[X_t^2]dt<\infty$.
    \item $X_t$ is previsible. (i.e., essentially it is $F_{t-}$ measurable.)
\end{itemize}
Then, for a simple process $X_t=\sum_{i=1}^n c_i \mathbbm{1}_{(t_i,t_{i+1}]}(t)$, with $c_i$ an $F_{t_i}-$ measurable random variable, we define 
\begin{align*}
    \int_0^T X_t dW_t := \sum_{i=1}^n c_i (W_{t_{i+1}}-W_{t_i}).
\end{align*}
In this situation, both Ito calculus and the upcoming Malliavin calculus will give you the same result. However, if we considered $(W_t)$ to be fractional Brownian motion, we would not have independent increments, which would usually make the second assumption fail. This is why we will need the more general Malliavin Calculus framework to cover this situation.\\

We first need to define the Malliavin derivative $DF$. Let $h(x_1, \dots, x_n)$ be a smooth function such that its partial derivatives have at most polynomial growth. Then for smooth and cylindrical random variables of the form $F=h(W(\varphi_1), \dots, W(\varphi_n))$, the Malliavin derivative $DF$ is the $\cH$-valued random variable 
\[ DF:=\sum_{k=1}^n\frac{\partial h}{\partial x_k}(W(\varphi_1),\dots, W (\varphi_n)) \,\varphi_k. \]
 $D:L^2(\Omega)\to L^2(\Omega; \cH)$ is a closable operator.\\ 
The next step is to define the adjoint operator of the Malliavin derivative, the divergence operator, also known as the Skorohod integral. First we define the Sobolev space $\mathbb D^{1,2}$ as the closure of the space of the smooth and cylindrical random variables under the norm
\[\|F\|_{1,2}=\sqrt{\be[F^2]+\be[\|DF\|_\cH^2]}. \]

We denote by $\text{Dom } \delta$  the domain of the divergence operator $\delta$, which is the set of  $u\in L^2(\Omega; \cH)$ such that  $|E[\langle DF, u\rangle_\cH]|\le c_F \|F\|_2$ with some constant $c_F$ depending on $F$, for all $F\in \mathbb D^{1,2}$. The divergence operator $\delta$ (also known as the Skorohod integral) is the adjoint of the Malliavin derivative operator $D$ defined by the duality  
\begin{equation}\label{e:duality}
\be[F\, \delta(u)]=E[\langle DF, u\rangle_\cH]~ \text{ for all } F\in \mathbb D^{1,2} \text{ and } u\in \text{Dom } \delta.
\end{equation}
 Thus, for $u\in \text{Dom } \delta$, we have $\delta(u)\in L^2(\Omega)$. Also we noticed from \eqref{e:duality} with $F\equiv1$ that we have $\be[\delta(u)]=0.$  We will use the following notation
 \begin{equation}\label{e:skorohod-int}
 \delta(u)=\int_{\R_+}\int_{\bR^d} u(t,x) W(dt,dx) 
 \quad\text{for all}\quad
 u\in \text{Dom } \delta.
 \end{equation}

Lastly, we will introduce the Wiener chaos expansion of a random variable $F\in L^2(\Omega)$. Let $\mathbf H_0=\R$, and for integers $n\ge1$, let $\mathbf H_n$ be the closed linear subspace of $L^2(\Omega)$ containing the set of random variables $\left\{H_n(W(\varphi)), \varphi\in\cH, \|\varphi\|_{\cH}=1\right\}$, where $H_n$ is the $n$-th Hermite polynomial (i.e., $H_n(x)=(-1)^ne^{x^2}\frac{d^n}{dx^n}(e^{-x^2})$).  Then $\mathbf H_n$ is called the $n$-th Wiener chaos of $W$.  Assuming $\mathcal F$ is the $\sigma$-field generated by $\{W(\varphi), \varphi\in \cH\}$, we have the following Wiener chaos decomposition
 \[
  L^2(\Omega, \mathcal F, P)=\bigoplus_{n=0}^\infty \mathbf H_n,
  \] 
where we have the orthogonality of each $\mathbf H_n$. For $n\ge1$, let $\cH^{\otimes n}$ be the $n$-th tensor product of $\cH$ and $\widetilde \cH^{\otimes n}$ be the symmetrization of $\cH^{\otimes n}$. Then the mapping $I_n$ defined by $I_n(h^{\otimes n}):=H_n(W(h))$ for $h\in \cH$ can be extended to a linear isometry between $\widetilde\cH^{\otimes n}$ and the $n$-th Wiener chaos $\mathbf H_n$. Thus, for any random variable $F\in L^2(\Omega, \mathcal F, P)$,  the following unique Wiener chaos expansion in $L^2(\Omega)$ holds true, 
 \[F=\be[F]+\sum_{n=1}^\infty I_n(f_n) \quad \text{for some } f_n\in \widetilde\cH^{\otimes n}.\]
Furthermore, noting that \[\be\left[\left |I_n(f_n)\right|^2\right]=n! \|f_n\|^2_{\cH^{\otimes n}},\]
 we have 
\begin{equation}\label{e:EF2}
\be[|F|^2]=\left(\be[F]\right)^2+\sum_{n=1}^\infty\be\left[\left |I_n(f_n)\right|^2\right]=\left(\be[F]\right)^2+\sum_{n=1}^\infty n! \|f_n\|^2_{\cH^{\otimes n}}.
\end{equation}

\section{Mild solution in Skorohod sense}
In this subsection, we define the mild Skorohod solution to \eqref{eqn:DO process} with noise $\dot{W}$ defined in \eqref{generalfracnoise} and derive its  Wiener chaos expansion. 

In the sequel we write $\{\cF_t\}_{t\ge 0}$ for the filtration generated by the time increments of $\dot W$. That is we set 
\[\cF_t=\sigma\{W(\mathbbm{1}_{[0,s]} \varphi); \,  0\le s \le t, \, \varphi\in \mathcal D(\R)\}\vee \mathcal N,\]
where $\mathcal N$ is the collection of null sets.  

The Skorohod integral (or divergence) of a random field $u$ can be
computed by  using the Wiener chaos expansion. More precisely,
suppose that $u=\{u(t,x) ; (t,x) \in \R_+ \times\R^d\}$ is a random
field such that for each $(t,x)$, $u(t,x)$ is an
$\mathcal{F}^W$-measurable and square integrable random  variable.
Then, for each $(t,x)$ we have a Wiener chaos expansion of the form
\begin{equation}  \label{exp1}
u(t,x) = \be \lc u(t,x) \rc + \sum_{n=1}^\infty I_n (f_n(\cdot,t,x)).
\end{equation}
Suppose also that
\[
\be \lc \int_0^\infty \int_0^\infty \int_{\R^{2d}}  |u(t,x) \, u(s,y)
| \, \gamma(s,t) f(x-y) \, dxdydsdt \rc <\infty.
\]
Then, we can interpret $u$ as a square  integrable
random function with values in $\mathcal{H}$ and the kernels $f_n$
in the expansion (\ref{exp1}) are functions in $\mathcal{H}
^{\otimes (n+1)}$ which are symmetric in the first $n$ variables. In
this situation, $u$ belongs to the domain of the divergence (that
is, $u$ is Skorohod integrable with respect to $W$) if and only if
the following series converges in $L^2(\Omega)$
\begin{equation}\label{eq:delta-u-chaos}
\delta(u)= \int_0 ^\infty \int_{\R^d}  u_{t,s} \, \delta W_{t,x} = W(\be[u]) + \sum_{n=1}^\infty I_{n+1} (\widetilde{f}_n(\cdot,t,x)),
\end{equation}
where $\widetilde{f}_n$ denotes the symmetrization of $f_n$ in all its $n+1$ variables.

\begin{definition}\label{def:mild-skorohod}
An $\{\cF_t\}_{t\ge 0}$-adapted random field $u=\{u(t,x), t\ge0, x\in \R^d\}$ is called a mild Skorohod solution to \eqref{eqn:DO process} if $\be[u^2(t,x)]<\infty$ for all $(t,x)\in \R_+\times \R^d$ and if it satisfies the following integral equation
\begin{equation}\label{e:solution}
u(t,x)=w(t,x)+\int_0^t \int_{\R^d} p_h (t-s,x-y) u(s,y) W(ds,dy). 
\end{equation}
where $w(t,x)$ is given in \eqref{e:h} for the heat equation. The stochastic integral on the right-hand side is a Skorohod integral as in \eqref{e:skorohod-int}. In particular, it is implicitly assumed that for each $t\ge 0, x\in \R$, the process $v_{t,x}(s,y)=p_h(t-s,x-y) u(s,y) \mathbbm{1}_{[0,t]}(s)$ lies in Dom $\delta$. See equation~\eqref{e:duality}.
\end{definition}

Let us say a few words about the chaos decomposition for the mild solution \eqref{e:solution}. First for $n\in \mathbb N$ we denote
\begin{align}\label{e:fn}
\lefteqn{f_n(s_1,x_1,\dots, s_n, x_n, t,x)}\\
\no&=\frac{1}{n!}\sum_{\sigma\in\Sigma_n} p_h(t-s_{\sigma(n)},x-x_{\sigma(n)})\cdots p_{h}(s_{\sigma(2)}-s_{\sigma(1)},x_{\sigma(2)}-x_{\sigma(1)})\\
\no&\phantom{\frac{1}{n!}}\times
w(s_{\sigma(1)}, x_{\sigma(1)})\mathbbm{1}_{[0<s_{\sigma(1)}<\dots<s_{\sigma(n)}<t]},
\end{align}
where $\Sigma_n$ is the set of permutations on $\{1, 2,\dots, n\}$, $p_h$ is the heat kernel defined by \eqref{heatkernel} and $w$ is the convolution \eqref{e:h}. Then the following result can be found, for instance, in \cite{Balan2016}. 

\begin{proposition}\label{prop:chaos}
There exists a unique mild Skorohod solution to \eqref{eqn:DO process} if and only if the  function $w(t,x)$ given by \eqref{e:h} is well-defined and the series   $\sum_{n=1}^\infty I_n(f_n(\cdot, t, x))$ converges in $L^2(\Omega)$ for all $t>0$, i.e.,  
\begin{equation}\label{e:chaos}
\sum_{n=1}^\infty n!\|f_n(\cdot, t,x)\|^2_{\cH^{\otimes n}}<\infty, ~ \text{ for all } t >0 \text{ and } x\in \R^d.
\end{equation}
Whenever \eqref{e:chaos} is met, we have the following Wiener chaos expansion for the solution $u$ to equation~\eqref{eqnwave:DO process}:
\begin{equation}\label{e:u-chaos}
u(t,x)=w(t,x)+\sum_{n=1}^\infty I_n(f_n(\cdot, t, x)).  
\end{equation}
\end{proposition}

In order to state some necessary and sufficient conditions allowing to solve \eqref{eqnwave:DO process}, we will restrict our attention to the case where the initial condition $u_0$ is constant. Without loss of generality, we can assume that $u_0 = 1$. Thus, we will have $w(t,x) = 1$, and
\begin{align}\label{e:fn'}
\no\lefteqn{f_n(s_1, x_1, \dots, s_n, x_n, t,x)}\\
\no &=  \frac{1}{n!}\sum_{\sigma\in\Sigma_n}  p_{h}(t-s_{\sigma(n)},x-x_{\sigma(n)})\cdots p_{h}(s_{\sigma(2)}-s_{\sigma(1)},x_{\sigma(2)}-x_{\sigma(1)})\\
&\phantom{\frac{1}{n!}}\times\mathbbm{1}_{[0<s_{\sigma(1)}<\dots<s_{\sigma(n)}<t]}. 
\end{align}

A Picard iteration for $u$ would yield:
\begin{align}
\label{picardscheme}
    u_{n+1}(t,x)&=w(t,x)+\int_0^t\int_{\bR^d} p_h(t-s,x-y)u_n(s,y) W(ds,dy).
\end{align}
Here is a brief justification. Let $v_n:=u_n-u_{n-1}$, and define $v_0=u_0$, we must have
\begin{align}
v_{n+1}(t,x)&=u_{n+1}(t,x)-u_n(t,x)=\int_0^t\int_{\bR^d} p_h(t-s,x-y)v_n(s,y) W(ds,dy).
\end{align}
Since $u=\lim_{n\rightarrow \infty} u_n$ and $u=\sum_{n=0}^\infty I_n(f_n)$  we have
\begin{align}
    I_n(f_n)=v_n \hspace{5mm}\text{and}\hspace{5mm} u_n=\sum_{k=0}^n I_k(f_k).
\end{align}
Now let's derive a formula for $f_n$. We have, with $u_0\equiv 1$,
\begin{align}
    v_1(t,x)&=\int_0^t\int_{\bR^d} p_h(t-s,x-y)u_0(s,y) W(ds,dy)=I_1(f_1),
\end{align}
and 
\begin{align*}
\lefteqn{v_2(t,x)}\\
&=\int_0^t\int_{\bR^d} p_h(t-s_1,x-y_1)v_1(s_1,y_1) W(ds_1,dy_1)\\
&= \int_0^t\int_{\bR^d} \int_0^{s_1}\int_{\bR^d} p_h(t-s_1,x-y_1)p_h(s_1-s_2,y_1-y_2) W(ds_1,dy_1)W(ds_2,dy_2)\\
&=\int_0^t\int_{\bR^d} \int_0^{t}\int_{\bR^d} \frac{1}{2}\, p_h(t-s_1,x-y_1)p_h(s_1-s_2,y_1-y_2) W(ds_1,dy_1)W(ds_2,dy_2)\\
&=I_2(f_2).
\end{align*}
A similar argument by induction establishes \eqref{picardscheme}.\\
Now we are ready to prove the existence and uniqueness of a solution for \eqref{eqn:DO process} in the next chapter.

\chapter{Existence and uniqueness of the solution}
In this chapter we are going to prove the existence and uniqueness of the solution to \eqref{eqn:DO process} driven by the general noise defined in Chapter 5.

For the sake of simplicity, we will establish existence and uniqueness in the case of $u_0\equiv 1$. The result would hold true for general bounded functions as well.

\section{Main result}
\begin{theorem}\label{thmSk1}
Suppose that the spectral measure is given by $\mu(d\xi) =C_\alpha|\xi|^{\alpha-d}d\xi$
and $\gamma(s_i,r_i)=s_i^{a_1}r_i^{a_1}|s_i-r_i|^{a_2}$. Then relation (\ref{e:chaos}) holds
for each $(t,x)$. Consequently, equation \eref{eqn:DO process} admits a
unique mild solution in the sense of Definition \ref{def:mild-skorohod}. Namely:
\begin{align}
u(t,x)=w(t,x)+\int_0^t \int_{\R^d} p_h (t-s,x-y) u(s,y) W(ds,dy). 
\end{align}
\end{theorem}

\begin{proof}
Fix $t>0$ and $x\in \mathbb{R}^d$. Set
$f_n(s,y,t,x)=f_n(s_1,y_1,\dots,s_n,y_n,t,x)$ and $\mu(d\xi)\equiv\prod_{i=1}^n \mu(d\xi_i)$.
Using the Fourier transform and Cauchy-Schwarz inequality, we obtain
\begin{eqnarray*}   
&&n!\|f_n(\cdot,t,x)\|^2_{\mathcal{H}^{\otimes n}}   \\  
 &\leq &
n! \int_{[0,t]^{2n}} \int_{\R^{nd}}
\mathcal{F}f_n(s,\cdot,t,x)(\xi)
\overline{\mathcal{F}f_n(r,\cdot,t,x)(\xi)}  \mu(d\xi)\prod_{i=1}^n
s_i^{a_1}r_i^{a_1}|s_i-r_i|^{a_2}dsdr   \\ 
 &\leq&n! 
\int_{[0,t]^{2n}}\left (\int_{\R^{nd}}
|\mathcal{F}f_n(s,\cdot,t,x)(\xi)|^2\prod_{i=1}^n s_i^{2a_1}\mu(d\xi)\right)^{\frac{1}{2}} 
 \\
&& \times \left(\int_{\R^{nd}}|\mathcal{F}f_n(r,\cdot,t,x)(\xi)|^2\prod_{i=1}^n r_i^{2a_1}\mu(d\xi)\right)^{\frac{1}{2}}  \prod_{i=1}^n |s_i-r_i|^{a_2}ds dr   \\
&\leq&n! 
\int_{[0,t]^{2n}}\left (\int_{\R^{nd}}
|\mathcal{F}f_n(s,\cdot,t,x)(\xi)|^2\mu(d\xi)\right)^{\frac{1}{2}} 
\left(\int_{\R^{nd}}|\mathcal{F}f_n(r,\cdot,t,x)(\xi)|^2\mu(d\xi)\right)^{\frac{1}{2}}  \\
&& \times \prod_{i=1}^n s_i^{a_1}r_i^{a_1}|s_i-r_i|^{a_2}ds dr,   \\
\end{eqnarray*}
and thus, thanks to the basic
inequality $ab\leq \frac{1}{2}(a^2+b^2)$
  and the fact that $\sup_{0<s<t}\int_0^t\gamma(s,r)\,dr$ (See Appendix Lemma~\ref{gammaintegrable}) is finite, this yields:
\begin{eqnarray*}   
n!\|f_n(\cdot,t,x)\|^2_{\mathcal{H}^{\otimes n}} &\leq &   
n! 
\int_{[0,t]^{2n}}\int_{\R^{nd}}  |\mathcal{F}f_n(s,\cdot,t,x)(\xi)|^2\mu(d\xi) \prod_{i=1}^n s_i^{a_1}r_i^{a_1}|s_i-r_i|^{a_2}ds dr  \notag \\
&\le&  C(t)^nn!   \label{eq3}
\int_{[0,t]^{n}}\int_{\R^{nd}}  |\mathcal{F}f_n(s,\cdot,t,x)(\xi)|^2\mu(d\xi)  ds,
\end{eqnarray*}
where $\displaystyle C(t)=2\sup_{0<s<t} \int_0^t \gamma(s,r)dr$.\\
From here, it remains to control the behavior in space. This is done in a similar way as in Hu-Huang-Nualart-Tindel\cite{HuHuangNualartTindel} for fractional noise. We provide details for the sake of completeness. It is readily checked from expression \eqref{e:fn'} that there exists a constant $C>0$ such that the Fourier transform of $f_n$ satisfies
\[
|\mathcal{F}f_n(s,\cdot,t,x)(\xi)|^2=\frac{C^n}{(n!)^2}
\prod_{i=1}^n e^{-(s_{\si(i+1)}-s_{\si(i)})|\xi_{\si(i)}+\cdots +
\xi_{\si(1)}|^2},
\]
where we have set $s_{\si(n+1)}=t$.  As a consequence,
\begin{eqnarray}   \notag
&& \int_{\R^{nd}}|\mathcal{F}f_n(s,\cdot,t,x)(\xi)|^2\mu(d\xi)\\
\notag &\leq&  \frac{C^n}{(n!)^2}
 \prod_{i=1}^n \sup_{\eta \in
\R^d}\int_{\R^d}e^{-(s_{\sigma(i+1)}-s_{\sigma(i)})|\xi_{\sigma(i)}+\eta|^2}\mu(d\xi_{\sigma(i)})\\
 &\leq&  \frac{C^n}{(n!)^2}
 \prod_{i=1}^n \int_{\R^d}e^{-(s_{\sigma(i+1)}-s_{\sigma(i)})|\xi_{\sigma(i)}|^2}\mu(d\xi_{\sigma(i)}),
 \end{eqnarray}
The inequality holds true since the sup will take place at $\eta=0$. This yields
\begin{align}
\label{eq4}
\no n!\|f_n(\cdot,t,x)\|^2_{\mathcal{H}^{\otimes n}}&\leq  \frac{C^n}{n!} \int_{\R^{nd}}  \int_{[0,t]^n}\prod_{i=1}^n e^{-(s_{i+1}-s_i)|\xi_i|^2} \, ds \, \mu(d\xi)\,\\
&\leq  C^n \int_{\R^{nd}}  \int_{T_n(t)}\prod_{i=1}^n e^{-(s_{i+1}-s_i)|\xi_i|^2} \, ds \, \mu(d\xi)\,,
\end{align}
where we denote by $T_n(t)$ the simplex
\begin{equation} \label{simplex}
 T_n(t) =\{0<s_1<\cdots <s_n<t\}.
 \end{equation}
Let us now estimate the right hand side of \eref{eq4}:
making the change of variables $s_{i+1}-s_{i}=w_i$ for $1\leq i \leq n-1$, and $t-s_n=w_n$,  and denoting $dw=dw_1 dw_2 \cdots dw_n$, we end up with
\begin{eqnarray*}
 n!\|f_n(\cdot,t,x)\|^2_{\mathcal{H}^{\otimes n}}  \leq  C^n \int_{\R^{nd}}\int_{S_{t,n}}e^{- \sum_{i=1}^n w_i |\xi_i|^2}
 dw \prod_{i=1}^n \mu(d\xi_i)\,,
\end{eqnarray*}
where $S_{t,n}=\{(w_1, \dots, w_n)\in [0,  \infty)^n: w_1 +\cdots +w_n \le t\}$. 
We also split the contribution of $\mu$ in the following way: 
fix $N\ge 1$ and set
\begin{equation}  \label{eqD}
C_N=\int_{|\xi|\geq N}\frac{\mu(d\xi)}{|\xi|^2},
\quad\text{and}\quad
D_N=\mu\{\xi\in \R^d: |\xi|\leq N\}.
\end{equation}
By Lemma \ref{lem1} below, we can write
\begin{equation}\label{eq:eqDD}
 n!\|f_n(\cdot,t,x)\|^2_{\mathcal{H}^{\otimes n}}  \leq 
 C^n \sum_{k=0}^n \binom{n}{k} \frac{t^k}{k!}D_N^k
(2C_N)^{n-k}\,.
\end{equation}
  Next we choose a sufficiently large $N$ such that
$2CC_N < 1$, which is possible because of Assumption~\ref{assumption:finite var}. Using the
inequality $\binom{n}{k} \leq 2^n$ for any positive integers $n$
and $0\leq k \leq n$,  we have
\begin{align*}
&\sum_{n=0}^{\infty}n!\|f_n(\cdot,t,x)\|^2_{\mathcal{H}^{\otimes n}}
\leq
 \sum_{n=0}^{\infty}C^n \sum_{k=0}^n \binom{n}{k}\frac {t^k}{k!}D_N^k (2C_N)^{n-k}\\
&\leq  \sum_{k=0}^{\infty}\sum_{n=k}^{\infty}C^n 2^n \frac{t^k}{k!}D_N^k (2C_N)^{n-k}
=
 \sum_{k=0}^{\infty}\frac{t^k}{k!}D_N^k (2C_N)^{-k}\sum_{n=k}^{\infty}(2CC_N)^n\\
&\leq
 \frac{1}{1-2C C_N} \sum_{k=0}^{\infty}\frac{t^k D_N^k
(2C_N)^{-k}(2CC_N)^k}{k!}< \infty\,.
\end{align*}
This proves the theorem.
\end{proof}

Next we establish the lemma that is used in the proof of Theorem \ref{thmSk1}.
\begin{lemma}[Hu-Huang-Nualart-Tindel\cite{HuHuangNualartTindel} Lemma 3.3 ]\label{lem1}
For any $N
> 0$ we let $D_N$ and $C_N$ be given by (\ref{eqD}). Then
we have
\[
 \int_{\R^{nd}}\int_{S_{t,n}}e^{- \sum_{i=1}^n w_i |\xi_i|^2}
 dw \prod_{i=1}^n \mu(d\xi_i)
\leq\sum_{k=0}^n \binom{n}{k}
\frac{t^k}{k!}D_N^k (2C_N)^{n-k}\,.
\]

\end{lemma}
\begin{proof}
The original proof was in \cite{HuHuangNualartTindel}, we have it here for completeness.
By our Assumption~\ref{assumption:finite var}, $C_N$ is finite for all
positive $N$. Let $I$ be a subset of $\{1,2,\dots,n\}$ and  $I^{c}
=\{1,2,\dots,n\}\setminus I$. Then we have
\begin{eqnarray*}
&&\int_{\R^{nd}}\int_{S_{t,n}}\prod_{i=1}^n e^{-w_i
|\xi_i|^2} \, dw \, \mu(d\xi)\\
&=&  \int_{\R^{nd}} \int_{ S_{t,n}}\prod_{i=1}^n e^{-w_i|\xi_i|^2}({\bf 1}_{\{|\xi_i|\leq N\}}+ {\bf 1}_{\{|\xi_i|> N\}})\, dw \, \mu(d\xi)\\
&=& \sum_{I \subset \{1,2,\dots,
n\}}\int_{\R^{nd}}\int_{ S_{t,n}}\prod_{i \in I} e^{-w_i
|\xi_i|^2}{\bf 1}_{\{|\xi_i|\leq N\}}\times \prod_{j \in I^{c}}
 e^{-w_j |\xi_j|^2} {\bf 1}_{\{|\xi_j|\geq N\}} \, dw \, \mu(d\xi).
 \end{eqnarray*}
 For the indices $i$ in the set $I$ we estimate $e^{-w_j |\xi_j|^2}$ by $1$. Then, using the inclusion
 \[
 S_{t,n}\subset S^I _{t} \times S^{I^c}_{t},
 \]
 where $S^I_{t} =\{(w_i ,i\in I): w_i\ge 0, \sum_{i\in I} w_i \le t\}$ and $S^{I^c}_{t} =\{(w_i ,i\in I^c): w_i\ge 0, \sum_{i\in I^c} w_i \le t\}$ we obtain
 \begin{eqnarray*}
&&\int_{\R^{nd}}\int_{S_{t,n}}\prod_{i=1}^n e^{-w_i
|\xi_i|^2}\, dw \, \mu(d\xi) \\
&\leq& \sum_{I \subset \{1,2,\cdots, n\}}\int_{\R^{nd}}\int_{ S^I_t\times S_t^{I^c}}\prod_{i \in I} {\bf 1}_{\{|\xi_i|\leq
N\}}\times \prod_{j \in I^{c}}
e^{-w_j |\xi_j|^2} {\bf 1}_{\{|\xi_j|\geq N\}} \, dw \, \mu(d\xi) .
\end{eqnarray*}
Furthermore, one can bound the integral over $S_t^{I^c}$ in the following way
\begin{equation*}
\int_{S_t^{I^c}} \prod_{j \in I^{c}} e^{-w_j |\xi_j|^2}  \, dw
\le
\int_{[0,t]^{I^c}} \prod_{j \in I^{c}} e^{-w_j |\xi_j|^2} \, dw
=
\prod_{j \in I^{c}} \frac{1-e^{-t |\xi_{j}|^{2}}}{|\xi_{j}|^{2}}
\le
\prod_{j \in I^{c}} \frac{1}{|\xi_{j}|^{2}}.
\end{equation*}
We can thus bound $\int_{\R^{nd}}\int_{S_{t,n}}\prod_{i=1}^n e^{-w_i |\xi_i|^2}\, dw \, \mu(d\xi)$ by:
\begin{multline*}
\sum_{I \subset \{1,2,\cdots,n\}}\frac{t^{|I|}}{|I|!}\big(\mu\{\xi\in \R^d: |\xi|\leq N\}\big)^{|I|}2^{|I^{c}|}\int_{|\xi_j|>N, \forall j \in I^{c}}\prod_{j \in I^{c}}\frac{\mu(d\xi_j)}{|\xi_j|^2}\\
= \sum_{I \subset
\{1,2,\cdots,n\}}\frac{t^{|I|}}{|I|!}D_N^{|I|}(2C_N)^{|I^{c}|}=
\sum_{k=0}^n \binom{n}{k} \frac{t^k}{k!}D_N^k (2C_N)^{n-k}\,,
\end{multline*}
which is our claim.
\end{proof}
\chapter{Moment formula for the solution}

In this section, we will have a quick introduction on the Feynman-Kac Formula for the heat equation that we will generalize to the Stochastic PDE. Further, we will see that a representation of the moments of the solution can be established, even in the situation where the solution itself does not admit a Feynman-Kac formula.

\section{The Feynman-Kac functional for the solution}
Consider the following heat equation:
\begin{align}
\label{heatfeynmankac}
    \frac{\partial u}{\partial t}=\frac{1}{2}\Delta u-V(t,x) u, \hspace{8mm} u(0,x)=u_0.
\end{align}
where $V(t,x)$ is a deterministic lower-bounded continuous function and $u_0$ is a constant.\\
The corresponding $u(t,x)$ with $V(t,x)=V(x)$ was shown by Kac\cite{kac} in 1949. The solution to \eqref{heatfeynmankac} is the following:
\begin{align}
    u(t,x)=E^{(0,x)}\Bigg\{u_0 e^{-\int_0^t V(t-s,B(s))ds}\Bigg\},
\end{align}
where $E^{(0,x)}$ is the conditional expectation given that the Brownian motion starts at $x$ at time $0$. This is also known as Feynman-Kac formula. Now let's sketch a heuristic derivation of it for Brownian motion. This was motivated by Kuo~\cite[Section 11.4]{kuo} and Karatzas and Shreve~\cite[Theorem 4.2]{KaratzaShreve}.

Let $B(t)$ be an $\mathbb{R}^n$-valued Brownian motion. Let $t\geq 0$ be fixed and consider the function $g(x)=u(t,x)$. The infinitestimal generator $\mathcal{A}g$ is defined to be the limit
\begin{align}
    (\mathcal{A}g)(x)=\lim_{\epsilon \downarrow 0}\frac{E^x[g(B_\epsilon)]-g(x)}{\epsilon}.
\end{align}
We know that for Brownian motion 
\begin{align}
    (\mathcal{A}g)(x)=\frac{1}{2}\Delta g(x).
\end{align}
Then we have
\begin{align}
    E^{(0,x)}g(B_\epsilon)=E^{(0,x)}\left[E^{(0,B_{\epsilon})}\Bigl[e^{\int_0^tV(t-s,B_s)\,ds}\Bigl]\right].
\end{align}
Using Markov property of the Brownian motion, we have
\begin{align}
\label{markovppty}
E^{(0,x)}\left[E^{(0,x)}\Bigl[e^{\int_0^tV(t-s,B_{s+\epsilon})\,ds}|\mathcal F_\epsilon\Bigl]\right]=E^{(0,x)}\left[e^{-\int_\epsilon ^{t+\epsilon}V(t+\epsilon-s,B_s)\,ds} \right].
\end{align}
Next, we reorder \eqref{markovppty} to get
\begin{align}
 \no   \lefteqn{E^{(0,x)}\left[e^{-\int_\epsilon ^{t+\epsilon}V(t+\epsilon-s,B_s)\,ds} \right]}\\
    &=E^{(0,x)}\left[e^{\int_0^{t+\epsilon}V(t+\epsilon-s,B_s)ds}\right] +E^{(0,x)}\left[e^{-\int_0^{t+\epsilon}V(t+\epsilon-s,B_s)ds}(e^{\int_0^\epsilon V(t+\epsilon-s,B_s)ds}-1)\right].
\end{align}
Lastly, we find
\begin{align}
\no   \lefteqn{(\mathcal{A}g)(x)}\\
\no   &=\lim_{\epsilon\downarrow 0}\frac{E^{(0,x)} [g(B_\epsilon )]-g(x)}{\epsilon}\\
\no   &=\lim_{\epsilon\downarrow 0}\frac{u(t+\epsilon,x)-u(t,x)}{\epsilon}+\lim_{\epsilon\downarrow 0} E^{(0,x)}\left[e^{-\int_0^{t+\epsilon} V(t+\epsilon-s,B_s)ds} \left( \frac{e^{\int_0^\epsilon V(t+\epsilon-s,B_s)ds}-1}{\epsilon}\right) \right]\\
&= \frac{\partial u}{\partial t}+u(t,x)V(t,x),
\end{align}
which solves equation \eqref{heatfeynmankac}.

When we consider the stochastic heat equation drive by space-time white noise, i.e. if $H=1/2$ and $a=1$ in equation~\eqref{eqn:DO process}, the Feynman-Kac functional can be extended. The functional below is well defined, and we have 
\begin{align}
    u(t,x)=E_B\Bigg{[}u_0(B_t^x)\exp \Bigg(\int_0^t\int_{\mathbb{R}^d}\delta_0(B_{t-r}^x-y)W(drdy)\Bigg)\Bigg{]},
\end{align}
where $B^x$ is a $d$-dimensional Brownian motion independent of $W$ and starting at $x\in\mathbb{R}^d$. This result can be found in \cite{HuHuangNualartTindel}.

However, for the generalized noise, the Feynman-Kac formula will possibly fail because of the correlation of the increments of the noise. Although we don't have a formula for the solution itself anymore, we can develop a Feynman-Kac type formula for the $n$-th moment expectation by first smoothing the solution. 
\section{Feynman-Kac formula for the moments}\label{sec:FK-moments}
Our next objective is to find a formula  for the moments of the mild solution to equation \eref{eqn:DO process}.
For any $\delta > 0$, we define the function
$\varphi_{\delta}(t)=\frac{1}{\delta}{\bf 1}_{[0,\delta]}(t)$ for $t
\in \R$.  Then, $\varphi_{\delta}(t)p_{\varepsilon}(x)$  provides an
approximation of the Dirac delta function $\delta_0(t,x)$ as
$\varepsilon$ and $\delta$ tend to zero.

We set
\begin{equation}\label{regW}
\dot{W}^{\varepsilon,\delta}_{t,x}=\int_0^t
\int_{\R^d}\varphi_{\delta}(t-s)p_{\varepsilon}(x-y)W(ds,dy)\,.
\end{equation}
Now we consider the approximation of equation \eref{eqn:DO process} defined
by
\begin{equation}\label{approx}
\frac{\partial u_{t,x}^{\varepsilon,\delta}}{\partial t}=\frac{1}{2}\Delta u_{t,x}^{\varepsilon,\delta}+\int_{0}^{s}\int_{\mathbb{R}^d}\varphi _{\delta }(s-r)p_{{\varepsilon }%
}(y-z)u_{s,y}^{\varepsilon ,\delta }\delta W_{r,z}\,.
\end{equation}

Here we take the integral as a Skorohod type integral. Furthermore, the mild or evolution version of  (\ref{approx})  is
\begin{equation}
u_{t,x}^{\varepsilon ,\delta }=u_0(x)+\int_{0}^{t}\int_{\mathbb{R}%
^d}p_{t-s}(x-y)(\int_{0}^{s}\int_{\mathbb{R}^d}\varphi _{\delta }(s-r)p_{{\varepsilon }%
}(y-z)u_{s,y}^{\varepsilon ,\delta }\delta W_{r,z}\,)dsdy.  \label{eq6}
\end{equation}%
Applying Fubini's
theorem yields
\begin{equation}
u_{t,x}^{\varepsilon ,\delta }=u_0(x)+\int_{0}^{t}\int_{\mathbb{R}%
^d}\left( \int_{0}^{t}\int_{\mathbb{R}^d}p_{t-s}(x-y)\varphi _{\delta
}(s-r)p_{{\varepsilon }}(y-z)u_{s,y}^{\varepsilon ,\delta }dsdy\right)
\delta W_{r,z}.  \label{eq7}
\end{equation}
This leads to the following definition.

\begin{definition}  \label{def2}
An adapted random field $u^{\varepsilon ,\delta }=\{u_{t,x}^{{\varepsilon
,\delta }};  t\geq 0,x\in \mathbb{R}^{d}\}$ is a mild solution to equation (%
\ref{approx}) if for each $(r,z)\in  [0,t]\times \mathbb{R}^d$ the
integral
\begin{equation*}
Y_{r,z}^{t,x}= \int_{0}^{t}\int_{\mathbb{R}^d}p_{t-s}(x-y)\varphi
_{\delta }(s-r)p_{{\varepsilon }}(y-z)u_{s,y}^{\varepsilon ,\delta
}dsdy
\end{equation*}%
exists and $Y^{t,x}$ is a Skorohod integrable process such that (\ref{eq7})
holds for each $(t,x)$.
\end{definition}

Notice that the above definition is equivalent to saying that $u_{t,x}^{\varepsilon
,\delta }\in L^{2}(\Omega )$, and for any random variable $F\in \mathbb{D}%
^{1,2}$ , we have%
\begin{equation}
 \be \lc Fu_{t,x}^{\varepsilon ,\delta }\rc  =\be \lc F\rc u_0(x)+\be \lc \langle  Y^{t,x},DF\rangle _{\mathcal{H}}\rc.  \label{eq8}
\end{equation}%

In order to derive a Feynman-Kac formula for the moment of order $k\ge 2$  of the solution to equation (\ref{eqn:DO process}) we need to introduce $k$  independent $d$-dimensional Brownian motions $B^j$, $j=1,\dots, k$, which are independent of the noise $W$ driving the equation. We shall study the probabilistic behavior of some random variables with double randomness. Hence, we denote by $\bp,\be$ the probability and expectation with respect
to the annealed randomness concerning the couple $(B,W)$, where
$B=(B^1, \dots, B^k)$, while we set respectively  $\be_{B}$ and
$\be_{W}$ for the expectation with respect to one randomness only.

With this notation in mind, define
\begin{equation}\label{eq9}
u_{t,x}^{\varepsilon,\delta}=\be_B \lc \exp \lp  W (
A_{t,x}^{\varepsilon,\delta})-\frac{1}{2}\alpha^{\varepsilon,\delta}_{t,x}\rp
\rc\,,
\end{equation}
where
\begin{equation}  \label{m3}
A_{t,x}^{\varepsilon,\delta}(r,y)=\frac 1\delta
\left(\int_0^{\delta \wedge (t-r)}
p_{\varepsilon}(B_{t-r-s}^x-y)ds\right) \mathbf{1}_{[0,t]} (r),
\quad\text{and}\quad
\alpha^{\varepsilon,\delta}_{t,x}=\|A^{\varepsilon,\delta}_{t,x}\|^2_{\mathcal{H}},
\end{equation}
for a standard $d$-dimensional Brownian motion $B$ independent of
$W$. Then one can prove that $u_{t,x}^{\varepsilon,\delta}$ is a
mild solution to equation (\ref{approx}) in the sense of Definition
\ref{def2}. The proof is similar to the proof of Proposition 5.2 in
\cite{HN}, the only change we have here is that we use the inner product form $\langle\cdot,\cdot\rangle_H$ defined in \eqref{innprod1}, in place of the inner product used in \cite{HN}. We omit the details.

The next theorem asserts that the random variables
$u_{t,x}^{\varepsilon,\delta}$ have moments of all orders, uniformly
bounded in $\varepsilon$ and $\delta$, and converge to the mild
solution of equation \eref{eqn:DO process}, as $\delta$ and $\varepsilon$ tend to zero. Moreover,
it provides an expression for the moments of the mild solution of
equation \eref{eqn:DO process}.

\begin{theorem}\label{SHEkthmoment}
Suppose $\gamma(s,r)=(sr)^{a_1}|s-r|^{a_2}$ and $\mu(d\xi) =C_\alpha|\xi|^{\alpha-d}d\xi$.  Then for any integer $k \geq 1$ we have
\begin{equation}\label{eq10}
\sup_{\varepsilon,\delta}\be \lc |u_{t,x}^{\varepsilon,\delta}|^k\rc< \infty\,,
\end{equation}
the limit $\lim_{\varepsilon \downarrow 0}\lim_{\delta
\downarrow 0} u_{t,x}^{\varepsilon,\delta}$ exists in $L^p$ for all
$p \geq 1$, and it coincides with the mild solution $u$ of equation
\eref{eqn:DO process}. Furthermore, we have for any  integer $k \geq 2$
\begin{equation}\label{momSk1}
\be \lc u_{t,x}^k\rc =u_0^k\, \be_B \lc  \exp\left(\sum_{1 \leq i < j
\leq k}\int_0^t \int_0^t \gamma (s,r)f(B_s^i-B_r^j)ds
dr\right)\rc\,,
\end{equation}
where $\{B^j; \, j=1,\dots, k\}$  is a family of $d$-dimensional  independent standard Brownian motions independent of $W$.
\end{theorem}

\begin{proof}
To simplify the proof we assume without loss of generality that $u_0$ is identically one. Fix an integer $k \geq 2$. Using \eref{eq9} we have
\begin{equation*}
\be \lc \lp u_{t,x}^{\varepsilon,\delta}\rp ^k\rc=\be_W \lc\prod_{j=1}^k
\be_B\lc \exp \lp   W(A^{\varepsilon,\delta,
B^j}_{t,x})-
\frac{1}{2}\alpha_{t,x}^{\varepsilon,\delta,B^j}\rp \rc \rc\,,
\end{equation*}
where for any $j=1,\dots,k$,  $A_{t,x}^{\varepsilon,\delta,B^j}$ and $\alpha_{t,x}^{\varepsilon,\delta,B^j}$ are evaluations of  \eqref{m3} using the Brownian motion $B^j$. Therefore,  since $W(A^{\varepsilon,\delta, B^j}_{t,x})$ is a Gaussian random variable conditionally on $B$, we obtain
\begin{eqnarray}\label{eq:exp-moments-utx-ep-delta}
\be \lc \lp u_{t,x}^{\varepsilon,\delta}\rp ^k\rc &=&
\be_B \lc 
\exp \lp\frac{1}{2}\|\sum_{j=1}^k A_{t,x}^{\varepsilon,\delta,B^j}\|^2_{\mathcal{H}}
-\frac{1}{2}\sum_{j=1}^k \alpha_{t,x}^{\varepsilon,\delta,B^j}\rp\rc \notag\\
&=& \be_B \lc 
\exp \lp\frac{1}{2}\|\sum_{j=1}^k A_{t,x}^{\varepsilon,\delta,B^j}\|^2_{\mathcal{H}}
-\frac{1}{2}\sum_{j=1}^k \| A_{t,x}^{\varepsilon,\delta,B^j}\|^2_{\mathcal{H}}\rp\rc   \notag\\
&=&\be_B \lc \exp \lp\sum_{1\leq i < j \leq k}\langle
A_{t,x}^{\varepsilon,\delta,B^i},
A_{t,x}^{\varepsilon,\delta,B^j}\rangle _{\mathcal{H}}\rp\rc\,.
\end{eqnarray}
Here, we used the fact that $E[e^{W(\varphi)}]=e^{\frac{1}{2}\norm{\varphi}_H^2}$ and for multiple Gaussian random variables $W_1,...,W_k$, we have $E[\prod_{j=1}^k e^{W_j}]=e^{\sum_{i,j} \frac{1}{2} \text{Cov}(W_i,W_j)}$.

Let us now evaluate the quantities $\langle A_{t,x}^{\varepsilon,\delta,B^i}, A_{t,x}^{\varepsilon,\delta,B^j}\rangle _{\mathcal{H}}$ above. By the definition of $A_{t,x}^{\varepsilon,\delta,B^i}$, for any $i\not= j$  we have
\begin{equation} \label{eq11}
\langle A_{t,x}^{\varepsilon,\delta,B^i},A_{t,x}^{\varepsilon,\delta,B^j}\rangle_{\mathcal{H}} =
 \int_0^t \int_0^t \int_{\R^{d}} \cf A^{\ep,\delta, B^i}_{t,x} (u,\cdot)(\xi) \, \overline{\cf A}^{\ep,\delta,B^j}_{t,x}(v,\cdot) (\xi)\gamma(u,v) \mu(d\xi) dudv.
 \end{equation}
 On the other hand, for $u\in [0,t]$ we can write
 \begin{eqnarray*}
\cf A^{\ep,\delta, B^i}_{t,x}(u,\cdot)(\xi)
&=&  \frac 1\delta \int_0 ^{ \delta \wedge(t-u)}  \cf p_{\ep}(B_{t-u-s}^{i}+x-\cdot) (\xi) ds  \\
&= &\frac 1\delta \int_0 ^{ \delta \wedge(t-u)}    \exp \lp
-\frac{\ep^{2}|\xi|^{2}}{2}+\imath \lla \xi ,B^i_{t-u-s}+x\rra \rp
ds.
\end{eqnarray*}
Thus
\begin{align}
&\langle A_{t,x}^{\varepsilon,\delta,B^i},A_{t,x}^{\varepsilon,\delta,B^j}\rangle_{\mathcal{H}} 
\label{m2} \\ 
&=\int_{\R^{d}} \left(   \iot \iot
   \left( \frac 1{\delta^2}  \int_0^{\delta\wedge v } \int_0^{\delta\wedge u} e^{\imath \lla \xi ,B^i_{u-s_1}- B^j_{v-s_2}\rra}     ds_1ds_2   \right)\gamma(u,v)    dudv  \right)    
   \times   e^{-\ep^2 |\xi|^2}   \mu(d\xi), \notag
\end{align}
and we divide the proof in several steps.

\smallskip
\noindent \textit{Step 1:}  We claim that,
\begin{equation}\label{eq12}
\lim_{\varepsilon \downarrow 0} \lim_{ \delta \downarrow 0} \langle A_{t,x}^{\varepsilon,\delta,B^i}, A_{t,x}^{\varepsilon,\delta,B^j}\rangle_{\mathcal{H}}=\iot \iot\gamma(u,v)f(B_{u}^i-B_{v}^j) du dv \,,
\end{equation}
where the convergence holds in   $L^1(\Omega)$. Notice first that
the right-hand side of equation (\ref{eq12}) is finite almost surely
because
\[
\be_B \lc \iot \iot \gamma(u,v)f(B_{u}^i-B_{v}^j) dudv\rc =\iot
\iot    \int_{\R^{ d}} \gamma(u,v) e^{-\frac {1}{2} (u+v) |\xi|^2}
\mu(d\xi) dudv
\]
and we show that this is finite  making the change of variables  $x=u-v$, $y=u+v$, and using the specific $\gamma$ and $\mu$ like in the proof of Theorem \ref{thmSk1}.

 In order to show the convergence (\ref{eq12}) we first let  $\delta$ tend to zero. Then, owing to the continuity of $B$ and applying some dominated convergence arguments to \eqref{m2}, we obtain
 the following limit almost surely and in $L^1(\Omega)$
 \begin{equation}\label{eq13}
 \lim_{ \delta \downarrow 0} \langle A_{t,x}^{\varepsilon,\delta,B^i}, A_{t,x}^{\varepsilon,\delta,B^j}\rangle_{\mathcal{H}}
 = 
 \int_{\R^d}  \lp   \iot \iot    e^{\imath \lla \xi ,B^i_{u}- B^j_{v}\rra}     \gamma(u,v) dudv\rp  e^{-\ep^2 |\xi|^2} \  \mu(d\xi) .
\end{equation}
Finally, it is easily checked that the right-hand side of (\ref{eq13}) converges in $L^1(\Omega)$ to the right-hand side of  (\ref{eq12}) as $\varepsilon$ tends to zero, by means of a simple dominated convergence argument again.

\smallskip
\noindent \textit{Step 2:}
For notational convenience, we denote by $B$ and $\widetilde{B}$ two independent $d$-dimensional Brownian motions, and  $\be$ will denote here the expectation with respect to both $B$ and $\widetilde{B}$.  We  claim that for any $\lambda > 0$
\begin{equation}\label{expint}
\sup_{\varepsilon,\delta}\be  \lc  
\exp \lp \lambda\lla A_{t,x}^{\varepsilon,\delta,B}, A_{t,x}^{\varepsilon,\delta,\widetilde{B}}\rra _{\mathcal{H}} \rp \rc< \infty\,.
\end{equation}
Indeed, starting from (\ref{m2}), making the change of variables $u-s
\rightarrow u$, $v-\tilde s \rightarrow v$,  assuming $\delta \le t$, and
using Fubini's theorem, we can write
\begin{align*}
  \lla A_{t,x}^{\varepsilon,\delta,B}, A_{t,x}^{\varepsilon,\delta,\widetilde{B}}\rra _{\mathcal{H}}
  =
\frac 1{\delta^2}  \int_0^\delta \int_0^\delta   \int_0^{t-s}\int_0^{t-\tilde s}\int_{\R^d}  \exp\lp-\imath (B_{u}-\widetilde{B}_{v})\cdot \xi \rp \\
\times \exp(-\varepsilon |\xi|^2)  \gamma(u+s,v+\tilde s) \, \mu(d\xi) \, dudvdsd\tilde s\,.
\end{align*}
We now control the moments of $\langle A_{t,x}^{\varepsilon,\delta,B}, A_{t,x}^{\varepsilon,\delta,\tilde{B}}\rangle _{\mathcal{H}}$ in order to reach exponential integrability:
\begin{align} \label{m7}
\no \lla A_{t,x}^{\varepsilon,\delta,B}, A_{t,x}^{\varepsilon,\delta,\widetilde{B}}\rra _{\mathcal{H}}^n
=\frac 1{\delta^{2n}}  \int_{O_{\delta,n}} \int_{\R^{dn}}
\exp\left(-\imath\sum_{l=1}^n (B_{u_l}-\widetilde{B}_{v_l})\cdot\xi_l\right)\\ 
\times e^{-\varepsilon \sum_{l=1}^n|\xi_l|^2} \prod_{l=1}^n
\gamma(u_l+s_l,v_l+\widetilde{s}_l) \, \mu(d\xi) \, dsd\tilde{s}dudv,
\end{align}
where $\mu(d\xi)=\prod_{l=1}^n \mu(d\xi_l)$, the differentials 
$ds,d\tilde{s},du,dv$ are defined similarly, and
\[
O_{\delta,n}=\left\{ (s, \widetilde{s}, u,v);\, 0\le s_l, \widetilde{s}_l
\le \delta, \, 0\le u_l \le t-s_l, \, 0\le v_l \le
t-\widetilde{s}_l , \text{ for all } 1\leq l \leq n\right\}.
\]. We will prove that
\begin{align}
\int_0^t\int_0^t dudv \gamma(u,v)<\infty,
\end{align}

i.e., the case when $n=1$ in the Appendix~\ref{gamma_exponential_finite }.
If this hold true, we have:
\begin{eqnarray}\label{eq:charac1}
\no \be\lc  \exp\left(-\imath\sum_{l=1}^n (B_{u_l}-\widetilde{B}_{v_l})\cdot\xi_l\right)\rc
&=&
\exp\lp -\frac12 \var\lp  \sum_{l=1}^n (B_{u_l}-\widetilde{B}_{v_l})\cdot\xi_l\rp \rp 
 \\
&=&
\exp\left(-\frac{1}{2}\sum_{1\leq i,j\leq n}
(u_i\wedge u_j+v_i\wedge v_j)
\xi_i\cdot \xi_j\right). \notag
\end{eqnarray}
This yields
\begin{eqnarray*}
\lefteqn{\be \lc \left\langle A_{t,x}^{\varepsilon,\delta,B},
A_{t,x}^{\varepsilon,\delta, \widetilde {B}}\right\rangle _{\mathcal{H}}^n \rc}\\
&\leq& C(t)^n
\int_{[0,t]^{2n}}\int_{\R^{dn}}\exp\left(-\frac{1}{2}\sum_{1\leq
i,j\leq n}(s_i\wedge s_j+\widetilde{s}_i\wedge \widetilde{s}_j)\xi_i\cdot
\xi_j\right) \mu(d\xi)dsd\tilde{s}\\
&\leq&C(t)^n \int_{\R^{dn}}\int_{[0,t]^n}\exp\left(-\sum_{1\leq i,j\leq n}(s_i\wedge s_j) \, \xi_i \cdot \xi_j\right)ds \mu(d\xi)\,.
\end{eqnarray*}
 Since
\begin{equation*}
\int_{\R^{dn}}\exp\left(-\sum_{1\leq i,j\leq n}(s_i\wedge s_j)\xi_i \cdot \xi_j\right)\mu(d\xi)
\end{equation*}
is a symmetric function of $s_1,s_2,\dots,s_n$, we can restrict our
integral to $T_n(t) =\{0<s_1< s_2< \cdots < s_n< t\}$. Hence, using the convention $s_0=0$,  we have
\begin{align}\label{eq:mAn}
\lefteqn{\be  \lc \left \langle A_{t,x}^{\varepsilon,\delta,B}, A_{t,x}^{\varepsilon,\delta, \widetilde {B}}\right \rangle _{\mathcal{H}}^n\rc}\\
\no &\leq  C(t)^n n!\int_{\R^{dn}}\int_{T_n(t)}\exp\left(-\sum_{1\leq i,j\leq n}(s_i\wedge s_j)\xi_i \cdot \xi_j\right)ds \mu(d\xi)\\
&=C(t)^n n! \int_{\R^{dn}}\int_{T_n(t)}\exp\left(-\sum_{i=1}^n
(s_i-s_{i-1})|\xi_i+ \cdots +\xi_n|^2\right)ds\mu(d\xi). \notag
\end{align}
Thus, using the same argument as in the proof of  the estimate (\ref{eq4}),  we end up with
\begin{eqnarray*}
\be  \lc \left \langle A_{t,x}^{\varepsilon,\delta,B}, A_{t,x}^{\varepsilon,\delta, \widetilde {B}}\right\rangle _{\mathcal{H}}^n\rc
&\leq&C(t)^n n! \int_{T_n(t)}\prod_{i=1}^n \left(\sup_{\eta \in \R^d}
\int_{\R^d}e^{-(s_i-s_{i-1})|\xi_i+\eta|^2}\mu(d\xi)\right)ds\\
&\le &C(t)^n n!
\int_{T_n(t)}\prod_{i=1}^n\left(\int_{\R^d}e^{-(s_i-s_{i-1})|\xi_i|^2}\mu(d\xi_i)\right)ds\,.
\end{eqnarray*}
Making  the change of variable $w_i=s_i-s_{i-1}$, the above integral
is equal to
\begin{equation*}
C(t)^n n! \int_{S_{t,n}} \int_{\R^{dn}}\prod_{i=1}^n e^{-w_i
|\xi_i|^2}\mu(d\xi)dw 
\le
C(t)^n n! \sum_{k=0}^n \binom{n}{k} \frac{t^k}{k!}D_N^k
(2C_N)^{n-k} ,
\end{equation*}
where we have resorted to Lemma \ref{lem1} for the last inequality.
Therefore,
\begin{eqnarray*}
\frac{1}{n!} \be \lc \left\langle A_{t,x}^{\varepsilon,\delta,B},
A_{t,x}^{\varepsilon,\delta, \widetilde {B}}\right\rangle _{\mathcal{H}}^n  \rc
\leq C(t)^n  \sum_{k=0}^n \binom{n}{k} \frac{t^k}{k!}D_N^k
(2C_N)^{n-k}\,,
\end{eqnarray*}
which is exactly the right hand side of \eqref{eq:eqDD}. Thus, along the same lines as in the proof of Theorem \ref{thmSk1}, we get
\begin{equation*}
\be \lc  \exp \lp \lambda
\left\langle A_{t,x}^{\varepsilon,\delta,B}, A_{t,x}^{\varepsilon,\delta, \widetilde {B}}\right\rangle _{\mathcal{H}} \rp\rc
=
\sum_{n=0}^{\infty}\frac{\lambda^n}{n!}\be \lc \left\langle A_{t,x}^{\varepsilon,\delta,B}, A_{t,x}^{\varepsilon,\delta, \widetilde {B}}\right\rangle _{\mathcal{H}}^n\rc\,<\infty,
\end{equation*}
which completes the proof of \eref{expint}.

\medskip
\noindent \textit{Step 3:} Starting from \eqref{eq:exp-moments-utx-ep-delta}, (\ref{eq12}) and (\ref{expint})  we
deduce that  $\be [ ( u_{t,x}^{\varepsilon,\delta})^k]$
converges as $\delta$ and $\varepsilon$ tend to zero to the
right-hand side of   \eref{momSk1}. On the other hand, we can also
write
\[
\be \lc  u_{t,x}^{\varepsilon,\delta} u_{t,x}^{\varepsilon',\delta'}  \rc=
\be_B \lc  \exp \lp\ \langle A^{\varepsilon,\delta,
B^1}_{t,x} , A^{\varepsilon',\delta',
B^2}_{t,x}  \rangle_{\mathcal{H}}\rp\rc\,.
\]
As before we can show that this converges as $\varepsilon,\delta,
\varepsilon', \delta'$ tend to zero. So,
$u_{t,x}^{\varepsilon,\delta}$ converges in $L^2$ to some limit
$v_{t,x}$, and the limit is actually in  $L^p$ , for all $p \geq 1$.
Moreover, $\be [v^k_{t,x}]$ equals to the right hand side of
\eref{momSk1}. Finally, letting $ \delta$ and $\varepsilon$ tend to
zero in equation \eref{eq8} we get
\begin{equation*}
\be [Fv_{t,x}]= \be[ F]  +\be \lc \langle DF, v p_{t-\cdot}(x-\cdot)\rangle_{\mathcal{H}}\rc,
\end{equation*}
which implies that the process $v$ is the solution of equation
\eref{eqn:DO process}, and by the uniqueness of the solution we have $v=u$.
\end{proof}

\chapter{Moment bounds}
In this Chapter, we will establish upper and lower bounds on the moments of the solution in the case of the generalized noise of Chapter 5. This will generalize the results of Chapter 4, and show that the asymptotic behavior of the moments only depends on the value of $2a_1 + a_2$, which corresponds to the order of the variance in time, and not on the specific values of $a_1$ and $a_2$. We start by establishing the upper bound result.

Here the three cases and the parameter $a$ are defined similarly to the settings in Chapter 2.

\section{Upper Bound}
\begin{theorem}
\label{thm:2nd moment upper bound}
Under the Assumptions \ref{assumption:finite var} and \ref{assumption:bounded} and with $H$ given by equation~\eqref{a1a2condition}, let $u$ be the solution to the equation~\eqref{eqn:DO process}. There is a positive constant $C<\infty$ such that for all $t\in \mathbb{R}_+$ and $x,y\in \mathbb{R}^d$,
\begin{align}
    E[u(t,x)u(t,y)]\leq C  u_0^2 \cdot \exp(C t^{\frac{4H-a}{2-a}}).
\end{align}
\end{theorem}
\begin{proof}
We will use some notation from Balan-Conus\cite{Balan2016}. More specifically, 
\begin{align}
\no\lefteqn{E[u(t,x_1)u(t,x_2)]}\\
\no&=E\Bigg[\Bigl(w_0(t,x)+\sum_{n=1}^\infty I_n(f_n(t,x))\Bigl)\Bigl(w_0(t,x)+\sum_{n=1}^\infty I_n(f_n(t,x))\Bigl)\Bigg]\\
\no&=E[1+\sum_{n=1}^\infty I_n^2(f_n)]\\
\no&=\sum_{n=1}^\infty \frac{1}{n!}\alpha_H^n \int_{[0,t]^{2n}}\prod_{j=1}^n (t_j)^{a_1}(s_j)^{a_1} |t_j-s_j|^{a_2} \psi_n(\mathbf{t},\mathbf{s})\,d\mathbf{t}d\mathbf{s},
\end{align}
where $\mathbf{t}=(t_1,t_2,...,t_n)$ and $d\mathbf{t}=dt_1dt_2...dt_n$ (same hold for $\mathbf{s}$), $\alpha_H =H(2H-1)$ is a constant introduced in \eqref{fractionalconstant} and also in Balan-Conus\cite[Section 1]{Balan2016}. The second equality comes from the orthogonality of the Wiener chaos expansion. The function $\psi_n(\textbf{t},\textbf{s})$ is defined as follows. We use a similar arument in \cite[Section 3]{Balan2016}, see also \eqref{e:fn'} for the notation.
\begin{align}
    \psi_n(\textbf{t},\textbf{s})&:=\int_{\mathbb{R}^{nd}}\mathcal{F}f_n(s,\cdot,t,x)(\xi)\\
 \no&\hspace{3mm}\times \overline{\mathcal{F}f_n(s,\cdot,t,x)(\xi)}\mu(d\xi_1)\cdots \mu(d\xi_n)\\
 \no   &=\int_{\mathbb{R}^{2nd}}\prod_{j=1}^n p_h(t_{\rho(j+1)}-t_{\rho(j)},x_{\rho(j+1)}-x_{\rho(j)})\\
 \no       &\hspace{3mm}\times \prod_{j=1}^n p_h(s_{\sigma(j+1)}-s_{\sigma(j)},y_{\sigma(j+1)}-y_{\sigma(j)})\\
\no &\hspace{7mm}\prod_{j=1}^n f(x_j-y_j)dxdy,
\end{align}
where $\rho$ is the permutation such that $0<t_{\rho(1)}<...<t_{\rho(n)}<t$ and $\sigma$ is the permutation such that $0<s_{\sigma(1)}<\dots<s_{\sigma(n)}<t$ (Note that $\psi_n$ can be written as a sum over all permutations with the appropriate indicator function as in \eqref{e:fn'}).

The next lemma established a bound for the space term, i.e., $\psi_n(\textbf{t},\textbf{s})$.
\begin{lemma}[Balan-Conus 2016]
We have
\begin{align}
\no    \psi_n(\textbf{t},\textbf{s})\leq u_0^2 K_w^n [\beta(\textbf{t})\beta(\textbf{s})]^{-\frac{a}{4}},
\end{align}
where $\beta(\textbf{t})=\prod_{j=1}^n (t_{\rho(j+1)}-t_{\rho(j)})$, and $\rho \in S_n$ is chosen such that $t_{\rho(1)}<...<t_{\rho(n)}$, and $t_{\rho(n+1)}=t.$ 
\end{lemma}
\begin{proof}
The proof is in \cite[Proposition 4.2]{Balan2016}
\end{proof}
We go back to our inner product term. With Plancherel Theorem, we eliminated the space term, and take the Wiener chaos expansion to the solution $u$. For the simplicity from now on, let's define 
\begin{align}
\no   \alpha_n(t):=\alpha_H^n \int_{[0,t]^{2n}}\prod_{j=1}^n (t_j)^{a_1}(s_j)^{a_1} |t_j-s_j|^{a_2} \psi_n(\mathbf{t},\mathbf{s})\,d\mathbf{t}d\mathbf{s}.
\end{align}
Then we plug in the bound, we have:
\begin{align}
\alpha_n(t) \leq \alpha_H^n u_0^2 \int_{[0,t]^n} \int_{[0,t]^n} K_w^n\beta(\textbf{t})^{-\frac{a}{4}}\textbf{t}^{a_1}\beta(\textbf{s})^{-\frac{a}{4}}\textbf{s}^{a_1} \prod_{i=1}^n|t_i-s_i|^{a_2}d\textbf{t}d\textbf{s},
\end{align}
where $\mathbf{s}^{a_1}=s_1^{a_1}s_2^{a_1}...s_n^{a_1}$ and $K$ is a constant depending on a (e.g., a=1, $K=\pi$).\\
Consider $\beta(\textbf{t})^{-\frac{a}{4}}\textbf{t}^{a_1}$ as a new function, using Lemma B.3 in Balan-Conus \cite{Balan2016}, we have
\begin{align}
\label{n factorial}
 \no  \lefteqn{\alpha_n(s)}\\
 \no&\leq C_H^n\int_{[0,t]^n}\bigg((\beta(\textbf{s})^{-\frac{a}{4}}\textbf{s}^{a_1})^{\frac{2}{a_2+2}}d\textbf{s}\bigg)^{a_2+2}\\
\no &\leq C_H^n \bigg(\int_{[0,t]^n}[(t-s_n)(s_n-s_{n-1})\cdots (s_2-s_1)]^{\frac{2}{a_2+2}\cdot (-\frac{a}{4})}\\
\no &\hspace{4mm}\times \textbf{s}^{\frac{2a_1}{a_2+2}}d\textbf{s}\bigg)^{a_2+2}\\
\no&\leq C_H^n \bigg(n!\int_{\mathbb{T}^n}[(t-s_n)(s_n-s_{n-1})\cdots (s_2-s_1)]^{\frac{2}{a_2+2}\cdot (-\frac{a}{4})}\\ &\hspace{4mm}\times \textbf{s}^{\frac{2a_1}{a_2+2}}d\textbf{s}\bigg)^{a_2+2}.
\end{align}
Next we can use a similar trick as in the DO noise case to find the order of $t$ and the coefficient depending on $n$.\\
Following Balan-Tudor\cite[Lemma 3.5]{BalanTudor} we have
\begin{align}
\no  &  \lefteqn{\alpha_n(t)}\\
\no&\leq t^{n(2H-a/2)}(n!)^{a_2+2}\prod_{i=0}^{n-1}\beta\Bigg(\frac{-a/2}{a_2+2}+1, i\cdot (\frac{-a/2}{a_2+2}+\frac{2a_1}{a_2+2}+1)+\frac{2a_1}{a_2+2}+1\Bigg)^{a_2+2}.
\end{align}
Expanding the beta function with Gamma function, the product becomes
\begin{align}
\label{gamma product}
  \prod_{i=0}^{n-1} (\frac{\Gamma(\frac{-a/2}{a_2+2}+1)\Gamma(i(\frac{-a/2}{a_2+2}+\frac{2a_1}{a_2+2}+1)+\frac{2a_1}{a_2+2}+1)}{\Gamma(1+(i+1)\frac{2a_1+a_2+2-a/2}{a_2+2})})^{a_2+2}.
\end{align}
Here we can use Stirling's formula, for all $i$ we have
\begin{align}
\no    c_1 (\frac{(i+1)(2H-a/2)}{a_2+2})^{-\frac{-\frac{a}{2}+a_2+2}{a_2+2}}&\leq\frac{\Gamma(i(\frac{-a/2}{a_2+2}+\frac{2a_1}{a_2+2}+1)+\frac{2a_1}{a_2+2}+1)}{\Gamma(1+(i+1)\frac{2a_1+a_2+2-a/2}{a_2+2})}\\
\no&\leq c_2 (\frac{(i+1)(2H-a/2)}{a_2+2})^{-\frac{-\frac{a}{2}+a_2+2}{a_2+2}},
\end{align}
for some constants $c_1,c_2$.\\
The product \eqref{gamma product}becomes 
\begin{align}
 \no   \Gamma \Bigl(\frac{-a/2}{a_2+2}+1\Bigl)^{n(a_2+2)} \Bigl(n!\frac{2H-a/2}{a_2+2}\Bigl)^{-(-a/2+a_2+2)}.
\end{align}
So we have the following bound for $\alpha_n(t)$:
\begin{align}
\label{wienerchaosalpha}
    \alpha_n(t)\leq u_0^2\frac{c K_h^n t^{n(2H-a/2)}}{(n!)^{-a/2}}.
\end{align}
Finally, we sum up all the $\alpha_n(t)$:
\begin{align}
 \no    E|u(t,x)|^2 &=w_0(t,x)+\sum_{n=1}^\infty \frac{1}{n!}\alpha_n(t) \\
 \no&\leq \sum_{n\geq 0} \frac{c K_h^n t^{n(2H-a/2)}}{(n!)^{-a/2+1}}\\
 &\sim C  u_0^2 \cdot \exp(C t^{\frac{4H-a}{2-a}}),
\end{align}
by Lemma~\ref{lemmaA1}.
\end{proof}

Next we prove the nth moment result, following Balan-Conus~\cite{Balan2016} part 5 (Hyperbolic case: Upper bound on the moments).
\begin{theorem}[$n$-th moment upper bound]
Let $u$ be the solution to the equation~\eqref{eqn:DO process} for any of the cases (1),(2),(3) in Chapter 2. Then there exists a constant $C$ such that
\begin{align}
\label{generalnthupperbound}
    E[u^n(t,x)]\leq C u_0^n \cdot \exp(C  n^{\frac{4-a}{2-a}} t^{\frac{4H-a}{2-a}}).
\end{align}
\end{theorem}

\begin{proof}
Using the chaos expansion and the hypercontractivity property we can derive the upper bound from the second moment as it was done in Balan-Conus\cite{Balan2016}.

We denote by $\norm{\cdot}_p$ the $L^p(\Omega)$-norm. We use the fact that for elements in a fixed Wiener chaos $\mathcal{H}_n$, the $\norm{\cdot}_p$-norms are equivalent. Namely, for any $I_n(f_n)\in \mathcal{H}_n$, \\
$\norm{I_n(f_n(t,x))}_p\leq (p-1)^{n/2} \norm{I_n(f_n(t,x))}_2=(p-1)^{n/2} \Bigg{(}\frac{1}{n!}\alpha_n(t)\Bigg{)}^{1/2}$.

In our case, we obtain the upper bound for $\norm{I_n(f_n(t,x))}_p$
\begin{align}
 \norm{I_n(f_n(t,x))}_p \leq u_0 C^n_{p,K_h}\Bigg(t^{n(2H-a/2)}\frac{1}{n!^{-a/2+1}}\Bigg)^{1/2},
\end{align}
where $C_{p,K_h}=(p-1)^{1/2} c^{1/2}K_h^{1/2}$ and $c$ depends on $H$ and $a$.\\
Recall Minkowski's inequality for integrals 
\begin{align}
\Bigg{[}\int_Y\Bigg{(}\int_X|F(x,y)|\mu(dx)\Bigg{)}^p\nu(dy)\Bigg{]}^{1/p}\leq \int_X\Bigg{(}\int_Y|F(x,y)|^p\nu(dy)\Bigg{)}^{1/p}\mu(dx).
\end{align}
We use this inequality for $(X,\mathcal{X})=(\mathcal{N},2^{\mathcal{N}},\mu)$ equipped with the counting measure, $(Y,\mathcal{Y},\nu)=(\Omega,\mathcal{F},P)$ and $F(n,\omega)=I_n(f_n(\omega,t,x)).$ We have
\begin{align}
\no    \norm{u(t,x)}_p&=\norm{\sum_{n\geq 0}I_n(f_n(t,x))}_p\leq \sum_{n\geq 0} \norm{I_n(f_n(t,x))}_p\\
\no    &\leq u_0 \sum_{n\geq 0} \frac{C_p^n K_h^n t^{n(H-a/4)}}{(n!)^{-a/4+1/2}}.
\end{align}
Using Lemma A.1(Appendix A), we conclude that for any $t>0$,
\begin{align}
\label{normupperbound}
    \norm{u(t,x)}_p\leq c_1 u_0 \exp\{ C_2 (p^{1/2} K^n_h t^{(2H-a/2})^{\frac{1}{1-a/2}}\},
\end{align}
where $C_2>0$ and $C_n>0$ are some constants depending on $a$.\\
Finally we take the pth power of \eqref{normupperbound}, which will prove our Theorem.
\end{proof}
\begin{remark}
Notice that given \eqref{a1a2condition}, \eqref{generalnthupperbound} can be written as 
\begin{align}
\no E[u^n(t,x)]\leq \exp(C n^{\frac{4-a}{2-a}} t^{\frac{4a_1+2a_2+4-a}{2-a}}).
\end{align}
Given that the variance of the noise is of asymptotic order $t^{2a_1+a_2}$, this result shows that the specific correlation structure is not relevant for the Lyapunov exponent. The same will hold for the lower bound below.
\end{remark}
\section{Lower Bound}
For the lower bound, the main idea comes from Hu-Huang-Nualart-Tindel\cite{HuHuangNualartTindel} and Balan-Conus\cite{Balan2016}.
\begin{theorem}[Lower Bound for SHE]
Let $u$ be the solution to the equation~\eqref{eqn:DO process} for any of the cases (1),(2),(3) in Chapter 2. Then there exists a constant $C$ such that
\begin{align}
\label{lowerboundheat}
    E[u(t,x)^k]\geq \exp(C k^{\frac{4-a}{2-a}} t^{\frac{4H-a}{2-a}}).
\end{align}

\end{theorem}
\begin{proof}
We consider the set
\begin{align}
\no   A_{\epsilon}=\{\sup_{1\leq i<j\leq k}\sup_{1\leq \ell\leq d}\sup_{0\leq s, r\leq 1} |B_s^{i,\ell}-B_r^{j,\ell}|\leq \epsilon\}.
\end{align}
Then owning to the formula \eqref{SHEkthmoment}, we have
\begin{align}
\no    E[u(t,x)]^k&= E\Bigg{[}\exp \Bigg{(}c\sum_{1\leq i<j\leq k} \int_0^t\int_0^t s^{a_1}r^{a_1}|s-r|^{a_2}|B_s^i-B_r^j|^{-a}dsdr\Bigg{)}\Bigg]\\
\no   &= E\Bigg{[}\exp \Bigg{(}c\sum_{1\leq i<j\leq k} t^{2a_1+a_2+2-a/2} \int_0^1\int_0^1 \tilde s^{a_1}\tilde r^{a_1}|\tilde s-\tilde r|^{a_2}|B_{\tilde s}^i-B_{\tilde r}^j|^{-a}d\tilde sd\tilde r\Bigg{)}\Bigg]\\
\no    &\geq  E\Bigg{[}   \exp \Bigg{(}t^{2a_1+a_2+2-a/2} c\sum_{1\leq i<j\leq k} \int_0^1\int_0^1 \tilde s^{a_1}\tilde r^{a_1}|\tilde s-\tilde r|^{a_2}|\epsilon|^{-a}d\tilde sd\tilde r\Bigg{)}\mathbbm{1}_{A_{\epsilon}}\Bigg{]}\\
\no    &=  \exp\Bigg{(}t^{2H-a/2} B(a_1+1,a_2+1)k(k-1) \frac{1}{H} \epsilon^{-a}\Bigg{)}\textbf{P}(A_\epsilon).
\end{align}
Here we have $\tilde s=\frac{s}{t}$, $\tilde r=\frac{r}{t}$. The last inequality comes from the beta function and the substitution $s=\beta r$, where $0<\beta<1$, and $B(x,y)$ is the Beta function $\frac{\Gamma(a_1+1)\Gamma(a_2+1)}{\Gamma(a_1+a_2+2)}$.
Then since the event $A_\epsilon$ is the same as \cite{HuHuangNualartTindel}, there exists an $\epsilon_0>0$ such that for $\epsilon\leq \epsilon_0$, we have
\begin{align}
    P(A_\epsilon)\geq \exp(-Ckd \epsilon^{-2}),
\end{align}
where $d$ is the space dimension.

Under the condition $\epsilon\leq \epsilon_0$, this entails:\\
\begin{align}
    E[u(t,x)]^k \geq \exp(c t^{2H-a/2}k^2 \epsilon^{-a}-Cdk \epsilon^{-2}).
\end{align}
In order to optimize this expression, we try to equate the two terms inside the exponential above, we get
\begin{align}
    \epsilon=\frac{t^{\frac{2H-\frac{a}{2}}{a-2}}(ck)^{\frac{1}{a-2}}}{(2Cd)^{\frac{1}{a-2}}},
\end{align}
and notice that for $k\geq 2$ and $t$ sufficiently large, the condition $\epsilon\geq \epsilon_0$ is fulfilled. Therefore, we conclude that for $t$ and $k$ large enough,
\begin{align}
    E[u(t,x)^k]\geq \exp\Bigg{(}\frac{c^{\frac{a}{2-a}} t^{\frac{4H-a}{2-a}}k^{\frac{4-a}{2-a}}}{8(2dC)^{\frac{a}{2-a}}}\Bigg{)},
\end{align}
which finishes the proof of \eqref{lowerboundheat}.
\end{proof}
\begin{remark}
For the wave equation, we conjecture that a similar result would hold, where the Lyapunov exponent only depends on $2a_1+a_2$ and not on their specific values. This is subject of ongoing research.
\end{remark}

\renewcommand{\thechapter}{}
\renewcommand{\thesection}{A}
\renewcommand{\chaptername}{Appendix}
\setcounter{theorem}{0}
\setcounter{chapter}{0}
\setcounter{section}{0}
\setcounter{equation}{0}
\addcontentsline{toc}{chapter}{Appendix}

\chapter*{Appendix}

\makeatletter
\renewcommand \theequation {%
A.%
\ifnum \c@chapter>\z@ \@arabic\c@chapter.%
\fi\@arabic\c@equation} \@addtoreset{equation}{section}
\@addtoreset{equation}{section} 
\makeatother

\begin{lemma}[Conus-Balan Lemma A.1]
\label{lemmaA1}
For any $a>0$, we have
\begin{align}
\no    \sum_{n\geq 0} \frac{x^n}{(n!)^a}\leq c_1 \exp (c_2 x^{1/a})\hspace{7mm} \text{for all }x>0,
\end{align}
where $c_1>0$ and $c_2>0$ are some constants depending on $a$.
\end{lemma}

\begin{lemma}
\label{gammaintegrable}
For the covariance function $\gamma(s,r)=s^{a_1}r^{a_1}|s-r|^{a_2}$, we have
\begin{align}
\sup_{0<s<t}\int_0^t\gamma(s,r)\,dr <\infty.
\end{align}
\end{lemma}
\begin{proof}
 Define $\rho=\frac{r}{t}$.\\
\begin{eqnarray*}
\no &&\int_0^t s^{a_1}r^{a_1}|s-r|^{a_2}\,dr\\
\no&=&t^{1+a_1}\int_0^1 s^{a_1}\rho^{a_1}|s-t\rho|^{a_2}\,d\rho\\
\no &=&t^{a_1+a_2+1}\int_0^1 s^{a_1}\rho^{a_1}|\frac{s}{t}-\rho|^{a_2}\,d\rho \\
\no &<& t^{2a_1+a_2+1}\int_0^1 \rho^{a_1}|\frac{s}{t}-\rho|^{a_2}\,d\rho 
\end{eqnarray*}
Let $\sigma=\frac{s}{t}, 0<\sigma<1$. Since $\rho<1$  
\begin{eqnarray*}
\no &&\int_0^1 \rho^{a_1}|\sigma-\rho|^{a_2}d\rho\\
\no &\leq&\int_0^\sigma (\sigma-\rho)^{a_2}d\rho+\int_\sigma^1 (\rho-\sigma)^{a_2}\,d\rho\\
&=&-\frac{(\sigma-\rho)^{a_2+1}}{a_2+1}\Bigg|_0^\sigma+ \frac{(\rho-\sigma)^{a_2+1}}{a_2+1}\Bigg|_\sigma^1\\
\no &=&\frac{\sigma^{a_2+1}+(1-\sigma)^{a_2+1}}{a_2+1}\leq \frac{2}{a_2+1}<\infty
\end{eqnarray*}
since $-1<a_2<0$.
\end{proof}

\begin{lemma}
\label{gamma_exponential_finite }
For the covariance function $\gamma(s,r)=s^{a_1}r^{a_1}|s-r|^{a_2}$, we have
\begin{align}
    \int_0^t\int_0^t dudv\gamma(u,v)<\infty \hspace{8mm} \text{for all } t>0.
\end{align}
\end{lemma}
\begin{proof}
From direct calculation we have
\begin{align}
\label{innerprod}
\no&\int_0^t\int_0^t \gamma(u,v)\,dudv\\
\no=&\int_0^t\int_0^t u^{a_1}v^{a_1}|u-v|^{a_2}\,dudv\\
\no=&2\int_0^t  du\int_0^u u^{a_1} v^{a_1}|u-v|^{a_2}\,dudv\\
=&2\int_0^t u^{2a_1} du\int_0^u \frac{v^{a_1}}{u^{a_1}}|u-\frac{v}{u}\cdot u |^{a_2}\,dudv.
\end{align}
Let $r=\frac{v}{u}$,$v=ru$, \eqref{innerprod} becomes
\begin{align}
\label{generalnoisecalc}
    &2\int_0^t du\, u^{2a_1+a_2+1} \int_0^1 \,dr\, r^{a_1}|1-r|^{a_2}<C(t).
\end{align}
Since $a_2>-1, a_1>0,$ $\int_0^1 \,dr \,r^{a_1}|1-r|^{a_2}<\infty$,  $2a_1+a_2+1=2H-1>-1$, we require $H>0$,which is true since $H>1/2$.
\end{proof}

\bibliographystyle{plain}
\bibliography{mybiblograph}
\addcontentsline{toc}{chapter}{Vita}
\chapter*{Vitae}
Ruxiao Qian was born in Changsha, China. He attended elementary school in Xichangjie, Kaifu District and graduated from Changjun High School in June 2011. He attended Tongji University in Shanghai and graduated in 2015. Ruxiao entered Lehigh University as a master student in the Analytical Finance program in 2016. He then transferred to the Mathematics Department in 2017 and worked as a Teaching Assistant. He began his dissertation work under the guidance of Professor Daniel Conus and was awarded the degree of Doctor of Philosophy in Mathematics in August 2022.

\end{document}